\documentclass[11pt]{gsm-l}

\usepackage{extsizes}

\usepackage{tikz}
\usepackage{rotating}

\usepackage{enumitem}
\setlist[itemize]{topsep=0pt,itemsep=0pt}
\setlist[enumerate]{topsep=0pt,itemsep=0pt}

\usepackage{scalerel}
\usepackage{tkz-euclide}
\usetikzlibrary{arrows, arrows.meta, bending, decorations.pathmorphing, decorations.markings} 

\usepackage{auto-pst-pdf}

\usepackage{pict2e, amssymb,latexsym,amstext,amsmath,amsthm,epsfig,ifthen}
\usepackage{color}
\usepackage{blkarray}
\usepackage[capitalize, nameinlink, noabbrev, nosort]{cleveref}
\usepackage{youngtab}
\usepackage[boxsize=.35 em]{ytableau}

\usepackage[dvips,centering,includehead,width=12.61cm,height=22.1cm,paperheight=23.8cm,paperwidth=17cm]{geometry}

\input epsf

\makeatletter
\newcommand{\manuallabel}[2]{\def\@currentlabel{#2}\label{#1}}
\makeatother

\manuallabel{ch:tp-examples}{1}
\manuallabel{sec:Ptolemy}{1.2}
\manuallabel{eq:grassmann-plucker-3term}{1.2.1}
\manuallabel{fig:A3assoc_poly}{1.4}
\manuallabel{exercise:solid-minors}{1.4.2}
\manuallabel{exercise:KL}{1.4.4}
\manuallabel{eq:K(z)}{1.4.2}
\manuallabel{eq:L(z)}{1.4.3}
\manuallabel{fig:schemes}{1.12}

\manuallabel{fig:quiver-mutation}{2.1}
\manuallabel{ex:mut-simple}{2.1.4}
\manuallabel{fig:quiver-triangulation}{2.2}
\manuallabel{sec:triangulations}{2.2}
\manuallabel{def:Q(T)-polygon}{2.2.1}
\manuallabel{ex:flip-is-mutation}{2.2.2}
\manuallabel{fig:two-quivers}{2.5}
\manuallabel{fig:chamber-quiver2}{2.6}
\manuallabel{ex:orientations-of-a-tree}{2.6.5}
\manuallabel{exercise:D4-E8}{2.6.8}
\manuallabel{def:finitemutationtype}{2.6.11}
\manuallabel{fig:E6E8}{2.12}
\manuallabel{def:Bmatrix}{2.7.1}
\manuallabel{eq:matrix-mutation}{2.7.1}
\manuallabel{ex:mat-mut-simple}{2.7.7}
\manuallabel{def:diagramofB}{2.7.10}
\manuallabel{pr:diagram-mutation}{2.7.11}
\manuallabel{lem:vector}{2.7.13}

\manuallabel{ex:mut-commute}{3}
\manuallabel{eq:Abc}{3.2.2}
\manuallabel{example:A(1,2)}{3.2.7}
\manuallabel{cor:y-hat}{3.5.5}
\manuallabel{lem:distributive}{3.6.3}
\manuallabel{def:triple-pattern}{3.6.10}
\manuallabel{eq:yj}{3.6.2}
\manuallabel{eq:y-mutation-trop}{3.6.3}
\manuallabel{eq:exchange-rel-xx}{3.6.4}

\manuallabel{ch:rings}{6}  
\manuallabel{sec:models-classical}{6.3}
\manuallabel{ex:realization-Dn}{6.3.5}
\manuallabel{ex:SL5}{6.5.6}

\manuallabel{ch:Grassmannians}{8}

\manuallabel{sec:exchangegraph}{9.1}
\manuallabel{sec:generalized-associahedra}{9.3}

\manuallabel{sec:finite-mutation-type-skew-symmetrizable}{10.2}
\manuallabel{thm:Dynkin-hereditary}{10.4.1}

\newtheorem{theorem}{Theorem}[section]
\newtheorem{proposition}[theorem]{Proposition}

\newtheorem{corollary}[theorem]{Corollary}
\newtheorem{lemma}[theorem]{Lemma}

\theoremstyle{definition}

\newtheorem{remark}[theorem]{Remark}
\newtheorem{example}[theorem]{Example}
\newtheorem{definition}[theorem]{Definition}
\newtheorem{problem}[theorem]{Problem}

\newtheorem{exercise}[theorem]{Exercise}

\numberwithin{section}{chapter}
\numberwithin{equation}{section}
\numberwithin{figure}{chapter}

\def\endproof{\hfill$\square$\medskip}

\def\ZZ{\mathbb{Z}}

\def\CC{\mathbb{C}}

\def\RR{\mathbb{R}}
\def\QQ{\mathbb{Q}}
\def\TT{\mathbb{T}}

\def\AA{\mathcal{A}}
\def\FFcal{\mathcal{F}}

\def\xx{\mathbf{x}}
\def\yy{\mathbf{y}}

\def\roots{\operatorname{roots}}

\def\Gr{\operatorname{Gr}}

\def\SL{\operatorname{SL}}
\def\GL{\operatorname{GL}}

\def\Qsf{\QQ_{\,\rm sf}}

\newcommand{\overunder}[2]{
\!\begin{array}{c}
\scriptstyle{#1}\\[-.1in]
-\!\!\!-\!\!\!-\\[-.1in]
\scriptstyle{#2}
\end{array}
\!
}

\def\Trop{\operatorname{Trop}}

\newcommand{\notch}{\scriptstyle\bowtie}

\newsavebox{\digon}
\savebox{\digon}(10,10)[bl]{
\qbezier(5,0)(0,5)(5,10)
\qbezier(5,0)(10,5)(5,10)
\put(5,0){\circle*{1}} 
\put(5,5){\circle*{1}} 
\put(5,10){\circle*{1}} 
}

\newsavebox{\lowbar}
\savebox{\lowbar}(10,10)[bl]{
\put(5,0){\line(0,1){5}}
}

\newsavebox{\lowtag}
\savebox{\lowtag}(10,10)[bl]{
\put(5,3.5){\makebox(0,0){$\notch$}}
}

\newsavebox{\highbar}
\savebox{\highbar}(10,10)[bl]{
\put(5,5){\line(0,1){5}}
}

\newsavebox{\hightag}
\savebox{\hightag}(10,10)[bl]{
\put(5,6.5){\makebox(0,0){$\notch$}}
}

\newcommand{\credit}[1]{\smallskip\noindent {\textbf{#1.\ }}} 

\definecolor{darkred}{rgb}{1,0,0}        
\definecolor{lightred}{rgb}{1,0.4,0}     
\definecolor{darkblue}{cmyk}{1,0.4,0,0.4}  
\definecolor{lightblue}{cmyk}{1,0.4,0,0}  
\definecolor{darkgreen}{cmyk}{1,0.5,1,0}  
\definecolor{lightgreen}{cmyk}{1,0,1,0}  

\makeindex

\begin{document}




\frontmatter

\title{
\textsf{\Huge \hbox{Introduction to Cluster Algebras}}\\[.1in]
\textsf{\Huge Chapters 4--5} \\[.2in]
{\rm\textsf{\LARGE (preliminary version)}}
}

\author{\Large \textsc{Sergey Fomin}}

\author{\Large \textsc{Lauren Williams}}

\author{\Large \textsc{Andrei Zelevinsky}}

\maketitle

\noindent
\textbf{\Huge Preface}

\vspace{1in}


\noindent
This is a preliminary draft of Chapters 4--5 of our forthcoming textbook
\textsl{Introduction to cluster algebras}, 
joint with Andrei Zelevinsky (1953--2013). 
Other chapters have been posted as 
\begin{itemize}[leftmargin=.2in]
\item  \texttt{arXiv:1608:05735} \hbox{(Chapters~1--3)}, 
\item \texttt{arXiv:2008.09189} (Chapter~6), and
\item \texttt{arXiv:2106.02160} (Chapter~7). 
\end{itemize}
We expect to post additional chapters in the not so distant future. 

\medskip

We thank Gregg Musiker for explanations concerning the \texttt{Sage} 
package~\cite{Musiker-Stump}. 

We are grateful to Colin Defant, Chris Fraser, Sergei Gelfand, Felix Gotti, 
Amal Mattoo, Hanna Mularczyk, Emmanuel Tsukerman, and Raluca Vlad  
for a number of comments on the earlier versions of these chapters.

Our work was partially supported by the NSF grants DMS-1361789, DMS-1664722 and
DMS-1600447. 

\medskip

Comments and suggestions are welcome. 

\bigskip

\rightline{Sergey Fomin}
\rightline{Lauren Williams}

\vfill


\bigskip

\noindent
\copyright \ 2017--2021 by 
Sergey Fomin, Lauren Williams, and Andrei Zelevinsky

\tableofcontents

\mainmatter

\setcounter{chapter}{3}




\chapter{New patterns 
from old}

This chapter provides several methods for obtaining new seed patterns
(or new cluster algebras) from existing ones.

\section{Restrictions and embeddings of quivers and matrices}
\label{sec:subquivers}

We begin by discussing some purely combinatorial constructions 
involving mutations of quivers or matrices---but not clusters or seeds. 

\begin{definition}
\label{def:tildeB_I}
Let $\tilde B$ be an $m \times n$ extended skew-symmetrizable matrix.
For a subset $I\subset [1,m]$, consider the matrix 
$\tilde B_I$ obtained from $\tilde B$ by restricting to the
row set $I$ and to the column set $I \cap [1,n]$.
It is easy to see that $\tilde B_I$ is again an extended 
skew-symmetrizable matrix.  We say that $\tilde B_I$ is 
\emph{obtained from $\tilde B$ by restriction to $I$}.  
More generally, we say that a matrix $\tilde B'$ is 
{obtained from $\tilde B$ by restriction} if $\tilde B'$
can be identified with a matrix $\tilde B_I$ as above.
(Note that we will use the convention that 
the rows and columns of
	$\tilde B_I$ are labeled by $I$ rather than by $\{1,2,\dots,|I|\}$.)

If $\tilde B=\tilde B(Q)$ 
is the extended exchange matrix of a quiver~$Q$, 
 then $\tilde B_I = \tilde B(Q_I)$ 
is the extended exchange matrix of the quiver~$Q_I$,
where $Q_I$  is obtained from $Q$ by taking the subset of vertices of~$Q$
indexed by $I$ along with all 
the arrows in~$Q$ that connect the vertices in~$I$.  Such a quiver
$Q_I$
is called  
a \emph{full} (or \emph{induced}) \emph{subquiver} of $Q$.
The vertices in $Q_I$ inherit the property of being frozen or 
mutable from the ambient quiver~$Q$. 
\end{definition}

The following property is easy to check.

\begin{lemma}
\label{ex:mut-local}
Mutation of matrices/quivers commutes with restriction. 
More precisely, if $\tilde B_I$ is the restriction of 
an extended skew-symmetrizable matrix~$\tilde B$
to a subset~$I$,
and $k\in I$ is mutable, 
then $\mu_k(\tilde B_I)= (\mu_k(\tilde B))_I$. 
\end{lemma}




\begin{definition}
\label{def:hereditary}
We say that a property of extended skew-symmetrizable matrices
is \emph{hereditary}
if it is preserved under restriction: 
for any matrix~$\tilde{B}$ which has this property,
the same holds true for all its submatrices~$\tilde{B}_I$.
For quivers, a property is hereditary if it is inherited by the full
subquivers of any quiver which has that property. 
\end{definition}

We are interested in hereditary properties that are preserved under
mutations. 
The first example of this kind
concerns the notion of finite mutation type, 
cf.\ Definition~\ref{def:finitemutationtype}; this definition
can be generalized in a straightforward manner
to extended exchange matrices.
Using 
\cref{ex:mut-local}, we obtain the following.

\begin{proposition}
Finite mutation type is a hereditary property.
\end{proposition}

In other words, an extended exchange matrix obtained by restriction from an 
extended exchange matrix of finite mutation type
will again have finite mutation type.




We will show later
that the property of being
mutation-equivalent to an orientation of 
a (possibly disconnected, simply laced) Dynkin
diagram is hereditary, 
see  \cref{finite-type-hereditary}.
See also Theorem~\ref{thm:Dynkin-hereditary} for 
a version of this statement that includes extended Dynkin diagrams.


\begin{example} 
\label{ex:mutation-acyclic}
Recall that an acyclic quiver is one containing no oriented cycles. 
A quiver with no frozen vertices is called 
\emph{mutation-acyclic} if it is mutation equivalent to an acyclic quiver.
It was shown in~\cite{BMR08}, using the machinery of
quiver representations, that 
the property of being mutation-acyclic is hereditary.
It would be  interesting to find an elementary proof. 
\end{example}


In light of \cref{ex:mutation-acyclic}, it is natural to consider the unoriented
analogue of the notion of mutation-acyclicity. 

\begin{remark}
A quiver is called \emph{arborizable} if it is
mutation-equivalent to an orientation of a forest 
(i.e., an undirected simple graph with no cycles).

Unfortunately, arborizability is not a hereditary property. 
A~counterexample is given in Figure \ref{fig:arborizable}.

\begin{figure}[ht]
\begin{center}
\setlength{\unitlength}{2pt} 
\begin{picture}(40,40)(0,0) 
\put(0,10){\makebox(0,0){$\mathbf{_{1}}$}}
\put(20,0){\makebox(0,0){$\mathbf{_{2}}$}}
\put(20,20){\makebox(0,0){$\mathbf{_{5}}$}}
\put(20,40){\makebox(0,0){$\mathbf{_{4}}$}}
\put(40,10){\makebox(0,0){$\mathbf{_{3}}$}}
\thicklines 
\put(2,9){\vector(2,-1){16}}
\put(22,1){\vector(2,1){16}}
\put(2,10){\vector(1,0){36}}
\put(38,11){\vector(-2,1){16}}
\put(18,19){\vector(-2,-1){16}}
\put(38.4,12.4){\vector(-2,3){17}}
\put(18.4,37.6){\vector(-2,-3){17}}
\end{picture}
\begin{picture}(40,40)(0,0) 
\put(20,20){\makebox(0,0){$\mathbf{_{\mu_5 \circ \mu_3 \circ \mu_4}}$}}
\put(5,17.5){\vector(1,0){30}}
\end{picture}
\begin{picture}(40,40)(0,0) 
\put(0,10){\makebox(0,0){$\mathbf{_{1}}$}}
\put(20,0){\makebox(0,0){$\mathbf{_{2}}$}}
\put(20,20){\makebox(0,0){$\mathbf{_{5}}$}}
\put(20,40){\makebox(0,0){$\mathbf{_{4}}$}}
\put(40,10){\makebox(0,0){$\mathbf{_{3}}$}}
\thicklines 
\put(2,11){\vector(2,1){16}}
\put(38,11){\vector(-2,1){16}}
\put(20,18){\vector(0,-1){16}}
\put(20,22){\vector(0,1){16}}
\end{picture}

\end{center}
\caption{The quiver $Q$ shown on the left is arborizable
since it is mutation equivalent to the quiver $Q'$ on the right. 
On the other hand, 
the full subquiver~$Q_I$ on the vertex set $I=\{1,2,3\}$
is not arborizable:
$Q_I$~is of finite mutation type, 
and its mutation class $[Q_I]$ consists 
of two quivers (up to isomorphism)
neither of which is an orientation of a tree.
}
\label{fig:arborizable}
\end{figure}
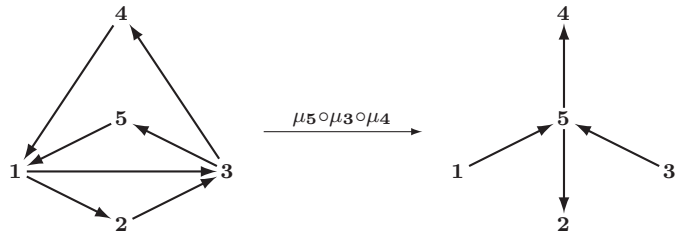
\end{remark}

\begin{lemma}
\label{pr:order-mut-classes}
For mutation classes
$\mathbf{Q}$ and $\mathbf{Q'}$, 
the following are equivalent:
\begin{itemize}[leftmargin=.3in]
\item[{\rm (i)}]
there exist $\tilde B\in \mathbf{Q}$ and $\tilde B' \in \mathbf{Q'}$ such that 
$\tilde B$ is obtained from $\tilde B'$ by restriction;
\item[{\rm (ii)}]
for any $\tilde B\in \mathbf{Q}$, there exists $\tilde B' \in \mathbf{Q'}$ 
such that $\tilde B$ is obtained from $\tilde B'$ by restriction.
\end{itemize}
\end{lemma}

\begin{proof}
The equivalence of (i) and~(ii) follows from \cref{ex:mut-local}.
\end{proof}

The notion of restriction descends to a partial order 
 on the set of mutation classes of extended exchange
matrices, as follows. 

\begin{definition}
\label{def:quiver-embedding}
Let $\mathbf{Q}$ and $\mathbf{Q'}$ be mutation classes.
We say that $\mathbf{Q}$ is \emph{embeddable} into~$\mathbf{Q'}$,
and write $\mathbf{Q}\le\mathbf{Q'}$,
if either of the two equivalent conditions (i)--(ii) in 
\cref{pr:order-mut-classes} holds.
In particular, when 
$\mathbf{Q}$ and $\mathbf{Q'}$ are mutation classes of quivers,
we say that $\mathbf{Q}$ is \emph{embeddable} into~$\mathbf{Q'}$,
if there exist quivers $Q\in\mathbf{Q}$ and $Q'\in\mathbf{Q'}$
such that $Q$ is a full subquiver of~$Q'$. 
\end{definition}

\begin{example}
\label{example:dynkin-embeddings}
In Figure~\ref{fig:E6E8},
the mutation class of each quiver (including any orientation of each of the Dynkin diagrams shown)
is embeddable into the mutation class of every quiver appearing in the rows below. 
We revisit (and extend) this example in Remark~\ref{rem:clustersubalgebra}. 
\end{example}

Recall that $[\tilde B]$ denotes the mutation class of an extended exchange matrix~$\tilde B$. 

\begin{remark}
Fix a mutation class~$\mathbf{R}$. 
Then embeddability into~$\mathbf{R}$ is a hereditary property. 
More precisely, if $[\tilde B]\le\mathbf{R}$, then $[\tilde B_I]\le\mathbf{R}$
for any matrix $\tilde B_I$ obtained by restriction from $\tilde B$.
\end{remark}

\begin{remark}
The equivalent conditions (i)--(ii) in
\cref{pr:order-mut-classes} do not imply the condition
\begin{itemize}[leftmargin=.3in]
\item[{\rm (iii)}]
for any $\tilde B' \in \mathbf{Q'}$, there exists 
 $\tilde B \in \mathbf{Q}$ such that 
 $\tilde B$ is obtained from $\tilde B'$ by restriction.
\end{itemize}
To see this, consider the
quivers $Q$ and $Q'$ shown in Figure~\ref{fig:arborizable}.
Let $\mathbf{Q} = [Q_I]$ with $I = \{1,2,3\}$ and 
$\mathbf{Q'} = [Q']$.  
Then (i)--(ii) hold while (iii) fails. 
(Note that that the mutation sequence 
relating the two quivers in \cref{fig:arborizable} uses
mutations at vertices outside of~$I$.)
\end{remark}

\begin{problem}
Let $T$ and $T'$ be finite trees, and let $\mathbf{T}$ and $\mathbf{T}'$ denote 
the mutation classes containing their respective orientations. 
Is it true that $\mathbf{T}$ is embeddable into $\mathbf{T}'$ 
if and only if $T$ can be obtained from $T'$ by contracting some edges? 
\end{problem}

\begin{remark}
Given mutation classes $\mathbf{Q}$ and~$\mathbf{R}$, 
the problem of deciding whether $\mathbf{Q}$  is embeddable into~$\mathbf{R}$
is generally very hard (unless $\mathbf{R}$ is finite). 
For example, for any positive integer~$k$,
there is no known algorithm to determine whether 
the two-vertex \emph{Kronecker quiver}
\begin{equation}
\label{eq:kronecker}
\bullet\!
\begin{array}{c}
\longrightarrow\\[-.09in]
\cdots\cdot\\[-.09in]
\longrightarrow
\end{array}
\!\bullet
\end{equation}
with $k$~arrows 
is embeddable into the mutation class of a given quiver~$R$.
That is, for any~$k$,
there is no known general method to determine whether $R$ can be mutated 
to a quiver containing two vertices connected by $k$ arrows. 
\end{remark}

\pagebreak[3]

\section{Seed subpatterns and cluster subalgebras}
\label{sec:subalgebra}



\begin{definition}
\label{def:freezing}
Let $(\tilde \xx, \tilde B)$ be a seed of rank~$n$, and let $x_i\in\tilde\xx$ be a cluster variable. 
	\emph{Freezing} at the index~$i$ (or, of the variable~$x_i$) is a
transformation of the seed that reclassifies~$i$
and~$x_i$ as frozen,
and accordingly removes the $i$th column from the exchange matrix~$\tilde B$.
(In addition, this would typically require a change of indexing, 
provided we want to keep using
the smaller indices $1,\dots,n-1$ for the mutable variables.) 
More generally, we can freeze any subset of cluster variables.  
The order of freezing does not matter.  

In the quiver case, freezing at a subset of mutable vertices
amounts to reclassifying all these vertices,
and the corresponding cluster variables,~as frozen;
and then removing all arrows connecting frozen vertices to each other. 
\end{definition}

\begin{lemma}
Freezing commutes with seed mutation. 
\end{lemma}

\begin{proof}
This property is straightforward from the definitions. 
\end{proof}

It is natural to try to extend the operation of restriction 
(to a full subquiver, or more generally to a submatrix of the exchange matrix) to the level of seeds. 
Note however that the naive notion
of restriction does \emph{not}, generally speaking, 
commute with seed mutation---because an exchange relation for a cluster variable
associated with a vertex in a subquiver
may well involve variables coming from outside the subquiver. 
This observation explains the additional constraints appearing in Definition~\ref{def:excision} below. 

\begin{definition}
\label{def:excision}
Let $(\tilde \xx, \tilde B)$ be a seed, and let 
	$I \sqcup J$ be a partition of $[1,m]$ such that 
$b_{jk}=0$ for any $j\in J$ and $k\in I \cap [1,n]$.
(In other words, none of the variables~$x_j$, for $j\in J$,
appears on the right-hand side of an exchange relation $x_k\, x_k'=\cdots$, for $k\in I$.) 
We then define the \emph{restricted seed} $(\tilde \xx_I, \tilde B_I)$ 
with the extended cluster $\tilde \xx_I=(x_i)_{i\in I}$
and the extended exchange (sub)matrix~$\tilde B_I$, cf.\ Definition~\ref{def:tildeB_I}. 
In the case where $\tilde B$ comes from a quiver~$Q$, 
this is equivalent to requiring that there are no arrows between
$I\cap [1,n]$ and $J$.
\end{definition}

\begin{example}
Let $Q$ and $Q'=\mu_k(Q)$ be the quivers 
shown in Figure~\ref{fig:quiver-mutation} on the left and on the right, respectively. 
The subset $I=\{a,q,k\}$ does not satisfy the conditions in Definition~\ref{def:excision}
(neither for~$Q$, nor for~$Q'$)
since the vertex~$k\in I$ is connected by arrows to vertices 
not in $I$.
On the other hand, if we first freeze~$k$, then we can restrict to~$I$ in~$Q'$
(but not in~$Q$). 
\end{example}

\begin{lemma}
Passing to a restricted seed commutes with seed mutation. 
\end{lemma}

It is easy to see that repeated applications of freezing and restriction
produce an outcome that can be achieved by a single application of freezing followed by restriction. 

\begin{definition}
\label{def:clustersubalgebra}
Let $\Sigma$ 
be a seed.
Freeze some subset of cluster variables,
as in Definition~\ref{def:freezing}, 
to obtain a seed~$\Sigma'$. 
After that, apply the construction 
in Definition~\ref{def:excision} to the seed~$\Sigma'$ 
to obtain a restricted  seed~$\Sigma''$. 
We then say that the seed pattern defined by~$\Sigma''$ 
	is a \emph{seed subpattern} of the seed pattern  defined by~$\Sigma$; 
and the cluster algebra associated to~$\Sigma''$ 
is a \emph{cluster subalgebra} of the cluster algebra associated
to~$\Sigma$. 
\end{definition}

Put simply, a cluster subpattern can be viewed as a part of the original pattern in which 
we are only allowed to exchange cluster variables labeled by a particular subset of indices,
and in which we discard (some of) the coefficient variables that do not appear at all
in the resulting exchange relations. 

Note that if passing to a seed subpattern (resp., cluster subalgebra)
involves freezing at least one cluster variable,
then its rank is smaller than the rank of the original pattern (resp., cluster algebra). 

\begin{example}
Let $\mathbf{P}_{n+3}$ be a convex polygon 
whose vertices are labeled $\{1,2,\dots,n+3\}$ in clockwise order.
As explained in Section~\ref{sec:triangulations}, 
each triangulation $T$ of $\mathbf{P}_{n+3}$ gives rise to a seed 
$(\tilde \xx(T), Q(T))$ whose mutable variables
are the Pl\"ucker coordinates $P_{ij}$ labeled by the diagonals
of~$T$, and 
whose frozen variables
are the Pl\"ucker coordinates labeled by the sides of~$\mathbf{P}_{n+3}$. 
We thus obtain a seed pattern and a cluster
algebra $R_{2,n+3}$ of rank~$n$.

Let $S=\{s_1< \dots < s_{\ell}\}$ 
be a subset of $\{1,2,\dots, n\}$, with $\ell\ge 4$, 
and let~$\mathbf{P}_S$ be the convex polygon on the vertex set~$S$.
Let $T$ be a tri\-angulation of $\mathbf{P}_{n+3}$ that includes
the sides of the polygon~$\mathbf{P}_S$,
and consequently contains a triangulation $T_S$ of~$\mathbf{P}_S$. 
Take the seed $(\tilde \xx(T), Q(T))$, 
and freeze all the cluster variables in~$\tilde \xx(T)$ 
corresponding to the sides of~$\mathbf{P}_S$.  
We can now restrict from $\tilde \xx(T)$ to the subset $\tilde \xx(T_S)$
consisting of the elements 
labeled by diagonals and sides of~$\mathbf{P}_S$.  
%
The resulting seed $(\tilde \xx(T_S), Q(T_S))$ defines a rank~$\ell-3$ cluster 
subalgebra (isomorphic to~$R_{2,\ell}$) of the cluster algebra~$R_{2,n+3}$.
\end{example}

The following properties of seed subpatterns and cluster subalgebras 
follow immediately from the definitions.

\pagebreak[3]

\begin{lemma}\ 
\begin{itemize}[leftmargin=.2in]
	\item Let $\mathbf{\Sigma}$  
be a seed pattern.
 A seed subpattern of a seed subpattern 
		of~$\mathbf{\Sigma}$
		is again a seed 
		subpattern of $\mathbf{\Sigma}$.
\item If a cluster algebra has finitely many cluster variables,
then so does any of its cluster subalgebras.
\item A connected component of a quiver gives rise to a seed 
subpattern.
\end{itemize}
\end{lemma}


\section{Changing the coefficients}

In this section we will establish a connection between seed patterns
utilizing the same exchange matrices but different coefficient tuples. 
For a more thorough and comprehensive treatment, see~\cite{Fraser-Quasi}. 

We begin by a brief discussion of a couple of very simple special instances of
coefficient change.

\begin{definition}
Let $(\tilde \xx, \tilde B)$ be a seed with an $m$-element extended cluster~$\tilde\xx$, 
and let $x_i\in\tilde\xx$ be 
a coefficient (or frozen) variable. 
	\emph{Trivialization} at the index~$i$ (or, of the variable~$x_i$) is a
transformation of the seed that removes~$x_i$ from~$\tilde\xx$,
and accordingly removes $i$ from the set of indices, 
and the $i$th row from the exchange matrix~$\tilde B$.
(As in the case of freezing, cf.\ \cref{def:freezing},
a renumbering may be required if we want to use
the indices $1,\dots,m-1$ after the trivialization.) 
More generally, we can trivialize any subset of coefficient variables; 
the order of operations does not matter.  

In the quiver case, trivialization 
amounts to a removal of a subset of frozen vertices, together with all arrows incident to them; 
and the removal of the corresponding coefficient variables. 
\end{definition}

\begin{lemma}
Trivialization of coefficients commutes with seed mutation. 
\end{lemma}

\begin{proof}
The key observation is that trivialization of a coefficient variable~$x_i$ 
can be interpreted as setting $x_i=1$. 
\end{proof}

It is also easy to see that trivializing coefficients commutes with taking a seed subpattern. 

\begin{remark}
Here is another simple way of transforming a seed pattern into a new one: 
introduce (any number of) additional ``dummy'' coefficient 
variables that do not appear in any exchange relations.
That is, enlarge the extended exchange matrices 
by adding rows consisting entirely of zeroes;
in the quiver case, just add isolated frozen vertices. 
This transformation changes the associated cluster algebra 
by tensoring it with the polynomial ring generated by the dummy variables. 
\end{remark}

We next describe a large class of ``rescaling'' transformations of seed patterns. 
Roughly, the idea is to multiply each cluster variable by a Laurent monomial 
in the coefficient variables, and then rewrite the exchange relations in terms of the ``rescaled''
variables. (We use quotation marks since we are not simply rescaling our variables,
but also sending them to a different ambient field of rational functions.)
The key property of this construction, formalized in Theorem~\ref{th:B-covering} below,
is that the resulting ``rescaled'' seeds are again related by (the same) mutations, 
yielding a seed pattern. 

\begin{theorem}
\label{th:B-covering}
Let $n$, $m$, $\bar m$ be positive integers, with 
$n \leq m$ and $n \leq \bar m$.
Let $\FFcal$ and $\bar\FFcal$ be the fields of rational functions in the variables
$x_1,\dots,x_m$ and $\bar x_1,\dots,\bar x_{\bar m}$, respectively. 
Let 
\[
\varphi:  \Qsf(x_1,\dots,x_m) \to \Trop(\bar x_{n+1},\dots,\bar x_{\bar m})
\]
be a semifield homomorphism
(determined by an arbitrary choice of 
Laurent monomials $\varphi(x_i)\!\in\! \Trop(\bar x_{n+1},\dots,\bar
x_{\bar m})$). 
Define the semifield~map
\begin{align*}
\psi:  \Qsf(x_1,\dots,x_m) &\to \Qsf(\bar x_1,\dots,\bar x_{\bar m})
\end{align*}
by setting 
\begin{equation}
\label{eq:psi(x_i)}
\psi(x_i)=
\begin{cases}
\bar x_i\,\varphi(x_i) & \text{if $i\le n$;}\\ 
\varphi(x_i)           & \text{if $i>n$.}
\end{cases}
\end{equation}
Let $(\xx(t), \yy(t), B(t))_{t\in\TT_n}$ be a seed pattern in~$\FFcal$
(cf.\ 
Definition \ref{def:triple-pattern}),
with 
\begin{align*}
\xx(t)&=(x_{1;t},\dots,x_{n;t})\in \FFcal^n,\\
\yy(t)&=(y_{1;t},\dots,y_{n;t})\in \Trop(x_{n+1},\dots,x_m)^n,\\
B(t)&=(b^t_{ij}),
\end{align*}
with the initial cluster 
$\xx(t_\circ)=(x_1,\dots,x_n)$,  
and with the frozen variables $x_{n+1},\dots,x_m$. 
Define $\bar\xx(t)=(\bar x_{1;t},\dots,\bar x_{n;t})$ and $\bar\yy(t)=(\bar y_{1;t},\dots,\bar y_{n;t})$~by
\begin{align}
\label{eq:bar-x}
\bar x_{i;t}&=\frac{\psi(x_{i;t})}{\varphi(x_{i;t})}, 
\\
\label{eq:bar-y}
\bar y_{k;t}&=\varphi(\hat y_{k;t}) 
= \varphi(y_{k;t}) \prod_{i=1}^n \varphi(x_{i;t})^{b^t_{ik}}. 
\end{align}
Then $(\bar\xx(t), \bar\yy(t), B(t))_{t\in\TT_n}$ is a seed pattern in~$\bar\FFcal$, 
with the same exchange matrices $B(t)$ and with the frozen variables 
$\bar x_{n+1},\dots,\bar x_{\bar m}$.
\end{theorem}

\noindent\textbf{Proof.} 
By Corollary~\ref{cor:y-hat}, the elements $\hat{y}_{i;t}$ satisfy the $Y$-pattern recurrences.
It follows that their images under the semifield homomorphism~$\varphi$,
cf.~\eqref{eq:bar-y}, 
satisfy~\eqref{eq:y-mutation-trop}, the tropical version of these recurrences. 
%
It remains to check that the elements $\bar x_{i;t}$ and $\bar y_{j;t}$ satisfy 
the exchange relations
\begin{equation}
\label{eq:exchange-rel-xx-bar}
\bar x_{k;t} \, \bar x_{k;t'}= \frac{\bar y_{k;t}}{\bar y_{k;t} \oplus 1}
\prod_{b^t_{ik}>0} \bar x_{i;t}^{b^t_{ik}}
+ 
\frac{1}{\bar y_{k;t} \oplus 1}
\prod_{b^t_{ik}<0} \bar x_{i;t}^{-b^t_{ik}} \,, 
\end{equation}
for $t\overunder{k}{}{t'}$ 
(cf.\ \eqref{eq:exchange-rel-xx}).

First note that by 
\eqref{eq:bar-y}  and the distributive property for the tropical semifield (Lemma~\ref{lem:distributive}),
we have that 
	\begin{equation}\label{eq:identity}
		\Bigl(\bar y_{k;t}\oplus 1\Bigr)
\prod_{b^t_{ik}<0} \varphi(x_{i;t})^{-b^t_{ik}} = 
\varphi(y_{k;t})\prod_{b^t_{ik}>0} \varphi(x_{i;t})^{b^t_{ik}}
\oplus
		\prod_{b^t_{ik}<0} \varphi(x_{i;t})^{-b^t_{ik}}.
	\end{equation}

We shall deduce \eqref{eq:exchange-rel-xx-bar} from the exchange relation
\begin{equation}
\label{eq:exchange-rel-xx-t}
x_{k;t} \, x_{k;t'}= \frac{1}{y_{k;t} \oplus 1}
\Bigl(
y_{k;t}\prod_{b^t_{ik}>0} x_{i;t}^{b^t_{ik}}
+ 
\prod_{b^t_{ik}<0} x_{i;t}^{-b^t_{ik}}
\Bigr). 
\end{equation}
Applying the semifield homomorphism $\psi$ (resp.,~$\varphi$) to
both sides of~\eqref{eq:exchange-rel-xx-t}, 
dividing respective images by each other,
using the fact that $\psi$ and $\varphi$ agree on 
the frozen variables $x_{n+1},\dots,x_m$,
and using 
\eqref{eq:bar-x}
and~\eqref{eq:identity},
we get:
\begin{align*}
\bar x_{k;t} \, \bar x_{k;t'}&=
\frac{\psi(x_{k;t}) \,\psi(x_{k;t'})}{\varphi(x_{k;t})\, \varphi(x_{k;t'})}\\
&=
\frac{\varphi(y_{k;t} \oplus 1)}{\psi(y_{k;t} \oplus 1)}\,\cdot\,
\frac{
\psi(y_{k;t})\prod_{b^t_{ik}>0} \psi(x_{i;t})^{b^t_{ik}}
+ 
\prod_{b^t_{ik}<0} \psi(x_{i;t})^{-b^t_{ik}}
}
{
\varphi(y_{k;t})\prod_{b^t_{ik}>0} \varphi(x_{i;t})^{b^t_{ik}}
\oplus
\prod_{b^t_{ik}<0} \varphi(x_{i;t})^{-b^t_{ik}}
}\\
&=
\frac{1}{\bar y_{k;t}\oplus 1}
\,\cdot\,
\frac{\varphi(y_{k;t})\prod_{b^t_{ik}>0} \psi(x_{i;t})^{b^t_{ik}}
+ 
\prod_{b^t_{ik}<0} \psi(x_{i;t})^{-b^t_{ik}}}
{\prod_{b^t_{ik}<0} \varphi(x_{i;t})^{-b^t_{ik}}}
\\
&=\frac{1}{\bar y_{k;t}\oplus 1}
\Bigl(
\varphi(y_{k;t})
\prod_{b^t_{ik}>0} \psi(x_{i;t})^{b^t_{ik}}
\prod_{b^t_{ik}<0} \varphi(x_{i;t})^{b^t_{ik}}
+ 
\prod_{b^t_{ik}<0} {\bar x_{i;t}}^{-b^t_{ik}}
\Bigr)\\
&=\frac{1}{\bar y_{k;t}\oplus 1}
\Bigl(
\bar y_{k;t}
\prod_{b^t_{ik}>0} {\bar x_{i;t}}^{b^t_{ik}}
+ 
\prod_{b^t_{ik}<0} {\bar x_{i;t}}^{-b^t_{ik}}
\Bigr). \qedhere
\end{align*}

We next restate \cref{th:B-covering} in terms of extended
exchange matrices. 

\begin{proposition}
\label{pr:B-covering}
Keep the assumptions and notation of \cref{th:B-covering}. 
Let~$\tilde B_\circ=\tilde B(t_\circ)$ be the initial
extended exchange matrix of the original \linebreak[3]
exchange pattern. 
Then the new initial seed $(\bar\xx(t_\circ), \bar\yy(t_\circ),B(t_\circ))$ 
has the extended exchange matrix 
$
\overline{\tilde B}_\circ=
\Psi\tilde B_\circ
$ where $\Psi=(\psi_{ij})$ is the 
$\bar m\times m$ matrix whose entries $\psi_{ij}$ are defined by
\begin{equation}
\label{eq:bar-B}
\psi(x_j)=\prod_{i=1}^{\bar m} 
\bar x_i^{\psi_{ij}}. 
\end{equation}
(Note that~\eqref{eq:psi(x_i)} implies that 
$\psi(x_j)$ is a Laurent monomial in $\bar x_1,\dots, \bar x_{\bar
  m}$.)
\end{proposition}

\begin{proof}
We use the notation 
$\overline{\tilde B}_\circ=(\bar b_{ik})$
and $\tilde B_\circ=(b_{ik})$. 
Recall that $\bar b_{ik}=b_{ik}$ for $i\le n$. 
We~have: 
\begin{align*}
\prod_{i\le\bar m} \bar x_i^{\bar b_{ik}}
&=\bar y_{k;t_\circ} \prod_{j\le n} \bar x_j^{\bar b_{jk}}
=\varphi(y_{k;t_\circ}) \prod_{j\le n} \varphi(x_j)^{b_{jk}} \prod_{j\le n} \bar x_j^{b_{jk}}\\
&= \prod_{j \le m} \varphi(x_j)^{b_{jk}} \prod_{j\le n} \bar x_j^{b_{jk}}
= \prod_{j\le m} \psi(x_j)^{b_{jk}}
= \prod_{i\le \bar m} \prod_{j\le m} \bar x_i^{\psi_{ij} b_{jk}},
\end{align*}
establishing  that
$\overline{\tilde B}_\circ=
\Psi\tilde B_\circ$.
(Here we used \eqref{eq:yj} and \eqref{eq:bar-B}.) 
\end{proof}

The following corollary will be particularly useful in 
the sequel; see the overview of \cref{ch:finitetype}
which follows \cref{thm:type}.

\begin{corollary}
\label{cor:z-span} 
Consider two seed patterns 
\begin{align*}
(\Sigma(t))_{t\in\TT_n} &= (\xx(t), \yy(t), B(t))_{t\in\TT_n},\\
(\bar \Sigma(t))_{t\in\TT_n} &= (\bar \xx(t), \bar \yy(t), \bar B(t))_{t\in\TT_n}
\end{align*} 
with the same exchange matrices~$B(t)= \bar B(t)$, for $t\in\TT_n$. 
Suppose that all rows of the initial extended exchange matrix 
$\overline{\tilde B}_{\circ}$ for the second seed pattern
lie in the $\ZZ$-span of the rows of the initial extended exchange matrix 
$\tilde B_{\circ}$ for the first seed pattern. 
If two labeled (resp., unlabeled) seeds $\Sigma(t) = \Sigma(t')$
coincide in the first pattern,
then the corresponding seeds 
$\bar \Sigma(t) = \bar \Sigma(t')$
in the second pattern coincide as well. 
\end{corollary}

\begin{proof}
We use the notation of \cref{pr:B-covering}.
In particular, we have 
 initial clusters
$\xx(t_\circ)=(x_1,\dots,x_n)$ and 
 $\bar\xx(t_\circ)=(\bar x_1,\dots,\bar x_{n})$, 
 with frozen variables $x_{n+1},\dots,x_m$  and 
	$\bar x_{n+1},\dots,\bar x_{\bar m}$ respectively.  We also have 
exchange matrices $\tilde B_\circ$ and $\overline{\tilde B}_\circ$, which are $m \times n$
and $\bar m \times n$ matrices whose top $n \times n$
submatrices equal $B(t_\circ)$.
Recall from the discussion surrounding 
	\eqref{eq:yj} 
	that an extended exchange matrix 
contains the same information as its top $n \times n$
submatrix together with the coefficient tuple~$\yy$.

The $\ZZ$-span condition in our hypothesis means that 
$\overline{\tilde B}_\circ=\Psi\tilde B_\circ$ 
where $\Psi=(\psi_{ij})$ is an integer $\bar m\times m$ matrix.
Since the top $n\times n$ submatrices of $\tilde B_\circ$ and $\overline{\tilde B}_\circ$
coincide, we may assume that $\Psi$ is a block matrix of the form
\[
\Psi=
\begin{bmatrix}
I & 0\\
\Psi_1 & \Psi_2
\end{bmatrix}
\]
where $I$ is the $n\times n$ identity matrix. 

	We then define the maps $\psi$ and~$\varphi$ by
\begin{equation*}
\psi(x_j)=\prod_{i=1}^{\bar m} 
\bar x_i^{\psi_{ij}}, \quad
\varphi(x_j)=\prod_{i=n+1}^{\bar m} 
\bar x_i^{\psi_{ij}}, 
\end{equation*}
for $1 \leq j \leq m$ 
(cf.~\eqref{eq:bar-B}), so as to agree with 
\eqref{eq:psi(x_i)} and 
\eqref{eq:bar-B}.

We claim that  {initial}  seeds 
$(\bar \xx(t_\circ), \bar \yy(t_\circ), \bar B(t_\circ))$ 
and $(\xx(t_\circ), \yy(t_\circ), B(t_\circ))$ 
are related via 
formulas \eqref{eq:bar-x}--\eqref{eq:bar-y};
once we know the claim, it follows from  
	Theorem~\ref{th:B-covering} that \emph{each} seed
$(\bar \xx(t), \bar \yy(t), \bar B(t))$ of the second seed pattern
is related to the corresponding seed 
$(\xx(t), \yy(t), B(t))$ of the first seed pattern 
via 
\eqref{eq:bar-x}--\eqref{eq:bar-y}.
And this implies the corollary.

To verify that our initial seeds satisfy \eqref{eq:bar-x}, note that 
for $1 \leq j \leq n$, we have 
$\frac{\psi(x_{j})}{\varphi(x_{j})} = \prod_{i=1}^n \bar x_{i}^{\psi_{ij}},$
which is equal to $\bar x_{j}$, since $\Psi$ restricts to the 
identity matrix on the first $n$ rows and columns.
To verify that our initial seeds also satisfy 
\eqref{eq:bar-y},
we need to use 
 \eqref{eq:yj}  to relate the $y$-variables to the extended cluster
 variables.
We find that the left-hand side of \eqref{eq:bar-y} 
becomes 
$\bar y_{k} = \prod_{\ell=n+1}^{\bar m} \bar x_\ell^{\bar b_{\ell k}}$, while
the right-hand side becomes 
$\varphi(y_{k}) \prod_{i=1}^n \varphi(x_{i})^{b_{ik}} = 
\varphi(\prod_{i=n+1}^m x_i^{b_{ik}}) \prod_{i=1}^n \varphi(x_i)^{b_{ik}}
= \prod_{i=1}^m \prod_{\ell = n+1}^{\bar m} \bar x_{\ell}^{\psi_{\ell i} b_{ik}}.$
We can see that the two sides  are equal by 
using 
the fact that 
$\overline{\tilde B}_\circ=\Psi\tilde B_\circ$, or equivalently, that 
$\bar b_{\ell k} = \sum_{i=1}^m \psi_{\ell i} b_{ik}.$
\end{proof}

\begin{remark}
\label{rem:full-z-rank}
The $\ZZ$-span condition in Corollary~\ref{cor:z-span} is in particular satisfied 
if the initial extended exchange matrix $\tilde B_\circ=\tilde B(t_\circ)$ has \emph{full $\ZZ$-rank},
i.e., the $\ZZ$-span of its rows is the entire lattice $\ZZ^n$ of integer row-vectors. 

This condition is also satisfied
if the second seed pattern has trivial coefficients, i.e., if it has no frozen variables. 
\end{remark}

\begin{exercise}
\label{exercise:A2B2G2}
Use Corollary~\ref{cor:z-span} to show that any seed pattern
with exchange matrices
\begin{equation}
B(t)=(-1)^t\,\begin{bmatrix}
 0 & 1\\
-c & 0
\end{bmatrix},
\end{equation}
with $c\in\{1,2,3\}$, 
has finitely many distinct seeds. 
To this end, consider the seed with the initial extended
exchange matrix
\[
\tilde B_\circ
=\begin{bmatrix}
 0 & 1\\
-c & 0\\
1 & 0
\end{bmatrix}. 
\]
Note that the rows of $\tilde B_\circ$ span $\ZZ^2$ over~$\ZZ$,
so by Remark~\ref{rem:full-z-rank},
it suffices to check that the seed pattern defined by~$\tilde B_\circ$
has finitely many seeds.
(In fact, it has five seeds for $c=1$,
six seeds for $c=2$,
and eight seeds for $c=3$.)
\end{exercise}

\begin{remark}
\label{rem:q-span}
In Corollary~\ref{cor:z-span}, the $\ZZ$-span condition
can be replaced by one involving a~$\QQ$-span. 
The proof remains essentially the same but requires allowing   
rational powers of the variables. 
This can be handled in two alternative ways:
algebraically, by working in the semifield of 
``subtraction-free Puiseux expressions;'' 
or analytically, 
by always choosing the branch of a fractional power that takes positive values 
at positive arguments. 
\end{remark}

\section{Folding}
\label{sec:folding}

Folding is a procedure that, under certain conditions, produces new
seed patterns from existing ones. 
The basic idea is to exploit symmetries of a quiver 
to construct a quotient object (a folded extended exchange matrix),
then design an equivariant mutation dynamics
that would drive the algebraic dynamics of ``folded seeds.'' 

In this text, we discuss folding in a somewhat limited generality;
see~\cite{dupont} for a more elaborate treatment. 
Our main application of folding will occur in 
Chapter~\ref{ch:finitetype}, 
where we use it to construct cluster algebras of finite types 
$BCFG$ from those of ``simply-laced'' types~$ADE$. 
Folding has also been used in~\cite{fst2} to classify 
the cluster algebras of finite mutation type; see Section~\ref{sec:finite-mutation-type-skew-symmetrizable}. 

We begin with a motivating example.
Consider the quiver 
\[
1 \longleftarrow 2 \longrightarrow 3 
\]
of type~$A_3$, with three mutable vertices. 
We notice the $\ZZ/2\ZZ$ symmetry of the quiver,
and place only \emph{two} distinct variables at its vertices, as follows:
\[
x_0 \longleftarrow x_1 \longrightarrow x_0\,.
\]
The exchange relations then become:
\begin{align*}
x_0 x_0' &= x_1 +1,\\
x_1 x_1' &= x_0^2+1. 
\end{align*}
If we want to preserve the symmetry, we 
can now mutate either at vertex~$2$, or simultaneously
at~$1$ and~$3$. 
Continuing in this fashion, we recover the seed pattern from
Example~\ref{example:A(1,2)}. 

\begin{definition}
\label{def:folding}
Let $Q$ be a labeled quiver, as in Definition~\ref{def:Bmatrix}. 
More explicitly, we assume that $Q$ has $m$ vertices labeled $1,\dots,m$;
the vertices labeled $1,\dots,n$ are mutable; 
the vertices labeled $n+1,\dots,m$ are frozen. 
Let $G$ be a group acting on the vertex set of~$Q$, 
or equivalently on $\{1,\dots,m\}$.
(For all practical purposes, it is safe to assume that the group $G$ is finite.) 
The notation $i\sim i'$ will mean that  $i$ and~$i'$ lie in the
same $G$-orbit. 
We say that the quiver~$Q$ (or the corresponding $m \times n$ extended
exchange matrix $\tilde B=\tilde B(Q)=(b_{ij})$) is 
\emph{$G$-admissible}~if
\begin{enumerate}
\item  \label{item:admissible-G1}
for any $i\sim i'$, index $i$ is mutable (i.e., $i\le n$) if and only if $i'$ is; 
\item \label{item:admissible-G2}
for any indices $i$ and~$j$, and any $g\in G$, 
we have $b_{ij}=b_{g(i),g(j)}\,$;
\item 
\label{item:admissible-G3}
for mutable indices $i\sim i'$, we have $b_{ii'}=0$; 
\item 
\label{item:admissible-G4}
for any $i\sim i'$, and any mutable~$j$, we have $b_{ij}b_{i'j}\ge
0$. 
\end{enumerate}
%
%
\pagebreak[3]
Assume that $Q$ is $G$-admissible. 
We call a $G$-orbit \emph{mutable} (resp., \emph{frozen})
if it consists of mutable (resp., frozen) vertices,
cf.\ condition~\eqref{item:admissible-G1} above. 
Let $\tilde B^G=\tilde B(Q)^G=(b^G_{IJ})$ 
	be the matrix whose rows (resp., columns) are
labeled by the $G$-orbits (resp., mutable $G$-orbits),
and whose entries are given~by 
\begin{equation}
\label{eq:B^G}
b^G_{IJ}=\sum_{i\in I} b_{ij}\,,  
\end{equation}
where $j$ is an arbitrary index in~$J$. 
(By condition \eqref{item:admissible-G2}, the right-hand side of~\eqref{eq:B^G}
does not depend on the choice of~$j$.) 
We then say that $\tilde B^G$ is obtained from~$\tilde B$ (or from the quiver~$Q$) 
by \emph{folding} with respect to the given $G$-action. 
\end{definition}

An example of folding is shown in Figure~\ref{fig:D4(1)}. 

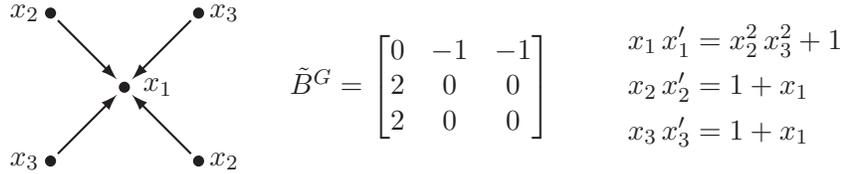
\begin{figure}[ht]
\setlength{\unitlength}{1.4pt} 
\begin{picture}(210,42)(0,0) 
\thicklines 
  \put(2,2){\vector(1,1){16}} 
  \put(2,38){\vector(1,-1){16}} 
  \put(38,38){\vector(-1,-1){16}} 
  \put(38,2){\vector(-1,1){16}} 

  \put(20,20){\circle*{3}} 
  \multiput(0,0)(40,0){2}{\circle*{3}} 
  \multiput(0,40)(40,0){2}{\circle*{3}} 

\put(-7,0){\makebox(0,0){$x_3$}}
\put(-7,40){\makebox(0,0){$x_2$}}
\put(47,0){\makebox(0,0){$x_2$}}
\put(47,40){\makebox(0,0){$x_3$}}
\put(29,20){\makebox(0,0){$x_1$}}

\put(100,20){\makebox(0,0){$\tilde B^G=\begin{bmatrix}
0 & -1 & -1\\
2 & 0 & 0\\
2 & 0 & 0
\end{bmatrix}$}}

\put(185,20){\makebox(0,0){$\begin{array}{l}
x_1\,x_1'=x_2^2\,x_3^2+1\\[.05in]
x_2\,x_2'=1+x_1\\[.05in]
x_3\,x_3'=1+x_1
\end{array}$}}

\end{picture} 

\caption{The quiver $Q$ shown on the left 
is $G$-admissible with respect to the action of the group $G=\ZZ/2\ZZ$ 
wherein the generator of~$G$ acts on the vertices 
of $Q$ by a $180^\circ$ rotation.
All $5$ vertices are mutable.}
\label{fig:D4(1)}
\end{figure}

\begin{remark}
Condition~\eqref{item:admissible-G3} 
can be restated as saying that each $G$-orbit $I$ is \emph{totally disconnected};  
that is, there is no arrow between
two vertices in~$I$. 
Condition~\eqref{item:admissible-G4}
means that there is no oriented path of length $2$ through a mutable
vertex that connects two vertices belonging to the same $G$-orbit. \linebreak[3]
These conditions are dictated by the following considerations. 
If~$i$ and $i'$ are in the same $G$-orbit~$I$, then in the folded
seed (to be defined below in this section), the same variable $x_I$ will be associated with both~$i$
and~$i'$. 
We~do not want $x_I$ to appear on the right-hand side of 
the exchange relation for~$x_I$ (hence~\eqref{item:admissible-G3}),
nor do we want $x_I$ to appear in both monomials on the right-hand side
of the exchange relation for some variable~$x_J$ 
(hence~\eqref{item:admissible-G4}).  
\end{remark}

\begin{lemma}
Let $Q$ be a $G$-admissible quiver.  
Then $\tilde B(Q)^G$ is an extended skew-symmetrizable matrix. 
\end{lemma}

\begin{proof}
As above, we use the notation $\tilde B=\tilde B(Q)=(b_{ij})$. 
Formula \eqref{eq:B^G} implies that, for $I$ and $J$ mutable, 
\begin{equation}
\label{eq:|J|b_IJ}
|J| \,b^G_{IJ}=\sum_{i\in I} \sum_{j\in J} b_{ij}
=- \sum_{j\in J} \sum_{i\in I} b_{ji}
= - |I| \,b^G_{JI}
\,, 
\end{equation}
so the square matrix $(|J| \,b^G_{IJ})$ is skew-symmetric,
and hence
	$\tilde B(Q)^G=(b^G_{IJ})$  is skew-symmetrizable.
\end{proof}

\begin{lemma}
\label{lem:folding-signs}
Let $Q$ be a $G$-admissible quiver, with $\tilde B(Q)=(b_{ij})$. 
Let $I$ and $J$ be two 
$G$-orbits, with $J$ mutable.  
Then either all entries~$b_{ij}$, for $i\!\in\! I$ and $j\!\in\! J$, 
are nonnegative, or all are nonpositive.
Consequently the following are equivalent: 
\begin{itemize}[leftmargin=.2in]
\item
$b^G_{IJ}>0$; 
\item
there exist $i\in I$ and $j\in J$ such that $b_{ij}> 0$; 
\item
for every $i\in I$, there exists $j\in J$ such that $b_{ij}> 0$.   
\end{itemize}
Similar equivalences hold with all inequality signs reversed. 
\end{lemma}

\begin{proof}
Let $i,i'\in I$ and $j,j'\in J$. 
Let $g\in G$ be such that $g(j')=j$. 
Using conditions~\eqref{item:admissible-G2}
and~\eqref{item:admissible-G4} of Definition~\ref{def:folding}, 
we get $b_{ij}b_{i'j'}=b_{ij}b_{g(i'),j}\ge 0$, as desired. 
The equivalence statements in the lemma follow by~\eqref{eq:B^G} together with
condition~\eqref{item:admissible-G2}. 
\end{proof}

For $Q$ a $G$-admissible quiver,  
condition~\eqref{item:admissible-G3} of Definition~\ref{def:folding}
ensures 
(cf.\ \linebreak[3]
Exercise~\ref{ex:mut-simple}\eqref{ex:mut-commute})
that the result of mutating $Q$ at the set of all vertices in a mutable
$G$-orbit~$K$ 
does not depend on the order of mutations. 
We will denote this composition of mutations by 
\begin{equation}
\label{eq:mu-K}
\mu_K=\prod_{k\in K}\mu_k\,,
\end{equation}
making sure to only use this notation when it is well defined. 

\begin{lemma}
\label{lem:mut-folding}
Let $Q$ be a $G$-admissible quiver, with $\tilde B=\tilde B(Q)$.  
Let $K$ be a mutable \hbox{$G$-orbit} \linebreak[3]
such that $\mu_K(Q)$ is also 
$G$-admissible. 
Then 
\[
(\mu_K(\tilde B))^G\!=\!\mu_K(\tilde B^G).
\] 
\end{lemma}

Note the abuse of notation in the last equality:
on the left, $\mu_K$~is a composition of mutations defined
by~\eqref{eq:mu-K}; on the right, $\mu_K$~is a single mutation
at the mutable index~$K$ of the folded matrix~$\tilde B^G$. 

\noindent\textbf{Proof.} 
Using the definition of matrix mutation~\eqref{eq:matrix-mutation}
in combination with Lemma~\ref{lem:folding-signs}, we obtain: 
\begin{equation}
\label{eq:mutation-folding0}
\mu_K(\tilde B^G)_{IJ} =
\begin{cases}
-b_{IJ}^G & \text{if $K\!\in\!\{I,J\}$;} \\[.05in]
b_{IJ}^G\!+\!b_{IK}^G b_{KJ}^G & \text{if $K\!\notin\!\{I,J\}$ and $b_{ik}> 0$ and 
$b_{kj}> 0$}\\
& \text{\qquad for some $i\in I, j\in J, k\in K$;}\\[.05in]
b_{IJ}^G\!-\!b_{IK}^G b_{KJ}^G & \text{if $K\!\notin\!\{I,J\}$ and $b_{ik}< 0$ and 
$b_{kj}< 0$}\\
& \text{\qquad for some $i\in I, j\in J, k\in K$;}\\[.05in]
b_{IJ}^G & \text{otherwise.}
\end{cases}
\end{equation}
On the other hand, mutating $\tilde B$ at each 
$k\in K$ (recall that $K$ is totally disconnected), we get: 
\begin{equation}
\label{eq:mutation-folding}
\mu_K(\tilde B)_{ij} =
\begin{cases}
-b_{ij} & \text{if $i\in K$ or $j\in K$ (i.e., $K\!\in\!\{I,J\}$);} \\[.05in]
b_{ij}+\displaystyle\sum_{k\in K} b_{ik}b_{kj} & \text{if $i\notin K$,
  $j\notin K$ (i.e., $K\!\notin\!\{I,J\}$), and}\\[-.15in]
& \text{\qquad $b_{ik}> 0$ and $b_{kj}> 0$ for some $k\in K$;}\\[.05in]
b_{ij}-\displaystyle\sum_{k\in K} b_{ik}b_{kj} & 
     \text{if $i\notin K$, $j\notin K$ (i.e., $K\!\notin\!\{I,J\}$), and}\\[-.15in]
& \text{\qquad $b_{ik}< 0$ and $b_{kj}< 0$ for some $k\in K$;}\\[.05in]
b_{ij} & \text{otherwise.}
\end{cases}
\end{equation}
%
By~\eqref{eq:B^G} (recall that $\mu_K(\tilde B)$ is $G$-admissible),
we have, for any $j\in J$: 
\[
((\mu_K(\tilde B))^G)_{IJ} = \sum_{i\in I} (\mu_K(\tilde B))_{ij}\,.
\]
We note furthermore that all terms in the last sum will fall into
the same case in~\eqref{eq:mutation-folding}.
It is now easy to check that $(\mu_K(\tilde B^G))_{IJ} = 
((\mu_K(\tilde B))^G)_{IJ}$.  For example, in the second 
case of \eqref{eq:mutation-folding}, we have: 
\begin{align*}
((\mu_K(\tilde B))^G)_{IJ} 
&= \sum_{i\in I} \bigl(b_{ij} + \sum_{k\in K} b_{ik}b_{kj}\bigr)\\
&= b_{IJ}^G + \sum_{k\in K} b_{kj} \sum_{i\in I} b_{ik}  \\
&= b_{IJ}^G + b_{IK}^G  b_{KJ}^G \\
&= \mu_K(\tilde B^G)_{IJ}. \qedhere
\end{align*}

\pagebreak[3]

The ``mutation commutes with folding'' statement in Lemma~\ref{lem:mut-folding}
comes with a caveat:
it requires admissibility of \emph{both} 
quivers $Q$ and~$\mu_K(Q)$ with respect to the group action at hand. 
Unfortunately, admissibility does not propagate via mutations: 
$Q$ may be $G$-admissible while $\mu_K(Q)$ is not. 

\pagebreak[3]

\begin{example}
Let $Q$ be an oriented $6$-cycle with six mutable vertices labeled $0$ through~$5$
in clockwise order. 
Let the generator of $G=\ZZ/2\ZZ$ act by sending each vertex~$i$ to 
the vertex $(i+3)\bmod 6$.
Then $Q$ is $G$-admissible but for any mutable 
	$\ZZ/2\ZZ$-orbit $K$, the quiver $\mu_K(Q)$ is not. 
\end{example}

We next proceed to the folding of seeds. 
This will require, in addition to an action of a group~$G$,
a choice of a semifield homomorphism that ``bundles together'' the variables
associated with the vertices in the same $G$-orbit. 

\begin{definition}
\label{def:seed-folding}
Let $G$ be a group acting on the set of indices $\{1,\dots, m\}$ so that 
every $g\in G$~maps the subset $\{1,\dots, n\}$ to itself. 
Let $m^G$ denote the number of orbits of this action. 
Let $\FFcal$ (resp., $\FFcal^G$) be a field isomorphic to the field of rational functions in~$m$ 
(resp.,~$m^G$) independent variables,
and let $\FFcal_\textrm{sf}$ (resp.,~$\FFcal_\textrm{sf}^G$)
denote the corresponding
semifields of subtraction-free rational expressions. 
Let 
$\psi: \FFcal_\textrm{sf} \to \FFcal_\textrm{sf}^G$
be a surjective semifield homomorphism. 

Let $Q$ be a quiver as above. 
A seed $\Sigma=(\tilde\xx,\tilde B(Q))$ in~$\FFcal$, 
with the extended cluster $\tilde\xx=(x_i)$, is called \emph{$(G,\psi)$-admissible} if 
\begin{itemize}[leftmargin=.2in]
\item
$Q$ is a $G$-admissible quiver; 
\item
for any $i\sim i'$, we have $\psi(x_i)=\psi(x_{i'})$. 
\end{itemize}
In this situation, we define a new ``folded'' seed 
$\Sigma^G=(\tilde\xx^G,\tilde B^G)$ in $\FFcal_\textrm{sf}^G\subset\FFcal^G$
whose extended exchange matrix $\tilde B^G$ is given
by~\eqref{eq:B^G}, 
and whose extended cluster 
$\tilde\xx^G=(x_I)$ has $m^G$ elements $x_I$ indexed by the $G$-orbits
and defined by $x_I= \psi(x_i)$, for~$i\in I$. 
Note that since $\psi$ is a surjective homomorphism, the elements $x_I$ generate~$\FFcal^G$,
hence are algebraically independent. 
\end{definition}

We can now extend Lemma~\ref{lem:mut-folding} to the folding of seeds. 

\begin{lemma}
\label{lem:mut-folded-seeds}
Let 
$\Sigma=(\tilde \xx,\tilde B(Q))$ be a  $(G,\psi)$-admissible seed
as above. 
Let $K$ be a mutable $G$-orbit. 
If the quiver $\mu_K(Q)$ is $G$-admissible,  
then the seed $\mu_K(\Sigma)$ is $(G,\psi)$-admissible, 
and moreover $(\mu_K(\Sigma))^G=\mu_K(\Sigma^G)$. 
\end{lemma}

\begin{proof}
Let $\mathcal{O}$ denote the set of $G$-orbits. 
As above, we use the notation 
\begin{alignat*}{3}
\Sigma&=(\tilde\xx,\tilde B), & \tilde\xx&=(x_i), & \tilde B&=\tilde B(Q)=(b_{ij}), \\
\Sigma^G&=(\tilde\xx^G,\tilde B^G), \quad &
\tilde\xx^G&=(x_I)_{I\in\mathcal{O}} ,\quad & \tilde B^G&=(b^G_{IJ}) 
\end{alignat*}
By Lemma \ref{lem:mut-folding}, all we need to show is that 
the extended clusters  in $\mu_K(\Sigma^G)$ and in $(\mu_K(\Sigma))^G$ are the same. 
The extended cluster in $\mu_K(\Sigma^G)$ is 
obtained from 
$\tilde\xx^G$ by replacing $x_K$ by the element $x'_K$ defined by 
\begin{align}
x_K x'_K &= \prod_{b_{IK}^G>0} x_I^{b_{IK}^G} +
            \prod_{b_{IK}^G<0} x_I^{-b_{IK}^G}.
\label{eq:firstfold}
\end{align}

The extended cluster in $\mu_K(\Sigma)$ is 
obtained from
$\tilde\xx$ by replacing each~$x_k$, for $k\in K$, 
by the element $x'_k$ defined by 
\begin{equation}\label{eq:firstmutate}
x_k x'_k = \prod_{b_{ik}>0} x_i^{b_{ik}} + \prod_{b_{ik}<0} x_i^{-b_{ik}}.
\end{equation}
The extended cluster in $(\mu_K(\Sigma))^G$ is then obtained 
by applying the homomorphism~$\psi$;
as we know, $\psi$ sends each~$x_i$ to $x_{[i]}$ where $[i]$ denotes the $G$-orbit containing~$i$.
As a result, the cluster contains the variables 
$x_I$ for $I\in \mathcal{O}-\{K\}$, 
together with $\psi(x'_k)$; here $k$ is an arbitrary element of~$K$.
We need to show that $\psi(x'_k)=x'_K$ where $x_K'$ is defined by~\eqref{eq:firstfold}.

Applying $\psi$ to \eqref{eq:firstmutate}, we get 
\begin{equation}\label{eq:firstmutate1}
x_K \psi(x'_k) = \prod_{b_{ik}>0} x_{[i]}^{b_{ik}} + 
\prod_{b_{ik}<0} x_{[i]}^{-b_{ik}}.
\end{equation}
Note that in the first monomial on the right-hand 
side of \eqref{eq:firstmutate1}, the exponent of a given 
$x_I$ will be nonzero if and only if 
there exists $i\in I$ such that $b_{ik}>0$, in which case the
exponent will be
$\sum_{i\in I} b_{ik}=b_{IK}^G$.  But by Lemma \ref{lem:folding-signs},
the existence of such an index~$i\in I$ is equivalent to the inequality 
$b_{IK}^G>0$.  It follows that the first monomials in the right-hand
sides of \eqref{eq:firstmutate1} and \eqref{eq:firstfold} agree.
A similar argument shows that the second monomials agree, 
and we are done.
\end{proof}

\begin{definition}
Let $G$ be a group acting on the vertex set of a quiver~$Q$.  
We say that $Q$ is \emph{globally foldable} with respect to $G$
if $Q$ is $G$-admissible and moreover 
for any sequence of mutable $G$-orbits $J_1,\dots,J_k$, the quiver
$(\mu_{J_k}\circ\cdots\circ\mu_{J_1})(Q)$ is $G$-admissible. 
\end{definition}

\begin{exercise}
For the quiver~$Q$ and the action of the group $G=\ZZ/2\ZZ$
described in Figure~\ref{fig:D4(1)}, 
show that $Q$ is globally foldable with respect to~$G$.
\end{exercise}

Lemma~\ref{lem:mut-folded-seeds} implies that 
if $Q$ is globally foldable, then we can fold all the seeds in the
corresponding seed pattern. 

\pagebreak[3]

\begin{corollary}
\label{cor:always-admissible}
Let $Q$ be a quiver which is globally foldable with respect to a group~$G$
acting on the set of its vertices.
Let $\Sigma=(\tilde\xx,\tilde B(Q))$ be a seed
in the field $\FFcal$ of rational functions freely generated 
by an extended cluster~$\tilde\xx=(x_i)$.
Let $\tilde\xx^G=(x_I)$ be a collection of formal variables labeled by the $G$-orbits~$I$,
and let $\FFcal^G$ denote the field of rational functions 
in these variables. 
Define the surjective homomorphism 
\begin{align*}
\psi:\FFcal_\mathrm{sf}&\to \FFcal_\mathrm{sf}^G\\
x_i&\mapsto x_I \quad (i\in I)
\end{align*}
of the corresponding semifields of subtraction-free rational
expressions, 
so that $\Sigma$ is a $(G,\psi)$-admissible seed. 
Then for any mutable $G$-orbits $J_1,\dots,J_k$,
the seed $(\mu_{J_k}\circ\cdots\circ\mu_{J_1})(\Sigma)$ 
is $(G,\psi)$-admissible, 
and moreover the folded seeds  $((\mu_{J_k}\circ\cdots\circ\mu_{J_1})(\Sigma))^G$ 
form a seed pattern in~$\FFcal^G$, with the initial extended exchange matrix $(\tilde B(Q))^G$. 
\end{corollary}

In general, it may be very difficult to determine whether a quiver 
is globally foldable.
Fortunately, the cases that we will 
need in Chapter~\ref{ch:finitetype}  for the purposes of finite type classification
will turn out to be easy to handle. 
One of these cases is discussed in Exercise~\ref{ex:E6F4} below. 
We will revisit this exercise in \cref{sec:exceptional2}.

\begin{exercise}
\label{ex:E6F4}
Let $Q$ be the quiver at the left of 
\cref{fig:E6-F4}, with all its vertices mutable. 
Let the generator of the group $G=\ZZ/2\ZZ$ act on the vertices
of $Q$ 
by fixing the vertices $1$ and~$2$,
exchanging the vertices $3$ and~$4$, and exchanging the vertices $5$ and~$6$.
Then $Q$ is $G$-admissible;
the folded skew-symmetrizable matrix  $B(Q)^G$ is shown at the right in \cref{fig:E6-F4}.
(The rows and columns of this matrix are labeled by the orbits
$\{1\}, \{2\}, \{3,4\}, \{5,6\}$.)
Show that $Q$ is globally foldable,
by cataloguing all quivers (up to $G$-equivariant isomorphism) 
which can be obtained from~$Q$ by iterating the transformations~$\mu_J$ 
associated with $G$-orbits~$J$.  
\end{exercise}

\begin{figure}[ht]
\setlength{\unitlength}{1.6pt}
\begin{picture}(180,30)(0,-22)
\thicklines
\put(2,0){\vector(1,0){16}}
\put(38,0){\vector(-1,0){16}}
\put(40,-2){\vector(0,-1){16}}
\put(78,0){\vector(-1,0){16}}
\put(42,0){\vector(1,0){16}}
\put(40,-20){\circle*{2}}
\multiput(0,0)(20,0){5}{\circle*{2}}
\put(0,5){\makebox(0,0){$5$}}
\put(20,5){\makebox(0,0){$3$}}
\put(40,5){\makebox(0,0){$2$}}
\put(40,-25){\makebox(0,0){$1$}}
\put(60,5){\makebox(0,0){$4$}}
\put(80,5){\makebox(0,0){$6$}}

\put(150,-10){\makebox(0,0){$\tilde B^G=\begin{bmatrix}
0 & -1 & 0 & 0\\
1 & 0 & 1& 0\\
0 & -2 & 0 &-1\\
0 & 0 & 1 & 0 
\end{bmatrix}$}}

\end{picture}

\caption{Folding a type $E_6$ quiver. 
}
\label{fig:E6-F4}
\end{figure}
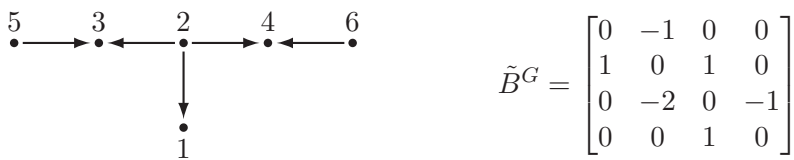

\begin{remark}
There exist skew-symmetrizable square matrices which cannot be obtained
by folding a globally foldable quiver, see \cite[Section~14]{Labardini-Zelevinsky}.
Consequently the technique of folding (in the current state of the art) 
is not powerful enough to reduce the study of general seed patterns,
and the associated cluster algebras, to the quiver case. 
\end{remark}


\chapter{\hbox{Finite type classification}}
\label{ch:finitetype}

\def\Seeds{\operatorname{\# seeds}}
\def\Vars{\operatorname{\# clvar}}

\vspace{-.3in}


A seed pattern (or the corresponding cluster algebra)
is said to be of \emph{finite type} if it 
has finitely many different seeds. 
%
To rephrase, a seed gives rise to a pattern (or cluster algebra) of
finite type if the process of iterated mutation produces finitely many
distinct seeds. 

If a seed pattern has finite type, then it obviously has finitely many distinct cluster variables. 
In fact the converse is also true, 
see \cref{cor:finitely-many-seeds}.

The main result of this chapter is the classification, originally
obtained in~\cite{ca2}, 
of seed patterns (equivalently cluster algebras) of finite type. 
It turns out that the property of a cluster algebra 
with an initial seed $(\xx, \yy, B)$ to be of finite type
depends only on (the mutation class~of) the exchange matrix~$B$ but not on the choice of 
a coefficient tuple~$\yy$.
Even more remarkably, such mutation classes 
are in one-to-one correspondence with Cartan matrices of finite type,
or equivalently with finite crystallographic root systems.

To show that a cluster algebra of finite type must come from a Cartan
matrix of finite type, we follow the approach of~\cite{ca2}.   In particular 
we show that a cluster algebra has finite type 
if and only if all its exchange matrices $B=(b_{ij})$ have the property 
that $|b_{ij} b_{ji}| \leq 3$ for any pair of indices $i,j$.

The fact that cluster algebras coming from finite-type Cartan matrices are of finite
type is established by a string of case-by-case arguments. 
For each type, we construct a particular seed pattern 
that has finitely many seeds,
and whose initial extended exchange matrix 
has full $\mathbb{Z}$-rank. \linebreak[3]
A~generalization to arbitrary coefficients
is then obtained via Remark~\ref{rem:full-z-rank}. \linebreak[3]
We note that 
the original proof \cite{ca2} of this direction of the classification
theorem used a different (root-theoretic) strategy 
relying on some rather delicate properties of \emph{generalized 
associahedra}, cf.\ 
Section~\ref{sec:generalized-associahedra}. 


We are not aware of a simple argument that would directly derive the classification 
of cluster algebras of finite type from one of the instances
of the classical Cartan-Killing classification (see Theorem~\ref{th:Cartan-Killing-list}). 

\section{Finite type classification in rank~$2$}
\label{sec:finite-type-rank2}

Any seed pattern of rank~$1$ has two seeds, so is of finite type. 
In this section, we determine which seed patterns of rank~2 are of finite type. 

Recall that the seeds 
in any seed pattern of rank~2 
can be labeled by integers $t\in\ZZ$.
Without loss of generality, we may assume
that the exchange matrices $B(t)$ are given by 
\begin{equation}
\label{eq:Bbc-again}
B(t)=(-1)^t 
\begin{bmatrix}
0 & b\\
-c & 0
\end{bmatrix}
\end{equation}
where the integers $b$ and $c$ are either both positive, or both equal to~$0$. 

\begin{theorem}
\label{th:finite-type-rank2}
A seed pattern of rank~$2$,
with exchange matrices given by~\eqref{eq:Bbc-again},
is of finite type if and only if $bc\le 3$. 
\end{theorem}

The statement of 
Theorem~\ref{th:finite-type-rank2} is very similar to
the classification of finite crystallographic reflection groups of
rank~2, which we recall in Proposition~\ref{rem:weyl-rank2} below.
(The proof of Theorem~\ref{th:finite-type-rank2} will not depend on this proposition.) 

\begin{proposition}
\label{rem:weyl-rank2}
Consider the subgroup $W\subset\GL_2$ generated by the reflections
\begin{equation}
\label{eq:s1-s2-bc}
s_1=\begin{bmatrix}
-1 & b\\
0 & 1
\end{bmatrix},\quad
s_2=\begin{bmatrix}
1 & 0\\
c & -1
\end{bmatrix}. 
\end{equation}
(As above, 
$b,c\in\ZZ$ are either both positive, or both equal to~$0$.)
Then $W$ is a finite group if and only if $bc \leq 3$.
\end{proposition}

\begin{proof}
Since $s_1^2=s_2^2=1$, the group~$W$ is finite if and only if the
element
\[
s_1 s_2=\begin{bmatrix}
bc-1 & -b\\
c & -1
\end{bmatrix}
\] 
is of finite order. 
An eigenvalue $\lambda$ of~$s_1s_2$ satisfies the
characteristic equation
\begin{equation}
\label{eq:lambda-char}
\lambda^2-(bc-2)\lambda+1=0. 
\end{equation}
If $bc\in\{1,2,3\}$, then the roots of \eqref{eq:lambda-char} are two distinct roots of
unity, so $s_1s_2$ has finite order
(specifically, order 3, 4, or~6, respectively). 
If $b=c=0$, then $(s_1s_2)^2=1$ by inspection. 
If $bc>4$, then the eigenvalues are real and not equal to $\pm 1$, 
so $W$ is infinite. 
For $bc=4$, one checks that
\[
(s_1s_2)^k=
\begin{bmatrix}
2k+1 & -kb\\
kc & -2k+1
\end{bmatrix},
\]
implying that $W$ is infinite in this case as well. 
\end{proof}

\begin{proof}[Proof of Theorem~\ref{th:finite-type-rank2}]
The case $b=c=0$ is trivial. 
The cases $bc\in\{1,2,3\}$ are handled by
the calculation from Exercise~\ref{exercise:A2B2G2}. 
An alternative (more conceptual) explanation of why the
rank~2 cluster algebras with $bc\le 3$ are of finite type
will be given later in this chapter.

Recall from \cref{exercise:A2B2G2} that 
any seed pattern 
with exchange matrices
\begin{equation}
B(t)=(-1)^t\,\begin{bmatrix}
 0 & 1\\
-c & 0
\end{bmatrix},
\end{equation}
with $c\in\{1,2,3\}$, 
has finitely many distinct seeds.

Now suppose that $bc\ge 4$.
Let 
$(\xx(t), \yy(t), B(t))_{t\in\ZZ}$
be a seed pattern in an ambient field~$\FFcal$
whose exchange matrices are given by~\eqref{eq:Bbc-again}.
We denote the cluster variables in this pattern
by $z_t$ ($t\in\ZZ$), so that 
\[
\dots,\,\xx(0)=(z_1,z_2),\ \ \xx(1)=(z_3,z_2),\ \ \xx(2)=(z_3,z_4),\ \ \xx(3)=(z_5,z_4),
\dots
\]
Our goal is to show that the set $\{z_t:t\in\ZZ\}\subset\FFcal$ is infinite.
(In fact, all cluster variables $z_t$ turn out to be
distinct.) 

Let $u$ be a formal variable, and consider the semifield
$U=\{u^r\,:\,r\in\RR\}$ of 
formal monomials in~$u$ with real exponents, 
with the operations defined~by
\begin{align*}
u^r\oplus u^s &= u^{\max(r,s)},\\
u^r\cdot u^s &= u^{r+s}. 
\end{align*}
We shall prove that the set $\{z_t\}$ is infinite 
by constructing a semifield homomorphism~$\psi:\FFcal\to U$
such that the image $\{\psi(z_t):t\in\ZZ\}\subset U$ is infinite. 
We consider the cases $bc>4$ and $bc=4$ separately. 

\textbf{Case~1:} $bc>4$. 
In this case, there is a real number
$\lambda>1$ satisfying~\eqref{eq:lambda-char}. 
Let $\psi$ be the map uniquely defined 
by setting the image of every frozen variable to be~1,
and setting
\begin{equation}
\label{eq:psi-initial}
\psi(z_1)=u^c,\quad \psi(z_2)=u^{\lambda+1}. 
\end{equation}
The exchange relations imply that the images $\psi(z_t)$ satisfy
\begin{equation}
\label{eq:Abc-trop}
\psi(z_{t-1})\,\psi(z_{t+1})=
\begin{cases}
\psi(z_t)^c \oplus 1 & \text{if $t$ is even;}\\[.05in]
\psi(z_t)^b \oplus 1 & \text{if $t$ is odd}
\end{cases}
\end{equation} \pagebreak[3]
(cf.\ \eqref{eq:Abc}). We claim that, for $k=0,1,2,\dots$, one has
\begin{equation}
\label{eq:z-via-lambda}
\psi(z_{2k+1})=u^{\lambda^k c},\quad \psi(z_{2k+2})=u^{\lambda^k(\lambda+1)}. 
\end{equation}
The base case $k=0$ holds in view of~\eqref{eq:psi-initial}. 
Induction step:
\begin{align*}
\psi(z_{2k+3})&=\frac{\psi(z_{2k+2})^c \oplus 1}{\psi(z_{2k+1})}
=u^{\lambda^k(\lambda+1)c-\lambda^k c}
=u^{\lambda^{k+1} c},\\
\psi(z_{2k+4})&=\frac{\psi(z_{2k+3})^b \oplus 1}{\psi(z_{2k+2})}
=u^{\lambda^{k+1} cb-\lambda^k(\lambda+1)}
=u^{\lambda^k( \lambda cb-\lambda-1)}
=u^{\lambda^{k+1}(\lambda+1)}, 
\end{align*}
where the last equality relies on~\eqref{eq:lambda-char}. 
Since $\lambda>1$, formulas \eqref{eq:z-via-lambda} imply that the
image set $\{\psi(z_t)\}\subset U$ is infinite. 

\textbf{Case~2:} $bc=4$. 
While the general logic of the argument remains the same,
it has to be adjusted since in this case, 
the only root of the equation~\eqref{eq:lambda-char} is $\lambda=1$. 
So we replace~\eqref{eq:psi-initial} by
\[
\psi(z_1)=u,\quad \psi(z_2)=u^b,
\]
and verify by induction that $\psi(z_{2k-1})=u^{2k-1}$
and $\psi(z_{2k+2})=u^{(k+1)b}$,
implying the claim.
Details are left to the reader. 
\end{proof}

Theorem~\ref{th:finite-type-rank2} suggests the following definition. 

\begin{definition}
\label{def:2-finite}
A skew-symmetrizable matrix~$B=(b_{ij})$ is called \emph{$2$-finite} 
if and only if 
for any matrix $B'$ mutation equivalent to~$B$ and any indices $i$ and~$j$, we have
$|b'_{ij}b'_{ji}| \leq 3$.
\end{definition}

\begin{corollary}
\label{cor:finite-type-2finite}
In a seed pattern of finite type, every exchange
matrix is $2$-finite. 
\end{corollary}

\begin{proof}
This is an immediate consequence of
Theorem~\ref{th:finite-type-rank2}.  
If an exchange matrix $B$ is mutation equivalent to $B'=(b_{ij}')$ such
that $|b'_{ij}b'_{ji}| \geq 4$ for some $i$ and~$j$, 
then 
``freezing" all the cluster variables in the corresponding seed 
except for $x_i$ and $x_j$, and 
	alternately applying mutations $\mu_i$ and~$\mu_j$ to the
corresponding seed, we obtain infinitely many 
distinct cluster variables, 
and infinitely many distinct seeds.  
\end{proof}

We will eventually show that the 
converse to \cref{cor:finite-type-2finite} holds as well, see
\cref{th:finite-type-bound-by-3}.

\pagebreak[3]

\section{Cartan matrices and Dynkin diagrams}\label{sec:Cartan}

The classification of cluster algebras of finite type turns out to be
completely parallel to the famous 
Cartan-Killing classification of semisimple Lie algebras, finite
crystallographic root systems, etc.
The latter classification can be found in many books,  e.g.,~\cite{fulton-harris,kac}.
In this section, we quickly review it, using the language of
\emph{Cartan matrices} and \emph{Dynkin diagrams}. 
We then explain the connection between 
(symmetrizable) Cartan matrices and (skew-symmetrizable) exchange
matrices. 

\begin{definition}
A square integer matrix~$A=(a_{ij})$ is called a
\emph{symmetrizable generalized Cartan matrix} if it satisfies the following conditions:
\begin{itemize}[leftmargin=.2in]
\item all diagonal entries of $A$ are equal to~$2$;
\item all off-diagonal entries of $A$ are non-positive;
\item
there exists a diagonal matrix~$D$ with positive diagonal entries such
that the matrix $DA$ is symmetric. 
\end{itemize} \pagebreak[3]
We call $A$ \emph{positive} if 
$DA$ is positive definite; 
this is equivalent to the positivity of all 
principal minors~$\Delta_{I,I}(A)$.  
In particular, 
any such matrix satisfies
\[
\Delta_{\{i,j\},\{i,j\}}(A)
=\det\begin{pmatrix}
2 & a_{ij}\\
a_{ji} & 2
\end{pmatrix}
=4-a_{ij} a_{ji}>0, 
\]
or equivalently
\begin{equation}
\label{eq:cartan-bound-3}
a_{ij} a_{ji} \leq 3 \ \text{for $i \neq j$.}
\end{equation}
Positive symmetrizable generalized Cartan matrices are often referred to simply
as \emph{Cartan matrices}, or \emph{Cartan matrices of finite type}.
\end{definition}

\begin{example}
In view of \eqref{eq:cartan-bound-3}, a $2 \times 2$ matrix
\[
A=\begin{bmatrix}
2 & -b\\
-c & 2
\end{bmatrix}
\]
is a Cartan matrix of finite type if and only if one of the following
holds:
\begin{itemize}[leftmargin=.2in]
\item
$b=c=0$; 
\item
$b=c=1$; 
\item
$b=1$, $c=2$ or $b=2$, $c=1$; 
\item
$b=1$, $c=3$ or $b=3$, $c=1$. 
\end{itemize}
Note that this matches the classifications in
Theorem~\ref{th:finite-type-rank2} and
\cref{rem:weyl-rank2}. 
The latter match has a well-known explanation: 
there is a canonical correspondence between
Cartan matrices of finite type and finite Weyl groups (or finite
crystallographic root systems). 
The relationship between these objects and cluster algebras of finite type is
much more subtle. 
\end{example}

\begin{remark}
A Cartan matrix encodes essential information 
about the geometry of a root system 
(or the corresponding Weyl group). 
The classification of finite crystallographic 
root systems (resp., associated reflection groups) can be
reduced to classifying Cartan matrices of finite type.
This standard material can be found in many books,
see, e.g.,~\cite{humphreys}.
\end{remark}

\begin{definition}
\label{def:coxeter-dynkin}
The \emph{Coxeter graph} of an $n \times n$ Cartan matrix~$A$ is
a simple graph with vertices $1, \dots, n$ in which vertices $i$ and
$j\neq i$ are
joined by an edge whenever $a_{ij} \neq 0$.
If $a_{ij} \in \{0, -1\}$ for all $i \neq j$, then $A$ is uniquely
determined by its Coxeter graph. 
(Such matrices are called \emph{simply-laced}.)
If $A$ is not simply-laced but of finite type then,
in view of \eqref{eq:cartan-bound-3}, one needs a
little additional information to specify $A$.
This is done by replacing the Coxeter graph of $A$ with its
\emph{Dynkin diagram} in which, 
instead of being connected by 
	a single edge, each pair of vertices $i$ and $j$ with
$a_{ij} a_{ji} > 1$ is shown as follows:
\begin{center}
\begin{tabular}{rcll}
$i$ & \hspace{-.1in}\raisebox{-0.7mm}{\includegraphics[width=10mm]{B2dynk.ps}}
       &\hspace{-.15in} $j$ & if
  $a_{ij}=-1$ and $a_{ji}=-2$;\\[.05in] 
$i$ & \hspace{-.1in}\raisebox{-0.7mm}{\includegraphics[width=10mm]{G2dynk.ps}}
       &\hspace{-.15in} $j$ & if $a_{ij}=-1$ and
  $a_{ji}=-3$. 
\end{tabular}
\end{center}
\end{definition}

Note that our usage of the terms \emph{Coxeter graph} and 
\emph{Dynkin diagram} is a bit non-standard:
Coxeter graphs are usually defined as edge-labeled graphs, 
and 
 Dynkin diagrams are often assumed to be connected.
(We make no such requirement.)
Also, as in~\cite{ca2}, we use the conventions
of~\cite{kac} 
(as opposed to those~in~\cite{bourbaki}) 
in going between Dynkin diagrams and Cartan matrices.

A~couple of examples are shown in Figure~\ref{fig:B3-C3}.
The notation $B_3$ and $C_3$ is explained
in Figure \ref{diagrams}.

\begin{figure}[ht]
\vspace{-.2in}
\[
\begin{array}{ccc}
B_3\quad
&
\setlength{\unitlength}{1.5pt}
\begin{picture}(40,17)(0,-2)
\thicklines
\put(9,0){\line(1,1){2.5}}
\put(9,0){\line(1,-1){2.5}}
\thinlines
\put(20,-0.15){\line(1,0){20}}
\put(0,1.25){\line(1,0){20}}
\put(0,-1.5){\line(1,0){20}}
\multiput(0,0)(20,0){3}{\circle*{3}}
\end{picture} &
\quad\begin{bmatrix}
2  & -2 & 0\\
-1  & 2  & -1\\
0  & -1  & 2\\
\end{bmatrix}
\\[.3in]
C_3\quad
&
\setlength{\unitlength}{1.5pt}
\begin{picture}(40,17)(0,-2)
\thicklines
\put(11,0){\line(-1,1){2.5}}
\put(11,0){\line(-1,-1){2.5}}
\thinlines
\put(20,-0.15){\line(1,0){20}}
\put(0,1.25){\line(1,0){20}}
\put(0,-1.5){\line(1,0){20}}
\multiput(0,0)(20,0){3}{\circle*{3}}
\end{picture}&
\quad\begin{bmatrix}
2  & -1 & 0\\
-2  & 2  & -1\\
0  & -1  & 2\\
\end{bmatrix}
\end{array}
\]
\vspace{-.1in}
\caption{Dynkin diagrams and Cartan matrices of types $B_3$ and $C_3$.}
\label{fig:B3-C3}
\end{figure}

\begin{remark}
It is important to stress that the meaning of double 
and triple arrows in a Dynkin diagram is very different from the meaning
of multiple arrows in a quiver. 
A double arrow 
\[
\,1\,\,
\raisebox{-0.7mm}{\includegraphics[width=10mm]{B2dynk.ps}}
\,\,2
\]
in a Dynkin diagram corresponds to the submatrix 
\[
\begin{bmatrix}
2  & -1\\
-2  & 2 \\
\end{bmatrix}
\]
of the associated Cartan matrix.
Meanwhile, a double arrow 
\[
1\begin{array}{c}
\longrightarrow\\[-.12in]
\longrightarrow
\end{array}
2
\] 
in a quiver
corresponds to the submatrix
\[
\begin{bmatrix}
0  & 2 \\
-2  & 0 \\
\end{bmatrix}
\]
of the associated exchange matrix. 
\end{remark}

A Cartan matrix is called \emph{indecomposable} if its 
Dynkin diagram is connected.
By a simultaneous permutation of rows and columns, any Cartan matrix~$A$ can
be transformed  into a block-diagonal matrix with indecomposable blocks.
This corresponds to decomposing the Dynkin diagram of $A$ into
connected components.
The \emph{type} of~$A$ (i.e., its
equivalence class with respect to simultaneous permutations of rows
and columns) 
is determined by specifying the multiplicity of each type of connected
Dynkin diagram in this decomposition.
Thus to classify Cartan matrices, one needs to produce
the list of all connected Dynkin diagrams.
The celebrated \emph{Cartan-Killing classification} 
asserts that this list is given as follows.

\begin{theorem}
\label{th:Cartan-Killing-list}
Figure~\ref{diagrams} gives a complete list of connected Dynkin
diagrams corresponding to indecomposable Cartan matrices of finite type.
\end{theorem}

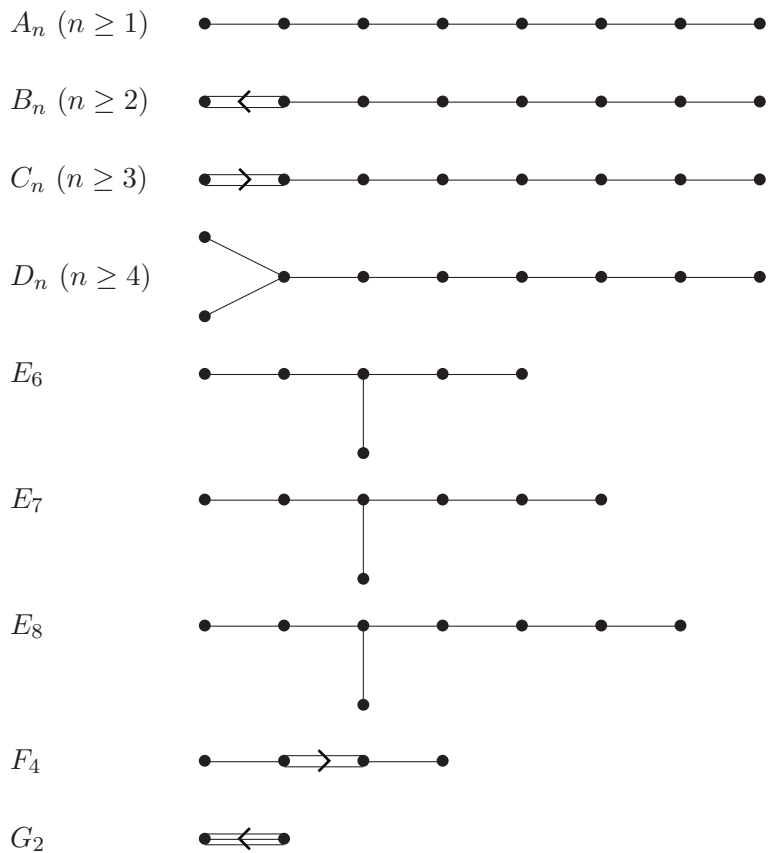
\begin{figure}[ht]
\hrule

\vspace{-.2in}
\[
\begin{array}{lcl}
A_n\  (n\geq 1) &&
\setlength{\unitlength}{1.5pt}
\begin{picture}(140,17)(0,-2)
\put(0,-0.15){\line(1,0){140}}
\multiput(0,0)(20,0){8}{\circle*{3}}
\end{picture}\\
B_n\ (n\geq 2)
&&
\setlength{\unitlength}{1.5pt}
\begin{picture}(140,17)(0,-2)
\thicklines
\put(9,0){\line(1,1){2.5}}
\put(9,0){\line(1,-1){2.5}}
\thinlines
\put(20,-0.15){\line(1,0){120}}
\put(0,1.25){\line(1,0){20}}
\put(0,-1.5){\line(1,0){20}}
\multiput(0,0)(20,0){8}{\circle*{3}}
\end{picture} \\
C_n\ (n\geq 3)
&&
\setlength{\unitlength}{1.5pt}
\begin{picture}(140,17)(0,-2)
\thicklines
\put(11,0){\line(-1,1){2.5}}
\put(11,0){\line(-1,-1){2.5}}
\thinlines
\put(20,-0.15){\line(1,0){120}}
\put(0,1.25){\line(1,0){20}}
\put(0,-1.5){\line(1,0){20}}
\multiput(0,0)(20,0){8}{\circle*{3}}
\end{picture}
\\[.1in]
D_n\ (n\geq 4)
&&
\setlength{\unitlength}{1.5pt}
\begin{picture}(140,17)(0,-2)
\put(20,-0.15){\line(1,0){120}}
\put(0,10){\line(2,-1){20}}
\put(0,-10){\line(2,1){20}}
\multiput(20,0)(20,0){7}{\circle*{3}}
\put(0,10){\circle*{3}}
\put(0,-10){\circle*{3}}
\end{picture}
\\[.1in]
E_6
&&
\setlength{\unitlength}{1.5pt}
\begin{picture}(140,17)(0,-2)
\put(0,-0.15){\line(1,0){80}}
\put(40,0){\line(0,-1){20}}
\put(40,-20){\circle*{3}}
\multiput(0,0)(20,0){5}{\circle*{3}}
\end{picture}
\\[.25in]
E_7
&&
\setlength{\unitlength}{1.5pt}
\begin{picture}(140,17)(0,-2)
\put(0,-0.15){\line(1,0){100}}
\put(40,0){\line(0,-1){20}}
\put(40,-20){\circle*{3}}
\multiput(0,0)(20,0){6}{\circle*{3}}
\end{picture}
\\[.25in]
E_8
&&
\setlength{\unitlength}{1.5pt}
\begin{picture}(140,17)(0,-2)
\put(0,-0.15){\line(1,0){120}}
\put(40,0){\line(0,-1){20}}
\put(40,-20){\circle*{3}}
\multiput(0,0)(20,0){7}{\circle*{3}}
\end{picture}
\\[.3in]
F_4
&&
\setlength{\unitlength}{1.5pt}
\begin{picture}(140,17)(0,-2)
\thicklines
\put(31,0){\line(-1,1){2.5}}
\put(31,0){\line(-1,-1){2.5}}
\thinlines
\put(20,1.25){\line(1,0){20}}
\put(20,-1.5){\line(1,0){20}}
\put(0,-0.25){\line(1,0){20}}
\put(40,-0.25){\line(1,0){20}}
\multiput(0,0)(20,0){4}{\circle*{3}}
\end{picture}
\\[-.0in]
G_2
&&
\setlength{\unitlength}{1.5pt}
\begin{picture}(140,17)(0,-2)
\thicklines
\put(9,0){\line(1,1){2.5}}
\put(9,0){\line(1,-1){2.5}}
\thinlines
\put(0,1.25){\line(1,0){20}}
\put(0,-1.5){\line(1,0){20}}
\put(0,-0.15){\line(1,0){20}}
\multiput(0,0)(20,0){2}{\circle*{3}}
\end{picture}
\end{array}
\]
\vspace{-.1in}
\caption{Dynkin diagrams of indecomposable Cartan matrices.  The 
subscript $n$ indicates the number of nodes in the diagram.}
\label{diagrams}
\end{figure}

We do not include the proof of Theorem~\ref{th:Cartan-Killing-list}
in this book, nor do we rely on this theorem anywhere in our proofs. 
%
We will, however, make extensive use of the standard nomenclature of
Dynkin diagrams 
presented in Figure~\ref{diagrams}. 
The notation $X_n \sqcup Y_{n'}$ will denote
 the disjoint union of two Dynkin diagrams
$X_n$ and $Y_{n'}$.

\pagebreak[3]

The relationship between Cartan matrices 
and skew-symmetrizable matrices 
is based on the following definition~\cite{ca2}.

\begin{definition}
\label{def:assoc-cartan}
Let $B=(b_{ij})$  be a skew-symmetrizable integer matrix.
Its \emph{Cartan counterpart} is the symmetrizable generalized Cartan
matrix 
\begin{equation}
\label{eq:CartanCounterpart}
A = A(B)=(a_{ij})
\end{equation}
of the same size, defined by
\begin{equation}
\label{eq:assoc-cartan}
a_{ij} =
\begin{cases}
2 & \text{if $i=j$;} \\ 
- |b_{ij}| & \text{if $i\neq j$.}
\end{cases}
\end{equation}
\end{definition}

The main result of this chapter is the following classification
of cluster algebras of finite type~\cite{ca2}. 

\begin{theorem}
\label{th:finite-type-classification}
A seed pattern (or the corresponding cluster algebra) is of finite
type if and only if 
it contains an exchange matrix $B$ 
	whose Cartan counterpart $A(B)$
	(see \cref{def:assoc-cartan})
is a Cartan matrix of finite type. 
\end{theorem}

The proof of \cref{th:finite-type-classification} 
spans Sections~\ref{sec:type-A}--\ref{sec:2-finite}. 
An overview of the proof is given below in this section. 

One important feature of Theorem~\ref{th:finite-type-classification}
is that the finite type property depends solely on the exchange
matrix~$B$ but not on the coefficient tuple~$\yy$. 
In other words, the top $n\times n$ submatrix~$B$ of an extended exchange
matrix~$\tilde B$ determines whether the seed pattern at hand has
finitely many seeds.
The bottom part of~$\tilde B$ has no effect on this property. 

\begin{definition}\label{def:type}
Let $X_n$ be a Dynkin diagram on $n$ vertices.  
A seed pattern of rank~$n$ 
(or the corresponding cluster algebra) 
	is said to be \emph{of (Cartan-Killing) type~$X_n$} if one of its exchange matrices~$B$ 
	has Cartan counterpart of  type~$X_n$. 
\end{definition}

\begin{example}
The matrices 
$$
\begin{bmatrix}
0 & 0\\
0 & 0
\end{bmatrix},
\begin{bmatrix}
0 & 1\\
-1 & 0
\end{bmatrix},
\begin{bmatrix}
0 & 1\\
-2 & 0
\end{bmatrix},
\begin{bmatrix}
0 & 1\\
-3 & 0
\end{bmatrix}$$
define seed patterns (and cluster algebras) of types 
$A_1 \sqcup A_1$, $A_2$, $B_2$, and~$G_2$, respectively.
\end{example}

\begin{remark}
\label{rem:simplylaced}
Suppose $X_n$ is simply laced, 
i.e.,\ is one of the types $A_n$, $D_n$, $E_6$, $E_7$, $E_8$.  Then 
a seed pattern is of type $X_n$ if
 one of its exchange matrices~$B$ 
	corresponds to a quiver 
	that is an orientation of a
Dynkin diagram of type~$X_n$.
We note that by Exercise~\ref{ex:orientations-of-a-tree}, all
orientations of a tree are mutation equivalent to each other, so if one of them is
present in the pattern, then all of them are.
\end{remark}

A priori, 
\cref{def:type} allows for the possibility that 
a given seed pattern is simultaneously 
of two different types.  
However, 
the following companion  to 
\cref{th:finite-type-classification}
shows that 
this cannot happen.



\begin{theorem}
\label{thm:type}
Let $B'$ and $B''$ be skew-symmetrizable square matrices whose Cartan
counterparts $A(B')$ and $A(B'')$ are Cartan matrices of
finite type. Then the following are equivalent: 
\begin{enumerate}[leftmargin=.3in]
\item[{\rm (1)}]
the Cartan matrices $A(B')$ and $A(B'')$ have the same type;
\item[{\rm (2)}]
$B'$ and $B''$ are mutation equivalent.
\end{enumerate}
\end{theorem}

This theorem will be proved in \cref{sec:type}.

By~\cref{th:finite-type-classification}, 
a seed pattern of finite type must contain exchange matrices whose 
Cartan counterparts are Cartan matrices of finite type. \linebreak[3]
By \cref{thm:type},
all these matrices have the same type.
Consequently 
the (Cartan-Killing) {type} of
a seed pattern (resp., cluster algebra) of finite type
is unambiguously defined.

\begin{remark}
\label{rem:clustersubalgebra}
Mutation classes of finite type can be partially ordered via embeddability
(cf.\ Definition~\ref{def:quiver-embedding}), which in turn can be
interpreted in the language of cluster subalgebras (cf.\ \cref{def:clustersubalgebra}).
For example, the inclusion of Dynkin diagrams
$A_n \subset D_{n+1}$
(see Figure \ref{diagrams}) yields an embedding
$\mathbf{A_n} \leq \mathbf{D_{n+1}}$ of mutation classes of type
$A_n$ and type $D_{n+1}$ quivers; consequently
there is a cluster subalgebra of type~$A_n$ within each
cluster algebra of type~$D_{n+1}$.
Similarly, we have embeddings
$\mathbf{D_n} \leq  \mathbf{E_{n+1}}$ for $5 \leq n \leq 7$, and
$\mathbf{E_6} \leq \mathbf{E_7} \leq \mathbf{E_8}$,
cf.\ Example~\ref{example:dynkin-embeddings} and
Figure~\ref{fig:E6E8}.

Going beyond the quiver case, we have the embeddings
$\mathbf{A_n} \leq \mathbf{B_{n+1}}$, $\mathbf{A_n} \leq 
\mathbf{C_{n+1}}$,
$\mathbf{B_3} \leq \mathbf{F_4}$, $\mathbf{C_3} \leq \mathbf{F_4}$,
and the corresponding inclusions for cluster algebras of finite type.
\end{remark}

\medskip







We conclude this section by an overview of the remainder of
\cref{ch:finitetype}.

Sections~\ref{sec:type-A}--\ref{sec:decomposable-types}
are dedicated to showing  that any seed pattern 
that has an exchange matrix whose
Cartan counterpart is of one of the types $A_n, B_n,\dots, G_2$ 
has finitely many seeds. 
This is done case by case. 
The idea is to explicitly construct, for each (indecomposable) type, 
a particular seed pattern whose exchange matrices
have full $\ZZ$-rank, and show that this pattern has finitely many
seeds. \linebreak[3]
Then an argument based on 
Corollary~\ref{cor:z-span} and Remark~\ref{rem:full-z-rank}
will imply the same for \emph{all} cluster algebras of the corresponding
type. 

In Sections~\ref{sec:type-A}--\ref{sec:type-D}
we exhibit seed patterns of types~$A_n$ and~$D_n$ 
possessing the requisite properties. 
In type~$A_n$, we utilize the construction involving the homogeneous
coordinate ring of the Grassmannian~$\Gr_{2,n+3}$, 
cf.\ Section~\ref{sec:Ptolemy}. 
Type $D_n$ is treated using a similar construction, admittedly much more technical
than in the type~$A_n$ case. 
In \cref{sec:BC-finite}, 
we handle the types $B_n$ and $C_n$ using the technique of folding 
introduced in Section~\ref{sec:folding}.

The exceptional types are treated in a different way.
In \cref{sec:exceptional}, we use a computer check to 
verify that the cluster algebras of type 
$E_8$ are of finite type.
This implies the same for the types $E_6$ and~$E_7$. 
Types~$F_4$ and~$G_2$ are then handled in \cref{sec:exceptional2}, 
via folding of~$E_6$ and~$D_4$, respectively. 

In \cref{sec:decomposable-types}, we discuss decomposable types,
and complete the first part of the proof 
of \cref{th:finite-type-classification}.

We note that 
  Sections~\ref{sec:type-A}--\ref{sec:exceptional} contain many
  incidental examples of cluster 
  algebras of finite type, in addition to those used in the proof of
  the classification theorem. 

\cref{thm:type} is proved in \cref{sec:type}.  
Also in this section, we complete and summarize 
the enumeration of the seeds (equivalently, clusters) and 
the cluster variables for all finite types. 

The proof  of \cref{th:finite-type-classification} is completed in
\cref{sec:2-finite},
by demonstrating that any seed pattern of finite type comes from
a Cartan matrix of finite type.  
This is done by exploiting the fact that 
every exchange matrix appearing in such a seed pattern is $2$-finite,
see \cref{cor:finite-type-2finite}. 


\cref{th:finite-type-classification} provides a
characterization of seed patterns of finite type 
which, while conceptually satisfying, 
is not particularly useful in practice.   
In Section~\ref{sec:quasi-cartan}, we  discuss an alternative
criterion for recognizing whether a seed pattern is of finite type. 
This criterion is formulated directly in terms of the given exchange matrix 
$B$ (as opposed to its mutation class).

\section{Seed patterns of type~$A_n$}
\label{sec:type-A}


The main result of this section is 
Theorem~\ref{th:type-A-is-finite}, asserting that seed patterns of type~$A_n$
are of finite type.  The proof uses the fact that 
seed patterns of type~$A_n$ are governed by the combinatorics
of triangulated polygons.
A~case in point is 
 the cluster algebra structure in 
the Pl\"ucker ring~$R_{2,n+3}$
(the homogeneous coordinate ring of~$\Gr_{2,n+3}$)  
discussed in Section~\ref{sec:Ptolemy}. 
In~this example, we verify that a cluster variable indexed~by 
a diagonal in a triangulation~$T$ only depends on the diagonal 
and not on the choice of~$T$, or a sequence of mutation steps relating $T$ to an initial triangulation.
After verifying that the cluster structure in~$R_{2,n+3}$ is 
of finite type, we check that one of its exchange matrices
has full $\ZZ$-rank, and the general case follows. 


Let $T$ be a triangulation of a convex $(n+3)$-gon~$\mathbf{P}_{n+3}$
by $n$ noncrossing diagonals labeled $1,\dots,n$.
We define the $n\times n$ matrix $B(T)=(b_{ij}(T))$ by 
\begin{equation}
\label{eq:B(T)}
b_{ij}(T)=
\begin{cases}
1 & \!\!\text{if $i$ and $j$ label two sides of a triangle
  in~$T$,}\\
& \text{\quad with $j$ following $i$ in the clockwise order;}\\
-1 & \!\!\text{if $i$ and $j$ label two sides of a triangle
  in~$T$,}\\
& \text{\quad with $i$ following $j$ in the clockwise order;}\\
0 & \!\!\text{if $i$ and $j$ do not belong to the same triangle
  in~$T$.}
\end{cases}
\end{equation}
The skew-symmetric matrix $B(T)$ corresponds to the mutable part of the quiver
$Q(T)$ described in Definition~\ref{def:Q(T)-polygon} and Figure~\ref{fig:quiver-triangulation}.


The following easy lemma provides an alternative definition of the
notion of a seed pattern of type~$A_n$.  

\begin{lemma}\label{lem:A-triangulation}
A seed pattern has type~$A_n$ if and only if one 
(equivalent\-ly, any) of its exchange matrices 
can be identified with the exchange matrix $B(T)$
corresponding to a triangulation~$T$ of a convex
$(n+3)$-gon, 
cf.~\eqref{eq:B(T)}. 
\end{lemma}

\begin{figure}[ht] 
\begin{center} 
\setlength{\unitlength}{2pt} 
\begin{picture}(60,60)(0,0) 
\thicklines 
  \multiput(0,20)(60,0){2}{\line(0,1){20}} 
  \multiput(20,0)(0,60){2}{\line(1,0){20}} 
  \multiput(0,40)(40,-40){2}{\line(1,1){20}} 
  \multiput(20,0)(40,40){2}{\line(-1,1){20}} 
 
  \multiput(20,0)(20,0){2}{\circle*{1}} 
  \multiput(20,60)(20,0){2}{\circle*{1}} 
  \multiput(0,20)(0,20){2}{\circle*{1}} 
  \multiput(60,20)(0,20){2}{\circle*{1}} 
 
\put(20,50){\circle{3}}
\put(20,40){\circle{3}}
\put(30,30){\circle{3}}
\put(40,30){\circle{3}}
\put(50,40){\circle{3}}

\thinlines 
\put(40,60){\line(-2,-1){40}} 
\put(40,60){\line(-1,-1){40}} 
\put(40,60){\line(-1,-3){20}} 
\put(40,60){\line(0,-1){60}} 
\put(40,60){\line(1,-2){20}} 
 
 
\put(40,-3){\makebox(0,0){$4$}} 
\put(63,20){\makebox(0,0){$3$}} 
\put(63,40){\makebox(0,0){$2$}} 
\put(40,63){\makebox(0,0){$1$}} 
\put(16,63){\makebox(0,0){$n+3$}} 
\put(-7,40){\makebox(0,0){$n+2$}} 
\put(-7,20){\makebox(0,0){$n+1$}} 
 
 \linethickness{1.2pt}

\put(20,48){{\vector(0,-1){6}}}
\put(22,38){{\vector(1,-1){6.5}}}
\put(42,32){{\vector(1,1){6.5}}}
\put(32,30){{\vector(1,0){6}}}

\end{picture} 
\end{center} 
\caption{
A triangulation~$T_\circ$ of the
polygon~$\mathbf{P}_{n+3}$ (here $n=5$). 
The mutable part of the quiver $Q(T_\circ)$ (see
Definition~\ref{def:Q(T)-polygon}) 
is the \emph{equioriented Dynkin quiver} of type~$A_n$.} 
\label{fig:octagon-special} 
\end{figure}
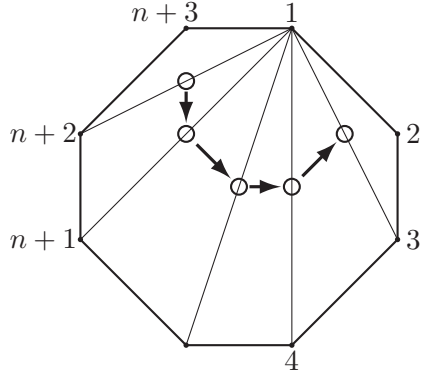 

\pagebreak[3]

\begin{theorem}
\label{th:type-A-is-finite}
Seed patterns of type~$A_n$ are of finite type. 
\end{theorem}

\begin{remark}
\label{rem:finite-type-coeffs}
Definition~\ref{def:type} imposes no
restrictions on the bottom part of the extended exchange
matrices, nor on the number of frozen variables. 
In light of Lemma 
\ref{lem:A-triangulation}, 
Theorem~\ref{th:type-A-is-finite} asserts that as long as the mutable
part of a quiver comes from a triangulated polygon,
the total number of seeds generated by the quiver is finite. 
\end{remark}

\begin{proof}
We start by  showing that a particular seed pattern of type~$A_n$, 
the one associated with the Pl\"ucker ring~$R_{2,n+3}$,
is of finite type.  

As in Section~\ref{sec:Ptolemy}, we label the  vertices of the
polygon~$\mathbf{P}_{n+3}$ clockwise by the numbers $1,\dots,n+3$. 
For a triangulation $T$ of~$\mathbf{P}_{n+3}$ as in Definition~\ref{def:Q(T)-polygon},
we use the labels $1,\dots,n$ for the diagonals of~$T$ (in arbitrary fashion), 
and use the labels $n+1,\dots,2n+3$ for the sides
of~$\mathbf{P}_{n+3}$, as follows: 
\begin{itemize}[leftmargin=.2in]
\item
the side with vertices $\ell$ and $\ell+1$ is labeled~$n+\ell$, for $\ell=1,\dots,n+2$;
\item
the side with vertices $1$ and $n+3$ is labeled $2n+3$. 
\end{itemize} 
We define the matrix $\tilde B(T)=(b_{ij}(T))$ 
by the formula~\eqref{eq:B(T)},
this time with $i\!\in\!\{1,\dots,2n+3\}$ and $j\!\in\!\{1,\dots,n\}$. 
Thus $\tilde B(T)$ is the extended exchange matrix
for the quiver $Q(T)$ from
Definition~\ref{def:Q(T)-polygon}. 
By Exercise~\ref{ex:flip-is-mutation}, flips of triangulations
translate into mutations of associated quivers. 

We now reformulate the construction described in Section~\ref{sec:Ptolemy}. 
Let $V=\CC^2$ be a 2-dimensional complex vector space,
and let $\langle\cdot,\cdot\rangle$ denote 
the standard skew-symmetric nondegenerate bilinear form on~$V$. 
Simply~put, $\langle u,v\rangle$ 
is the determinant of the $2\times 2$ matrix with
columns $u,v\in V$. 
Let~$\mathcal{K}$ be the field of rational functions on
$V^{n+3}$ written in terms of $2n+6$ variables, 
the coordinates of $n+3$ vectors $v_1,\dots,v_{n+3}$. 
The special linear group naturally acts on~$V$, hence on~$V^{n+3}$ and
on~$\mathcal{K}$. 
Let $\FFcal=\mathcal{K}^{\SL_2}$ be the sub\-field of $\SL_2$-invariant
rational functions. 
All the action will take place in~$\FFcal$. 

We associate the Pl\"ucker coordinate $P_{ij}=\langle v_i,v_j\rangle\in\FFcal$
with the line segment connecting vertices $i$ and~$j$. 
For each triangulation~$T$ of $\mathbf{P}_{n+3}$, we let $\tilde \xx(T)$ be the collection of
the $2n+3$ Pl\"ucker coordinates $P_{ij}$ associated with the 
sides and diagonals of~$T$, as in Section~\ref{sec:Ptolemy}. 
By Lemma \ref{lem:independent} below, the elements of 
$\tilde \xx(T)$ are algebraically independent.
We view the Pl\"ucker coordinates associated with the sides of
$\mathbf{P}_{n+3}$ as frozen variables. 
As observed earlier, the Grassmann-Pl\"ucker relations~\eqref{eq:grassmann-plucker-3term}
satisfied by the elements $P_{ij}$ 
can be interpreted as exchange relations encoded by the matrices~$\tilde B(T)$.

We set $\Sigma(T)=(\tilde \xx(T),\tilde B(T))$.
Therefore the seeds obtained from an initial seed
$\Sigma(T_\circ)$  
form a seed pattern of type~$A_n$, and 
the initial cluster variables consist of  the
Pl\"ucker coordinates labeling the triangulation $T_\circ$. 
Because the Pl\"ucker coordinates
satisfy the exchange relations, which correspond to flips of triangulations,
this gives a canonical identification of cluster variables and 
clusters with 
diagonals 
and triangulations
of $\mathbf{P}_{n+3}$.
(A priori, a seed and the cluster variables in it could depend not only
on the triangulation they correspond to, but the sequence of mutations
we used to arrive at that triangulation from the initial seed.)
Therefore the number of distinct seeds in this pattern is  finite. 

We complete the proof of Theorem~\ref{th:type-A-is-finite} using an
argument based on
Corollary~\ref{cor:z-span} and Remark~\ref{rem:full-z-rank}. 
All we need to do is to check that for some triangulation~$T_\circ$,
the matrix $\tilde B(T_\circ)$ has full $\ZZ$-rank. 
Taking $T_\circ$ as in Figure~\ref{fig:octagon-special}, we obtain the
matrix  
\[
\tilde B(T_\circ)
=
\left[\,\begin{matrix}
0  & -1 & 0  & \cdots & 0 & 0 \\
1  & 0  & -1 & \cdots & 0 & 0 \\
0  & 1  & 0  & \cdots & 0 & 0 \\[-.05in]
\vdots & \vdots & \vdots & \ddots & \vdots & \vdots \\
0 & 0 & 0 & \cdots & 0 & -1 \\
0 & 0 & 0 & \cdots &1 & 0 \\
\hline\\[-.17in]
-1 & 0 &0 & \cdots & 0 & 0\\
1 & 0 &0 & \cdots & 0 & 0\\[-.05in]
\vdots & \vdots & \vdots & \ddots & \vdots & \vdots 
\end{matrix}\,
\right],
\]
where the line is drawn under the $n$th row.
 The matrix is easily seen to have full~$\ZZ$-rank. 
\end{proof}

\begin{lemma}\label{lem:independent}
For any triangulation~$T$ of the polygon~$\mathbf{P}_{n+3}$,
the elements of $\tilde \xx(T)$ are algebraically independent. 
Thus $(\tilde \xx(T),\tilde B(T))$ is a seed~in~$\FFcal$.  
\end{lemma}

\begin{proof}
One way to establish this is to observe that $\tilde\xx(T)$ generates
the field~$\FFcal$ (since each Pl\"ucker coordinate~$P_{ij}$ is a
rational function in~$\tilde\xx(T)$), 
and combine this with the fact that the transcendence
degree of~$\FFcal$ over~$\CC$ (equivalently, the dimension of the affine cone
over~$\Gr_{2,n+3}$) is $2n+3$. 
\end{proof}

\pagebreak[3]

\begin{remark} 
An alternative proof of Theorem~\ref{th:type-A-is-finite} 
can be based on the description of the fundamental group of
the graph whose vertices are the triangulations of the
polygon~$\mathbf{P}_{n+3}$, and whose edges correspond to the flips. 
(We view this graph as a $1$-dimensional simplicial complex, 
the $1$-skeleton of the $n$-dimensional associahedron.) 
The fundamental group of this graph is generated by $4$-cycles and
$5$-cycles
(the boundaries of $2$-dimensional faces of the associahedron) pinned down to a basepoint.  
For each of these cycles, the corresponding sequence of $4$ or $5$
mutations in a seed pattern of type~$A_n$ brings us back to the
original seed; this follows from the analysis of the type~$A_2$ case in
Section~\ref{sec:finite-type-rank2}. 
Consequently, the seeds in such a pattern can be labeled by the
triangulations of~$\mathbf{P}_{n+3}$, implying the claim of finite
type. 
\end{remark}

\begin{corollary}
\label{cor:An-structure}
Cluster variables in a seed pattern of type~$A_n$ can be
labeled by the diagonals of a convex $(n+3)$-gon~$\mathbf{P}_{n+3}$ so that
\begin{itemize}[leftmargin=.2in]
\item
clusters correspond to triangulations of the polygon~$\mathbf{P}_{n+3}$ by
noncrossing diagonals, 
\item
mutations correspond to flips, and 
\item
exchange matrices are given by~\eqref{eq:B(T)}. 
\end{itemize}
Cluster variables labeled by different diagonals are
distinct, so there are altogether $\frac{n(n+3)}{2}$ cluster variables
and $\frac{1}{n+2}\binom{2n+2}{n+1}$ seeds (and as many clusters). 
\end{corollary}

\begin{proof}
It is well known that the number of triangulations a convex $(n+3)$-gon has is equal to the 
Catalan number $\frac{1}{n+2}\binom{2n+2}{n+1}$, see
e.g., \cite[Exercise~6.19a]{ec2}. 
So the only claim remaining to be proved is that all these cluster
variables are distinct in any seed pattern of type~$A_n$. 
Let $x$ and $x'$ be two cluster variables labeled by distinct
diagonals $d$ and~$d'$. 
If $d$ and~$d'$ do not cross each other, then there is a
cluster containing $x$ and~$x'$, so $x$ and~$x'$ are algebraically
independent and therefore distinct. 
If $d$ and $d'$ do cross, then there is an exchange
relation of the form $x\,x'=M_1+M_2$ where $M_1$ and~$M_2$ are
monomials in the elements of some extended cluster~$\tilde\xx$
containing~$x$.
Now the equality $x=x'$ would imply $x^2=M_1+M_2$, contradicting the
condition that the elements of~$\tilde\xx$ are algebraically
independent. 
\end{proof}

In the rest of this section, we examine several seed
patterns (or cluster algebras) of type~$A_n$ which naturally arise in
various mathematical contexts.

\begin{exercise}
A \emph{frieze pattern}~\cite{conway-coxeter, coxeter-frieze} is a table of the form
\[
\text{$n+2$ rows}\left\{
\begin{array}{ccccccccccccccccc}
\cdots &  1  &&  1  &&  1  &&  1  &&  1  &&  1  &&  1  &&  1 & \cdots \\
\cdots &    & *  &&  *  &&  *  &&  *  &&  *  &&  *  &&  * && \cdots\\
\cdots &  *  && *  &&  *  &&  *  &&  *  &&  *  &&  *  &&  * & \cdots\\
\cdots &    & *  &&  *  &&  *  &&  *  &&  *  &&  *  &&  * && \cdots\\
\cdots &  *  && *  &&  *  &&  *  &&  *  &&  *  &&  *  &&  * & \cdots\\
\cdots &    & 1  &&  1  &&  1  &&  1  &&  1  &&  1  &&  1 && \cdots
\end{array}
\right.
\]
with (say) positive integer entries such that every quadruple 
\[
\begin{array}{ccc}
 &B&\\
A&&C\\
&D
\end{array}
\]
satisfies $AC-BD=1$. 
Identify the entries in a frieze pattern with cluster variables in a
seed pattern of type~$A_n$.
How many distinct entries does a frieze pattern have? 
What kind(s) of periodicity does it possess? 
\end{exercise}

\begin{example}
\label{example:SL4}
Let us discuss, somewhat informally, the example of a seed pattern
associated with the basic affine space for~$\SL_4$. 
Choose the initial seed for this pattern as shown in 
Figure~\ref{fig:special-seed-SL4/U}. 
(This seed has already appeared in Figure~\ref{fig:two-quivers};
it corresponds to a particular choice of a wiring diagram.) 
The variables $P_2$, $P_3$, and~$P_{23}$
are mutable; the remaining six variables are frozen. 
The mutable part of the initial quiver is an oriented $3$-cycle;
as such, it is easily identified as the mutable part of a quiver
associated with a particular triangulation of a hexagon.
Thus we are dealing here with a seed of type~$A_3$. 

\begin{figure}[ht]
\begin{center}
\setlength{\unitlength}{2.5pt} 
\begin{picture}(60,20)(0,10) 

\put( 0,10){\makebox(0,0){$P_{1}$}}
\put(20,10){\makebox(0,0){$P_{2}$}}
\put(40,10){\makebox(0,0){$P_{3}$}}
\put(60,10){\makebox(0,0){$P_{4}$}}

\put(10,20){\makebox(0,0){$P_{12}$}}
\put(30,20){\makebox(0,0){$P_{23}$}}
\put(50,20){\makebox(0,0){$P_{34}$}}

\put(20,30){\makebox(0,0){$P_{123}$}}
\put(42,30){\makebox(0,0){$P_{234}$}}

\thicklines 

\put(16,10){\vector(-1,0){12}}
\put(36,10){\vector(-1,0){12}}
\put(56,10){\vector(-1,0){12}}

\put(26,20){\vector(-1,0){12}}
\put(46,20){\vector(-1,0){12}}

\put(22,12){\vector(1,1){6}}
\put(42,12){\vector(1,1){6}}
\put(32,22){\vector(1,1){6}}

\put(12,17){\vector(1,-1){5}}
\put(32,17){\vector(1,-1){5}}
\put(22,27){\vector(1,-1){5}}

\end{picture}
\vspace{-.1in}
\end{center}
\caption{A seed of flag minors in $\CC[\SL_4]^U$.
}
\label{fig:special-seed-SL4/U}
\end{figure}

\noindent
A matrix in $\SL_4$ has $2^4-2=14$ nontrivial flag minors:
the $6$ frozen variables
\[
P_1, \, P_{12}, \, P_{123}, \, P_4, \, P_{34}, \, P_{234}
\]
(recall that they correspond to the unbounded chambers in a wiring
diagram)
and $8$ additional flag minors 
\[
P_2, \, P_3, \, P_{13}, \, P_{14}, \, P_{23},\, P_{24}, \, P_{124}, \,
P_{134}, 
\]
all of which can be obtained by the mutation process from our initial
seed.
Note however that a seed pattern of type~$A_3$ should have $9$ cluster
variables---so one of them is still missing! 

Examining the initial seed shown in
Figure~\ref{fig:special-seed-SL4/U},
we see that it can be mutated in three possible ways. 
The corresponding exchange relations are:
\begin{align*}
P_2 \,P_{13}&={P_{12}}\,P_3+{P_1}\,P_{23}\,,\\
P_3\,P_{24}&=
{P_{4}}\,P_{23}+{P_{34}}\,P_{2}\,,\\
P_{23}\,\Omega &=
P_{123}\,P_{34}\,P_2+P_{12}\,P_{234}\,P_3\,. 
\end{align*}
The first two relations correspond to
the two braid moves that can be applied to the initial wiring diagram~$D_\circ$. 
The third relation is different in nature: 
it produces a seed that does not correspond to any wiring diagram,  
as the new cluster variable~$\Omega$ is not a
flag minor. 
Ostensibly, $\Omega$~is a rational expression (indeed, a Laurent polynomial)
in the flag minors. 
One can check that in~fact
\begin{equation}
\label{eq:Omega}
\Omega=\dfrac{P_{123}\,P_{34}\,P_2+P_{12}\,P_{234}\,P_3}{P_{23}}
=-P_1P_{234}+P_2P_{134}
\end{equation}
---so $\Omega$ is not merely a rational function on
the basic affine space but a \emph{regular} function.
It follows that the corresponding cluster algebra is precisely the
invariant ring $\CC[\SL_4]^U$ (recall that the latter is generated by the
flag minors). 
It turns out that this phenomenon holds for any special linear group~$\SL_k$, 
resulting in a cluster algebra structure in
$\CC[\operatorname{SL}_k]^U$. 

In the case under consideration, we get 14 
distinct extended clusters, in agreement with Corollary~\ref{cor:An-structure}. 
See Figure~\ref{fig:sl_4/U}.  
\end{example}

\pagebreak[3]

\begin{figure}[ht]
\begin{center}
\setlength{\unitlength}{0.75pt} 
\begin{picture}(200,195)(0,-15) 
\thinlines 

\put(10,10){\line(-1,-1){10}}
\put(40,80){\line(1,2){20}}
\put(40,80){\line(1,-1){30}}
\put(70,50){\line(3,2){30}}
\put(70,50){\line(1,-4){10}}
\put(80,10){\line(1,0){40}}
\put(60,120){\line(8,-3){40}}
\put(100,70){\line(0,1){35}}
\put(100,70){\line(3,-2){30}}
\put(120,10){\line(1,4){10}}
\put(130,50){\line(1,1){30}}
\put(100,105){\line(8,3){40}}
\put(100,175){\line(0,1){15}}
\put(140,120){\line(1,-2){20}}
\put(190,10){\line(1,-1){10}}

\qbezier(10,10)(10,50)(40,80)
\qbezier(10,10)(40,-5)(80,10)
\qbezier(60,120)(65,155)(100,175)

\qbezier(190,10)(190,50)(160,80)
\qbezier(190,10)(160,-5)(120,10)
\qbezier(140,120)(135,155)(100,175)

\put(10,10){\circle*{4}} 
\put(40,80){\circle*{4}} 
\put(60,120){\circle*{4}} 
\put(70,50){\circle*{4}} 
\put(80,10){\circle*{4}} 
\put(100,70){\circle*{4}} 
\put(100,105){\circle*{4}} 
\put(100,175){\circle*{4}} 
\put(120,10){\circle*{4}} 
\put(130,50){\circle*{4}} 
\put(140,120){\circle*{4}} 
\put(160,80){\circle*{4}} 
\put(190,10){\circle*{4}} 

\put(10,120){\makebox(0,0){$P_{124}$}}
\put(190,120){\makebox(0,0){$P_{134}$}}
\put(100,140){\makebox(0,0){$\Omega$}}
\put(100,40){\makebox(0,0){$P_{23}$}}
\put(100,-17){\makebox(0,0){$P_{14}$}}
\put(73,85){\makebox(0,0){$P_{2}$}}
\put(127,85){\makebox(0,0){$P_{3}$}}
\put(40,30){\makebox(0,0){$P_{24}$}}
\put(160,30){\makebox(0,0){$P_{13}$}}


\end{picture}

\end{center}
\caption{Clusters in $\CC[\operatorname{SL}_4]^U$.
The 14 clusters for this seed pattern are shown as
the vertices of a graph; the edges of the graph correspond to seed mutations.  
Note that there is one additional vertex at infinity,
so the graph should be viewed as drawn on a sphere rather than a
plane. 
The regions are labeled by cluster variables. 
Each cluster consists of three elements labeling the regions
adjacent to the corresponding vertex. 
The 6 frozen variables are not~shown. 
This graph is isomorphic to the $1$-skeleton of the three-dimensional
associahedron, shown in Figure~\ref{fig:A3assoc_poly}.}
\label{fig:sl_4/U}
\end{figure}
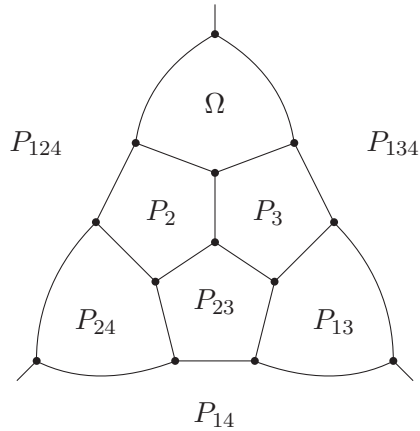

\begin{example}
\label{ex:tridiagonal}
We conclude this section by presenting a family of
seed patterns of type~$A_n$ introduced in~\cite{yangzel}.
They correspond to cluster structures in particular \emph{double Bruhat cells} for the
special linear groups~$\SL_{n+1}(\CC)$,
more specifically in the cells associated with
\emph{pairs of Coxeter elements} in the associated symmetric
group~$\mathcal{S}_{n+1}$.
This construction can be extended to arbitrary simply connected
semisimple complex
Lie groups, see~\cite{yangzel}. 

Let $L_n\subset \SL_{n+1}(\CC)$ be the subvariety of tridiagonal
matrices 
\begin{equation}
\label{eq:tridiagonal}
z=\begin{bmatrix}v_1 & q_1 & 0 & \cdots & 0 \\
                  1  &  v_2 & q_2 & \ddots & \vdots \\
                  0  & 1 & v_3 & \ddots  & 0 \\
                  \vdots & \ddots & \ddots & \ddots  & q_n\\
                  0 &  \cdots &  0 & 1 & v_{n+1}
                  \end{bmatrix} 
\end{equation}
of determinant~1. 
For $i,j\in\{1,\dots,n+3\}$ satisfying $i+2\le j$, 
consider the \emph{solid principal minor} 
(cf.\ Exercise~\ref{exercise:solid-minors})
\[
U_{ij} 
=\Delta_{[i,j-2],[i,j-2]} \in \CC[L_n], 
\]
the determinant of the
submatrix with rows and columns $i, i+1, \dots, j-2$.
For example, $U_{i,i+2}\!=\! v_i$ and $U_{1,n+3} \!=\! \det(z)\! =\! 1$.
By convention, $U_{i,i+1}\!=\!1$.

\begin{exercise}
Prove that these functions satisfy the relations
\begin{equation}
\label{eq:exchange-A-principal-u}
U_{ik}\,U_{j\ell} 
= q_{j-1} q_j \cdots q_{k-2}\, U_{ij}\,U_{k\ell} + U_{i\ell} \,U_{jk}\,, 
\end{equation}
for $1 \leq i<j<k<l\le n+3$.
Then show that these relations are the exchange relations in a
particular seed
pattern of type~$A_n$. 
More precisely, show that there is a seed pattern of type~$A_n$,
with the frozen variables $q_1, \dots, q_n$, 
in which the cluster variables~$U_{ij}$ associated to the diagonals $[i,j]$ of
the convex polygon~$\mathbf{P}_{n+3}$ 
satisfy the exchange relations~\eqref{eq:exchange-A-principal-u}. 
\end{exercise}

What is required to show that 
the relations \eqref{eq:exchange-A-principal-u} are the exchange
relations in a seed pattern of type~$A_n$?
Let us extend each $n\times n$ exchange matrix
$B(T)$ associated to a triangulation~$T$ of~$\mathbf{P}_{n+3}$
(see~\eqref{eq:B(T)}) to the $2n\times n$ matrix~$\tilde B(T)$ 
whose columns 
encode the relations~\eqref{eq:exchange-A-principal-u} corresponding
to the $n$ possible flips from~$T$. 
One then needs to verify that whenever 
triangulations $T$ and $T'$ are related by a flip, 
the associated matrices $\tilde B(T)$ and $\tilde B(T')$ are
related by the corresponding mutation. 


One of the clusters in this seed pattern consists of the $n$ 
leading principal minors
$U_{13}\,, \dots, U_{1,n+2}\,$.
The exchange relations from this cluster 
are the relations \eqref{eq:exchange-A-principal-u} with 
$(i,j,k,\ell)=(1,k-1,k,k+1)$. 
They can be rewritten as follows: 
\begin{equation}
\label{eq:exchange-A-initial}
U_{1,k+1}  = v_{k-1} \,U_{1,k}  - q_{k-2} \,U_{1,k-1}  \quad (k = 3,
\dots, n+2).
\end{equation}
These relations play important roles in the classical theory
of orthogonal polynomials in one variable~\cite{szego},
in the study of a generalized Toda lattice~\cite{kostant-toda},
and in other mathematical contexts. 
\end{example}

\section{Seed patterns of type $D_n$}
\label{sec:type-D}

In this section we show that seed patterns of type $D_n$
are of finite type.  
The proof of this theorem, while substantially more technical than the
proof of Theorem~\ref{th:type-A-is-finite}, 
follows the same general plan.
It relies on two main ingredients.
The first ingredient is a combinatorial construction 
(``tagged arcs'' in a punctured disk) 
that enables us to explicitly describe the
combinatorics of mutations in type~$D_n$ and introduce the relevant
nomenclature.
The second ingredient is an algebraic construction of a particular
seed pattern of type~$D_n$. 
In this pattern, each cluster variable
associated with a tagged arc has an intrinsic definition independent
of the mutation path; this
will imply that the number of seeds is finite.
The ``full $\ZZ$-rank'' argument will then allow us to generalize the
finiteness statement to arbitrary coefficients.

While type $D_n$ Dynkin diagrams are usually defined
for $n \geq 4$, in this section we will allow
for the possibility of $n=3$
(in which case one recovers type~$A_3$). 

\begin{exercise}
Show that a seed pattern is of type~$D_n$ if and only if 
one of its exchange matrices 
corresponds to a quiver which is an oriented $n$-cycle. 
\end{exercise}

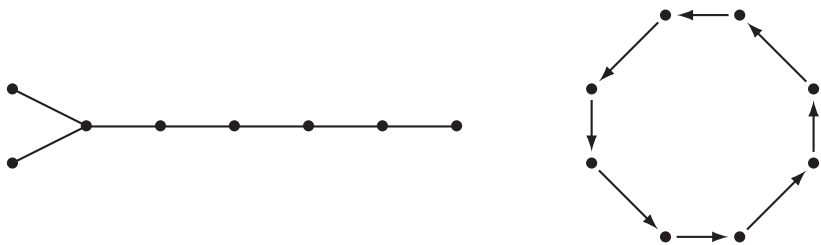
\begin{figure}[ht]
\setlength{\unitlength}{1.4pt}
\begin{picture}(120,17)(0,-30)
\thicklines 
\put(20,-0.15){\line(1,0){100}}
\put(0,10){\line(2,-1){20}}
\put(0,-10){\line(2,1){20}}
\multiput(20,0)(20,0){6}{\circle*{3}}
\put(0,10){\circle*{3}}
\put(0,-10){\circle*{3}}
\end{picture}
\qquad\qquad
\setlength{\unitlength}{1.4pt} 
\begin{picture}(60,60)(0,0) 
\thicklines 
  \put(23,0){\vector(1,0){14}} 
  \put(37,60){\vector(-1,0){14}} 
  \put(60,23){\vector(0,1){14}} 
  \put(0,37){\vector(0,-1){14}} 

  \put(42,2){\vector(1,1){16}} 
  \put(18,58){\vector(-1,-1){16}} 
  \put(2,18){\vector(1,-1){16}} 
  \put(58,42){\vector(-1,1){16}} 

  \multiput(20,0)(20,0){2}{\circle*{3}} 
  \multiput(20,60)(20,0){2}{\circle*{3}} 
  \multiput(0,20)(0,20){2}{\circle*{3}} 
  \multiput(60,20)(0,20){2}{\circle*{3}} 
\end{picture} 

\caption{Dynkin diagram of type~$D_n$ and an oriented $n$-cycle.}
\label{fig:Dn}
\end{figure}

\pagebreak[3]

\begin{theorem}
\label{th:type-D-is-finite}
Seed patterns of type~$D_n$ are of finite type. 
\end{theorem}

As in 
Remark~\ref{rem:finite-type-coeffs}, 
the key point of Theorem~\ref{th:type-D-is-finite} is that 
a pattern of type~$D_n$ has finitely many seeds regardless of the number
of frozen variables and of the entries in the bottom parts of extended
exchange matrices.

\begin{definition} 
\label{def:arcs-Dn}
Let $\mathbf{P}_n^\bullet$ be a convex $n$-gon ($n\ge 3$) with a
distinguished point~$p$ (a~\emph{puncture}) in its interior. 
We label the vertices of $\mathbf{P}_n^\bullet$ clockwise from $1$ to~$n$. 
These vertices and the puncture~$p$ are
the \emph{marked points} of~$\mathbf{P}_n^\bullet$. 

\pagebreak[3]

An~\emph{arc} in $\mathbf{P}_n^\bullet$
is a non-selfintersecting curve~$\gamma$ in~$\mathbf{P}_n^\bullet$
such that
\begin{itemize}[leftmargin=.2in]
\item
the endpoints of $\gamma$ are two different marked points; 
\item
except for its endpoints, 
$\gamma$ is disjoint from the boundary of $\mathbf{P}_n^\bullet$ and
  from the puncture~$p$; 
\item
$\gamma$ does not cut out an unpunctured digon. 
\end{itemize}
We consider each arc up to isotopy within the class of such curves. 

The arcs incident to the puncture are called \emph{radii}. 
\end{definition} 

\begin{definition} 
\label{def:tagged-arcs-Dn}
A \emph{tagged arc} in $\mathbf{P}_n^\bullet$ is either an ordinary
non-radius arc, or a radius that has been labeled (``tagged'') 
in one of two possible ways, \emph{plain} or \emph{notched}. 
Two tagged arcs are called \emph{compatible} with one
  another if their  untagged versions do not cross each other,
with the following modification: 
the plain and notched versions of the same radius are
compatible, but the the plain and notched versions of two different
radii are not.

A \emph{tagged triangulation} is a maximal (by inclusion) collection of
pairwise compatible tagged arcs. 
See Figure~\ref{fig:tagged-arc-d3}.
\end{definition}

\vspace{-.1in}

\begin{figure}[ht]
  \centering
  \includegraphics{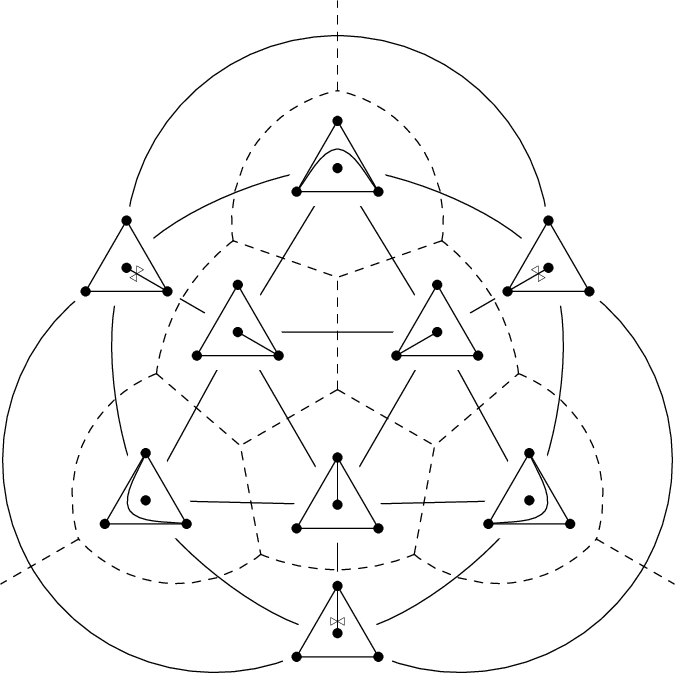}
  \caption{Tagged arcs in a once-punctured
    triangle~$\mathbf{P}_3^\bullet$. 
Solid lines indicate which arcs are compatible. The vertices of the dashed graph
    correspond to tagged triangulations; its edges correspond to flips.
	Note that there is one additional vertex at infinity,
so the graph should be viewed as drawn on a sphere rather than a
	plane.} 
 \label{fig:tagged-arc-d3}
\end{figure}

\pagebreak[3]

\begin{lemma}
\label{lem:tagged-flip}
Any tagged triangulation~$T$ of $\mathbf{P}_n^\bullet$ consists of $n$
tagged~arcs. 
Any tagged arc in a tagged triangulation~$T$ can be replaced in a
unique way by a tagged arc not belonging to~$T$, to form a new tagged
triangulation~$T'$. 
\end{lemma}

\pagebreak[3]

\begin{proof}
It is easy to see that tagged triangulations come in three flavors: 
\begin{enumerate}[leftmargin=.3in]
\item a triangulation in the usual (topological) sense, with every radius plain; 
\item a triangulation in the usual sense, with every radius notched; 
\item the plain and notched versions of the same radius
  inside a punctured digon. Outside of the digon, it is an ordinary
  triangulation. 
\end{enumerate}
In each of these cases, the total number of tagged arcs is~$n$. 
\end{proof}

\pagebreak[3]

In the situation described in Lemma~\ref{lem:tagged-flip},
we say that the tagged triangulations $T$ and~$T'$ are related by a
\emph{flip}. 

\begin{exercise}
Verify that any two tagged triangulations of $\mathbf{P}_n^\bullet$ are connected
via a sequence of flips. 
\end{exercise}

We next define an extended exchange matrix $\tilde B(T)$ associated with a
  tagged triangulation~$T$ of the punctured
  polygon~$\mathbf{P}_n^\bullet$. 
The construction is similar to the one in type~$A_n$. 
In the cases (1) and~(2) above, the rule~\eqref{eq:B(T)} is used; 
in the case~(3), some adjustments are needed around the radii. 

Figure~\ref{fig:octagon-D8} illustrates the
recipe used to define the matrix $\tilde B(T)$, or equivalently the
corresponding quiver.
The vertices of the quiver corresponding to boundary segments 
(i.e., the sides of the polygon) are
frozen;
the ones corresponding to tagged arcs are mutable. 
Details are left to the reader. 

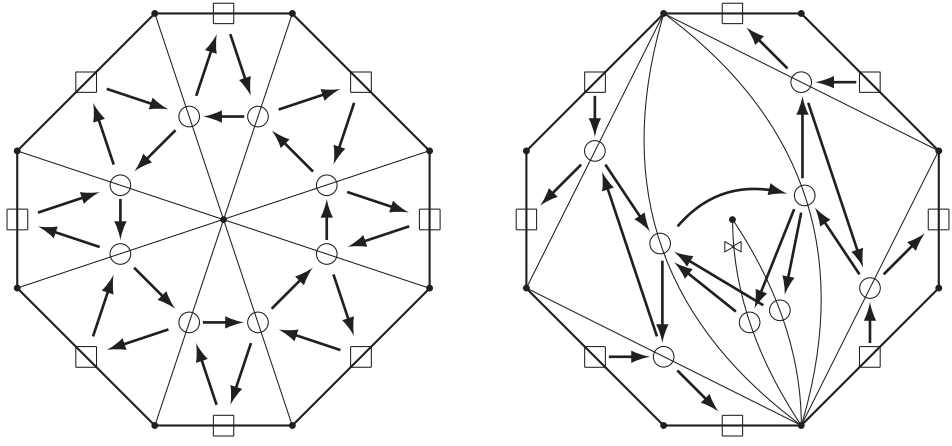
\begin{figure}[ht] 
\begin{center} 
\vspace{-.1in}
\setlength{\unitlength}{2.6pt} 
\begin{picture}(60,66)(0,-3) 
\thicklines 
  \multiput(0,20)(60,0){2}{\line(0,1){20}} 
  \multiput(20,0)(0,60){2}{\line(1,0){20}} 
  \multiput(0,40)(40,-40){2}{\line(1,1){20}} 
  \multiput(20,0)(40,40){2}{\line(-1,1){20}} 
 
  \multiput(20,0)(20,0){2}{\circle*{1}} 
  \multiput(20,60)(20,0){2}{\circle*{1}} 
  \multiput(0,20)(0,20){2}{\circle*{1}} 
  \multiput(60,20)(0,20){2}{\circle*{1}} 
 
  \put(30,30){\circle*{1}} 
 
\thinlines 
\put(40,60){\line(-1,-3){20}} 
\put(20,60){\line(1,-3){20}} 
\put(0,20){\line(3,1){60}} 
\put(0,40){\line(3,-1){60}} 
 

\multiput(28.5,-1.5)(3,0){2}{\line(0,1){3}}
\multiput(28.5,-1.5)(0,3){2}{\line(1,0){3}}

\multiput(28.5,58.5)(3,0){2}{\line(0,1){3}}
\multiput(28.5,58.5)(0,3){2}{\line(1,0){3}}

\multiput(8.5,48.5)(3,0){2}{\line(0,1){3}}
\multiput(8.5,48.5)(0,3){2}{\line(1,0){3}}

\multiput(48.5,48.5)(3,0){2}{\line(0,1){3}}
\multiput(48.5,48.5)(0,3){2}{\line(1,0){3}}

\multiput(8.5,8.5)(3,0){2}{\line(0,1){3}}
\multiput(8.5,8.5)(0,3){2}{\line(1,0){3}}

\multiput(48.5,8.5)(3,0){2}{\line(0,1){3}}
\multiput(48.5,8.5)(0,3){2}{\line(1,0){3}}

\multiput(-1.5,28.5)(3,0){2}{\line(0,1){3}}
\multiput(-1.5,28.5)(0,3){2}{\line(1,0){3}}

\multiput(58.5,28.5)(3,0){2}{\line(0,1){3}}
\multiput(58.5,28.5)(0,3){2}{\line(1,0){3}}


\put(25,45){\circle{3}}
\put(35,45){\circle{3}}
\put(25,15){\circle{3}}
\put(35,15){\circle{3}}
\put(15,25){\circle{3}}
\put(15,35){\circle{3}}
\put(45,25){\circle{3}}
\put(45,35){\circle{3}}

\linethickness{1pt}

\put(33,45){{\vector(-1,0){6}}}
\put(26,48){{\vector(1,3){3}}}
\put(31,57){{\vector(1,-3){3}}}

\put(27,15){{\vector(1,0){6}}}
\put(34,12){{\vector(-1,-3){3}}}
\put(29,3){{\vector(-1,3){3}}}

\put(38,46){{\vector(3,1){9}}}
\put(49,47){{\vector(-1,-3){3}}}
\put(43,37){{\vector(-1,1){6}}}

\put(22,14){{\vector(-3,-1){9}}}
\put(11,13){{\vector(1,3){3}}}
\put(17,23){{\vector(1,-1){6}}}

\put(37,17){{\vector(1,1){6}}}
\put(46,22){{\vector(1,-3){3}}}
\put(47,11){{\vector(-3,1){9}}}

\put(23,43){{\vector(-1,-1){6}}}
\put(14,38){{\vector(-1,3){3}}}
\put(13,49){{\vector(3,-1){9}}}

\put(45,27){{\vector(0,1){6}}}
\put(48,34){{\vector(3,-1){9}}}
\put(57,29){{\vector(-3,-1){9}}}

\put(15,33){{\vector(0,-1){6}}}
\put(12,26){{\vector(-3,1){9}}}
\put(3,31){{\vector(3,1){9}}}

\end{picture} 
\qquad{\ \ }
\begin{picture}(60,66)(0,-3) 
\thicklines 
  \multiput(0,20)(60,0){2}{\line(0,1){20}} 
  \multiput(20,0)(0,60){2}{\line(1,0){20}} 
  \multiput(0,40)(40,-40){2}{\line(1,1){20}} 
  \multiput(20,0)(40,40){2}{\line(-1,1){20}} 
 
  \multiput(20,0)(20,0){2}{\circle*{1}} 
  \multiput(20,60)(20,0){2}{\circle*{1}} 
  \multiput(0,20)(0,20){2}{\circle*{1}} 
  \multiput(60,20)(0,20){2}{\circle*{1}} 
 
  \put(30,30){\circle*{1}} 
 
\thinlines 
\put(0,20){\line(1,2){20}} 
\put(0,20){\line(2,-1){40}} 
\put(40,0){\line(1,2){20}} 
\put(20,60){\line(2,-1){40}} 
 
\qbezier(40,0)(30,15)(30,30)
\qbezier(40,0)(40,15)(30,30)
\put(30.1,26){\makebox(0,0){$\notch$}} 

\qbezier(40,0)(9,23)(20,60)
\qbezier(20,60)(51,37)(40,0)

\multiput(28.5,-1.5)(3,0){2}{\line(0,1){3}}
\multiput(28.5,-1.5)(0,3){2}{\line(1,0){3}}

\multiput(28.5,58.5)(3,0){2}{\line(0,1){3}}
\multiput(28.5,58.5)(0,3){2}{\line(1,0){3}}

\multiput(8.5,48.5)(3,0){2}{\line(0,1){3}}
\multiput(8.5,48.5)(0,3){2}{\line(1,0){3}}

\multiput(48.5,48.5)(3,0){2}{\line(0,1){3}}
\multiput(48.5,48.5)(0,3){2}{\line(1,0){3}}

\multiput(8.5,8.5)(3,0){2}{\line(0,1){3}}
\multiput(8.5,8.5)(0,3){2}{\line(1,0){3}}

\multiput(48.5,8.5)(3,0){2}{\line(0,1){3}}
\multiput(48.5,8.5)(0,3){2}{\line(1,0){3}}

\multiput(-1.5,28.5)(3,0){2}{\line(0,1){3}}
\multiput(-1.5,28.5)(0,3){2}{\line(1,0){3}}

\multiput(58.5,28.5)(3,0){2}{\line(0,1){3}}
\multiput(58.5,28.5)(0,3){2}{\line(1,0){3}}

\put(10,40){\circle{3}}
\put(50,20){\circle{3}}
\put(20,10){\circle{3}}
\put(40,50){\circle{3}}

\put(19.5,26.5){\circle{3}}
\put(40.5,33.5){\circle{3}}

\put(32.5,15){\circle{3}}
\put(36.9,16.8){\circle{3}}


\linethickness{1pt}

\qbezier(22,29)(28,36)(37,34)
\put(37,34){{\vector(4,-1){1}}}

\put(40,31){{\vector(-1,-5){2.4}}}
\put(34.5,17.5){{\vector(-5,3){13}}}
\put(39,31.5){{\vector(-2,-5){5.8}}}
\put(30.5,16.6){{\vector(-5,4){9}}}

\put(40.2,36){{\vector(0,1){12}}}
\put(41,47){{\vector(1,-3){8}}}
\put(48.5,22){{\vector(-2,3){6.5}}}

\put(19.8,24){{\vector(0,-1){12}}}
\put(19,13){{\vector(-1,3){8}}}
\put(11.5,38){{\vector(2,-3){6.5}}}

\put(50,12){{\vector(0,1){6}}}
\put(52,22){{\vector(1,1){6}}}

\put(10,48){{\vector(0,-1){6}}}
\put(8,38){{\vector(-1,-1){6}}}

\put(48,50){{\vector(-1,0){6}}}
\put(38,52){{\vector(-1,1){6}}}

\put(12,10){{\vector(1,0){6}}}
\put(22,8){{\vector(1,-1){6}}}

\end{picture} 
\end{center} 
\caption{Quivers associated with tagged triangulations of a once-punctured 
convex polygon~$\mathbf{P}_n^\bullet$.} 
\label{fig:octagon-D8} 
\end{figure} 

\begin{exercise}
Verify that if two tagged triangulations $T$ and~$T'$ are related by a flip, 
then the extended skew-symmetric 
matrices $\tilde B(T)$
and~$\tilde B(T')$ are related by the corresponding matrix mutation. 
\end{exercise}

Unfortunately, the matrices $\tilde B(T)$ have rank $<\!n$, so exhibiting
a seed pat\-tern with these matrices and finitely many seeds would not
yield a proof of Theorem~\ref{th:type-D-is-finite}, cf.\ Remark~\ref{rem:full-z-rank}.
We will overcome this obstacle by introducing additional frozen variables besides
those labeled by boundary segments. 

\medskip

We now prepare the algebraic ingredients for the proof of
Theorem~\ref{th:type-D-is-finite}. 
Similarly to the type~$A_n$ case, the idea is to  interpret
tagged arcs in a once-punctured polygon
as a family of rational functions.  These rational 
functions 
satisfy the 
type $D_n$ exchange relations, which in turn are 
associated to flips of tagged arcs.
The astute reader will notice that 
the algebraic construction we associate to the type~$D_n$
case  contains the algebraic construction associated 
to the type~$A_{n-1}$ case (Pl\"ucker coordinates of an
$(n+2)$-tuple of vectors in~$\CC^2$), 
reflective of the fact that the 
type~$A_{n-1}$ Dynkin diagram is contained inside the 
type~$D_{n}$ Dynkin diagram.

As in Section~\ref{sec:type-A}, we set $V=\CC^2$, 
and let $\langle u, v\rangle$ be 
the determinant of the $2 \times 2$ matrix with columns
$u,v \in V$.
Let $\mathcal{K}$ be the field of rational functions on
$V^n\times V\times V\times\CC^2\cong\CC^{2n+6}$, 
written in terms of $2n+6$ variables: 
the coordinates of $n+2$ vectors 
\[
v_1,\dots,v_n,a,\overline{a} \in V,  
\]
plus two additional variables $\lambda$ and~$\overline{\lambda}$.

Let $A\in\operatorname{End}(V)$ denote the linear operator defined by 
\begin{equation}
\label{eq:Adefn}
A v = \frac{\overline{\lambda} \langle v,a \rangle \,\overline{a} - \lambda \langle
  v,\overline{a} \rangle \,a}{\langle \overline{a}, a\rangle},  
\end{equation}
so that $a$ (resp.~$\overline{a}$) is an eigenvector for~$A$ with
eigenvalue~$\lambda$ (resp.~$\overline{\lambda}$). 
Let
\[
a^{\notch}=\frac{\overline{\lambda}-\lambda}{\langle
  \overline{a}, a\rangle} \,\overline{a}; 
\]
thus $a^{\notch}$ is also an eigenvector
for~$A$, with eigenvalue~$\overline{\lambda}$. 
We choose this normalization because of the following property,
which is immediate from the definitions. 

\begin{lemma}
\label{lem:v-Av}
For $v\in V$, we have
$\langle v,Av \rangle = \langle v, a\rangle \langle v, a^{\notch} \rangle
$.
\end{lemma} 

We next describe a seed pattern inside~$\mathcal{K}$, 
and in fact inside its subfield $\mathcal{F}$ 
of $\SL_2(\CC)\!\times\! \CC^*$-invariant rational
functions. 
(The group $\SL_2$ acts in the standard way on each of the vectors
$v_1,\dots,v_n,a,\overline{a}$, and fixes $\lambda$
and~$\overline{\lambda}$.
The group $\CC^*$ acts by rescaling the vector~$\overline{a}$.
The subfield $\mathcal{F}$ can be thought of as a field of rational functions on
a $(2n\!+\!2)$-dimensional variety, and indeed, extended clusters 
in our seed pattern will have size $2n+2$.)
Informally speaking, we are going to think of the vectors
$v_1,\dots,v_n$ as located at the corresponding vertices 
of~$\mathbf{P}_n^\bullet$, 
and we shall associate both eigenvectors $a$ and~$a^{\notch}$ with the
puncture~$p$. 
We make a \emph{cut} running from the puncture~$p$ 
to the boundary segment (the side of the polygon) 
that connects the vertices $1$ and~$n$. We will think of
crossing the cut as picking up an application of~$A$. 

\begin{definition}
\label{def:P-gamma}
We associate an element $P_\gamma\in\mathcal{K}$ 
to each tagged arc or boundary segment~$\gamma$
in~$\mathbf{P}_n^\bullet$, 
as follows
(see Figure~\ref{fig:octagon-D8-arcs}): 
\[
P_\gamma\!=\!
\begin{cases}
\langle v_i,v_j \rangle 
&\text{if $\gamma$ doesn't cross the cut, and has endpoints $i$ and $j>i$;}\\
\langle v_j,Av_i \rangle 
&\text{if $\gamma$ crosses the cut, and has endpoints $i$ and $j>i$;}\\
\langle v_i,a \rangle 
&\text{if $\gamma$ is a plain radius with endpoints $p$ and~$i$;}\\
\langle v_i,a^{\notch} \rangle 
&\text{if $\gamma$ is a notched radius with endpoints $p$ and~$i$.}
\end{cases}
\]
\end{definition}

\begin{figure}[ht] 
\begin{center} 
\setlength{\unitlength}{3.8pt} 
\begin{picture}(60,66)(0,0) 
\thicklines 
  \multiput(0,20)(60,0){2}{\line(0,1){20}} 
  \multiput(20,0)(0,60){2}{\line(1,0){20}} 
  \multiput(0,40)(40,-40){2}{\line(1,1){20}} 
  \multiput(20,0)(40,40){2}{\line(-1,1){20}} 
 
  \multiput(20,0)(20,0){2}{\circle*{1}} 
  \multiput(20,60)(20,0){2}{\circle*{1}} 
  \multiput(0,20)(0,20){2}{\circle*{1}} 
  \multiput(60,20)(0,20){2}{\circle*{1}} 
 
  \put(30,30){\circle*{1}} 
 
\thinlines 
\put(40,0){\line(-2,1){40}}

\qbezier(40,0)(30,15)(30,30)
\qbezier(40,0)(40,15)(30,30)
\put(30.1,26){\makebox(0,0){$\notch$}} 

\qbezier(40,0)(55,80)(0,20)

\multiput(49,9)(2,0){2}{\line(0,1){2}}
\multiput(49,9)(0,2){2}{\line(1,0){2}}

\put(20,10){\circle{2}}


\put(32.5,15){\circle{2}}
\put(36.9,16.8){\circle{2}}

\put(37.5,45){\circle{2}}

\put(20,62){\makebox(0,0){$n$}} 
\put(40,62){\makebox(0,0){$1$}} 
\put(40,-2){\makebox(0,0){$i$}} 
\put(-2,20){\makebox(0,0){$j$}} 
\put(27.8,30){\makebox(0,0){$p$}} 

\put(57,9){\makebox(0,0){$\langle v_{i-1},v_i \rangle$}} 
\put(40,48){\makebox(0,0){$\langle v_j,Av_i \rangle$}} 
\put(19,7){\makebox(0,0){$\langle v_i,v_j \rangle$}} 
\put(41.8,16.8){\makebox(0,0){$\langle v_i,a \rangle$}} 
\put(26.5,15){\makebox(0,0){$\langle v_i,a^{\notch} \rangle$}} 

\linethickness{1.5pt}

\put(30,30){{\line(0,1){35}}}
\put(30,67){\makebox(0,0){$A$}}

\end{picture} 
\end{center} 
\caption{Elements of the field~$\mathcal{K}$ associated with tagged
  arcs in~$\mathbf{P}_n^\bullet$.
The boundary segment crossed by the cut corresponds to 
$\langle v_n,Av_1\rangle$.} 
\label{fig:octagon-D8-arcs} 
\end{figure}
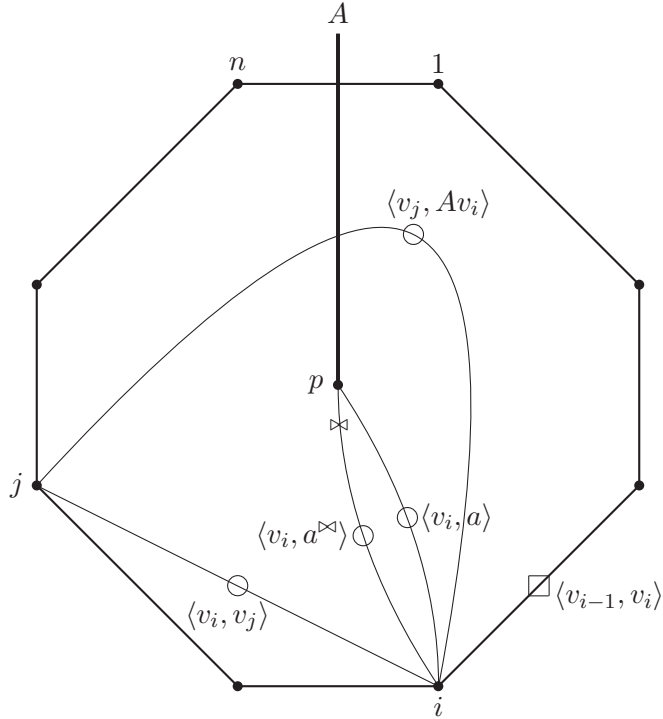 

We can now verify that the $P_{\gamma}$ satisfy a family of relations
that topologists know as ``skein relations" for tagged arcs.

For example, Figure \ref{fig:skein} illustrates the equation
\begin{equation}\label{eq:firstskein}
\langle v_i, v_l \rangle \langle v_k, Av_i \rangle = 
\langle v_k, v_l \rangle 
\langle v_i, a \rangle
\langle v_i, a^{\notch} \rangle
 +
	\langle v_i, v_k \rangle \langle v_l, Av_i \rangle,
\end{equation}
which follows from the 
Grass\-mann-Pl\"ucker relation~\eqref{eq:grassmann-plucker-3term} 
combined with Lemma~\ref{lem:v-Av}.

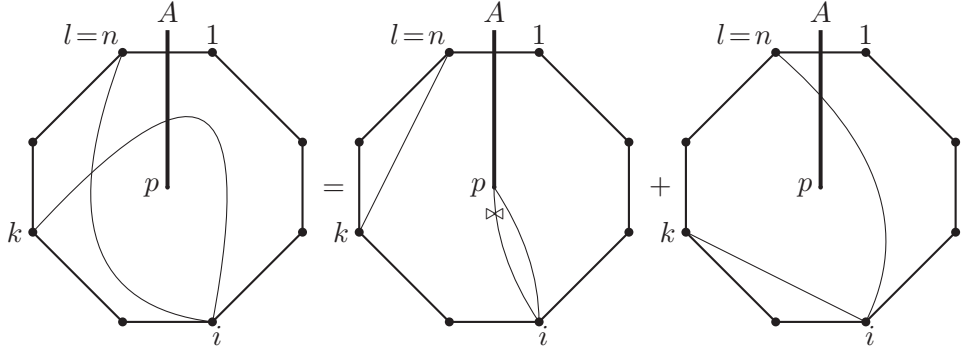
\begin{figure}[ht] 
\begin{center} 
\setlength{\unitlength}{1.7pt} 

\begin{picture}(60,66)(0,0) 
\thicklines 
  \multiput(0,20)(60,0){2}{\line(0,1){20}} 
  \multiput(20,0)(0,60){2}{\line(1,0){20}} 
  \multiput(0,40)(40,-40){2}{\line(1,1){20}} 
  \multiput(20,0)(40,40){2}{\line(-1,1){20}} 
 
  \multiput(20,0)(20,0){2}{\circle*{2}} 
  \multiput(20,60)(20,0){2}{\circle*{2}} 
  \multiput(0,20)(0,20){2}{\circle*{2}} 
  \multiput(60,20)(0,20){2}{\circle*{2}} 
 
  \put(30,30){\circle*{1}} 
 
\thinlines 

\qbezier(40,0)(55,80)(0,20)
\qbezier(40,0)(-1,5)(20,60)

\put(13,64){\makebox(0,0){$l\!=\!n$}} 
\put(40,64){\makebox(0,0){$1$}} 
\put(41,-3){\makebox(0,0){$i$}} 
\put(-4,20){\makebox(0,0){$k$}} 
\put(26.5,30){\makebox(0,0){$p$}} 

\put(67,29.5){\makebox(0,0){$=$}} 


\linethickness{1.5pt}

\put(30,30){{\line(0,1){35}}}
\put(30,69){\makebox(0,0){$A$}} 
\end{picture} 
\hspace{.5cm}
\begin{picture}(60,66)(0,0) 
\thicklines 
  \multiput(0,20)(60,0){2}{\line(0,1){20}} 
  \multiput(20,0)(0,60){2}{\line(1,0){20}} 
  \multiput(0,40)(40,-40){2}{\line(1,1){20}} 
  \multiput(20,0)(40,40){2}{\line(-1,1){20}} 
 
  \multiput(20,0)(20,0){2}{\circle*{2}} 
  \multiput(20,60)(20,0){2}{\circle*{2}} 
  \multiput(0,20)(0,20){2}{\circle*{2}} 
  \multiput(60,20)(0,20){2}{\circle*{2}} 
 
  \put(30,30){\circle*{1}} 
 
\thinlines 
\put(0,20){\line(1,2){20}}

\qbezier(40,0)(30,15)(30,30)
\qbezier(40,0)(40,15)(30,30)
\put(30.2,24){\makebox(0,0){$\notch$}} 

\put(13,64){\makebox(0,0){$l\!=\!n$}} 
\put(40,64){\makebox(0,0){$1$}} 
\put(41,-3){\makebox(0,0){$i$}} 
\put(-4,20){\makebox(0,0){$k$}} 
\put(26.5,30){\makebox(0,0){$p$}} 

\put(67,30){\makebox(0,0){$+$}} 


\linethickness{1.5pt}

\put(30,30){{\line(0,1){35}}}
\put(30,69){\makebox(0,0){$A$}}

\end{picture} 
\hspace{.5cm}
\begin{picture}(60,66)(0,0) 
\thicklines 
  \multiput(0,20)(60,0){2}{\line(0,1){20}} 
  \multiput(20,0)(0,60){2}{\line(1,0){20}} 
  \multiput(0,40)(40,-40){2}{\line(1,1){20}} 
  \multiput(20,0)(40,40){2}{\line(-1,1){20}} 
 
  \multiput(20,0)(20,0){2}{\circle*{2}} 
  \multiput(20,60)(20,0){2}{\circle*{2}} 
  \multiput(0,20)(0,20){2}{\circle*{2}} 
  \multiput(60,20)(0,20){2}{\circle*{2}} 
 
  \put(30,30){\circle*{1}} 
 
\thinlines 
\put(40,0){\line(-2,1){40}}

\qbezier(40,0)(55,30)(20,60)

\put(13,64){\makebox(0,0){$l\!=\!n$}} 
\put(40,64){\makebox(0,0){$1$}} 
\put(41,-3){\makebox(0,0){$i$}} 
\put(-4,20){\makebox(0,0){$k$}} 
\put(26.5,30){\makebox(0,0){$p$}} 


\linethickness{1.5pt}

\put(30,30){{\line(0,1){35}}}
\put(30,69){\makebox(0,0){$A$}}

\end{picture} 
\vspace{-.1in}
\end{center} 
\caption{Pictorial representation of the relation
\eqref{eq:firstskein} (or~\eqref{eq:firstskein-repeated}).}
\label{fig:skein} 
\end{figure} 


More generally, we have the following relations. 
The reader may want to draw a figure corresponding to each relation to 
get some intuition about the form of the relations.

\begin{proposition}
\label{prop:exchange-rel-Dn}
The elements $P_\gamma$ described in Definition~\ref{def:P-gamma}
satisfy the following relations:
\begin{align}
\label{eq:arcs-Dn-exchange}
\langle v_i,v_k \rangle \langle v_j,v_l \rangle 
&=\langle v_i,v_j \rangle \langle v_k,v_l\rangle +\langle v_i,v_l \rangle \langle v_j,v_k \rangle,
\\
\langle v_j,v_l \rangle \langle v_k,A v_i \rangle
&=\langle v_j,v_k \rangle \langle v_l,A v_i \rangle+\langle v_j,A v_i \rangle \langle v_k,v_l \rangle,
\\
\langle v_k,A v_i \rangle \langle v_l,A v_j \rangle
&=\lambda \overline{\lambda} \langle v_i,v_j \rangle \langle v_k,v_l \rangle+\langle v_k,A v_j \rangle \langle v_l,Av_i \rangle,
\\
\langle v_i, v_k \rangle \langle v_l,A v_j \rangle
&=\langle v_l,Av_i \rangle \langle v_j, v_k \rangle+\langle v_l,A v_k \rangle \langle v_i,v_j \rangle,
\\
\label{eq:firstskein-repeated}
\langle v_i, v_l \rangle \langle v_k, Av_i \rangle &= 
\langle v_i, v_k \rangle \langle v_l, Av_i \rangle +
\langle v_k, v_l \rangle \langle v_i, a \rangle \langle v_i, a^{\notch} \rangle, \\
\langle v_j, Av_i \rangle \langle v_l, Av_j \rangle &=
\lambda \overline{\lambda} \langle v_i, v_j \rangle\langle v_j, v_l \rangle+
\langle v_j, a^{\notch} \rangle \langle v_j,a \rangle \langle v_l, Av_i \rangle,\\
\langle v_i, v_l \rangle \langle v_l, Av_j \rangle &=
\langle v_l, Av_i \rangle \langle v_j, v_l \rangle +
\langle v_l, a^{\notch} \rangle \langle v_l, a \rangle \langle v_i, v_j \rangle\\
\label{eq:one-radius-exchange}
\langle v_i, v_k \rangle \langle v_j,a \rangle
&=\langle v_i,v_j \rangle \langle v_k, a \rangle+\langle v_i,a \rangle \langle v_j,v_k \rangle,
\\
\langle v_j, Av_i \rangle \langle v_k,a \rangle
&=\overline{\lambda} \langle v_j,v_k \rangle \langle v_i, a \rangle+\langle v_j,a \rangle \langle v_k,A v_i \rangle,
\\
\label{eq:crosscut-radius-exchange}
\langle v_k, Av_j \rangle \langle v_i,a \rangle
&=\langle v_k,Av_i \rangle \langle v_j, a \rangle+\lambda \langle v_k,a \rangle \langle v_i,v_j \rangle,
\\
\label{eq:two-radii-exchange}
\langle v_i, a^{\notch} \rangle \langle v_j,a \rangle 
&=\overline{\lambda} \langle v_i,v_j \rangle +\langle v_j,A v_i \rangle.
\end{align}
where $1\le i<j<k<\ell\le n$. 
In addition, they satisfy the relations obtained from
\eqref{eq:one-radius-exchange}--\eqref{eq:two-radii-exchange} 
by interchanging $\lambda$ with~$\overline{\lambda}$ and $a$
with~$a^{\notch}$ throughout. 
\end{proposition}

\proof
Each of these relations follows from a suitable instance of the
Grass\-mann-Pl\"ucker relation~\eqref{eq:grassmann-plucker-3term}, 
using Lemma~\ref{lem:v-Av}
and the identities
\begin{align*}
\langle Av,Av' \rangle &=\det(A)\,\langle v,v'\rangle
=\lambda \overline{\lambda}\langle v,v' \rangle,\\
\langle Av,a \rangle &=\overline{\lambda}\langle v,a \rangle,\\
\langle Av,a^{\notch} \rangle &=\lambda \langle v,a^{\notch} \rangle.
\end{align*}
For example, \eqref{eq:crosscut-radius-exchange} can be obtained
from the identity 
\[
\langle v_k ,Av_j \rangle \langle Av_i,a \rangle
= \langle v_k , Av_i \rangle \langle Av_j, a \rangle+
\langle v_k , a \rangle \langle Av_i, Av_j \rangle, 
\]
while 
\eqref{eq:two-radii-exchange} follows (using \cref{lem:v-Av}) from 
\[
\langle v_i, Av_i \rangle \langle v_j,a \rangle 
=\langle v_i,v_j \rangle \langle Av_i,a \rangle +\langle
v_i,a\rangle\langle v_j,A v_i \rangle. \qed
\]

For a tagged triangulation $T$ of the punctured
polygon~$\mathbf{P}_n^\bullet$, let $\tilde\xx(T)$ 
be the $(2n+2)$-tuple 
consisting of the elements~$P_\gamma$ labeled by the tagged arcs and boundary
segments in~$T$, together with $\lambda$ and~$\overline{\lambda}$.
We view the elements $P_\gamma\in\tilde\xx(T)$ labeled by tagged arcs as cluster
variables, and those labeled by the boundary segments as frozen variables; 
$\lambda$~and~$\overline{\lambda}$ are frozen variables as well. 

Let $\gamma$ be a tagged arc in a tagged triangulation~$T$,
and let $\gamma'$ be the tagged arc that replaces~$\gamma$ when the
latter is flipped. 
One can check that in every such instance, exactly one of the
relations in Proposition~\ref{prop:exchange-rel-Dn}
has the product~$P_\gamma\,P_{\gamma'}$ on the left-hand side;
the corresponding right-hand side is always a sum of two monomials in the
elements of~$\tilde\xx(T)$. 
We let $\tilde B^\bullet(T)$ denote the matrix 
encoding these relations for all
tagged arcs in~$T$.  
The matrix $\tilde B^\bullet(T)$ 
 can be seen to be an extension of the matrix
$\tilde B(T)$ by two extra rows corresponding to $\lambda$
and~$\overline{\lambda}$. 
(Put differently, setting $\lambda=\overline{\lambda}=1$ produces
relations encoded by~$\tilde B(T)$.) 

\begin{example}
\label{example:Dtriangulation}
Figure~\ref{fig:T-circ} shows a triangulation~$T_\circ$ of a
punctured pentagon, 
and the associated matrix~$\tilde B^\bullet(T_\circ)$.
Columns $1$ and $5$ of $\tilde B^\bullet(T_\circ)$ encode the exchange relations
\begin{align*}
\langle v_5, Av_2 \rangle \langle v_1,a \rangle &= 
\langle v_5, Av_1 \rangle \langle v_2,a \rangle +
\lambda \langle v_5,a \rangle \langle v_1, v_2 \rangle = 
x_{10} x_2 + \lambda x_5 x_6 \\
\langle v_4, Av_1 \rangle 
\langle v_5, a \rangle &= 
\overline{\lambda} \langle v_4, v_5 \rangle 
\langle v_1, a \rangle + 
\langle v_4, a \rangle 
\langle v_5, A v_1 \rangle = 
\overline{\lambda} x_9 x_1 + x_4 x_{10}.
\end{align*}
The matrix 
$\tilde B^\bullet(T_\circ)$ 
has full $\ZZ$-rank. 
However, the submatrix of $\tilde B(T_\circ)$  consisting of the first
ten rows does not have full rank (each row sum is~$0$).

This example can be straightforwardly generalized to $n \neq 5$.
\end{example}

\begin{figure}[ht] 
\begin{center} 
\setlength{\unitlength}{2.5pt} 
\begin{picture}(144,62)(-5,-3) 
\thicklines 
 
 
  \put(30,30){\circle*{1}} 
 
\thinlines 
\put(30,30){\line(-1,1){20}}
\put(30,30){\line(1,1){20}}
\put(30,30){\line(0,-1){30}}
\put(30,30){\line(-5,-3){25}}
\put(30,30){\line(5,-3){25}}

\put(10,50){\line(1,0){40}}
\put(10,50){\line(-1,-7){5}}
\put(50,50){\line(1,-7){5}}
\put(30,0){\line(-5,3){25}}
\put(30,0){\line(5,3){25}}

{\put(10,50){\circle*{1.5}}}
{\put(50,50){\circle*{1.5}}}
{\put(30,0){\circle*{1.5}}}
{\put(5,15){\circle*{1.5}}}
{\put(55,15){\circle*{1.5}}}


\linethickness{1.5pt}

\put(30,58){\makebox(0,0){$A$}} 

\put(44,40){\makebox(0,0){$x_1$}}
\put(44,25){\makebox(0,0){$x_2$}}
\put(32.5,15){\makebox(0,0){$x_3$}}
\put(16,25){\makebox(0,0){$x_4$}}
\put(16,40){\makebox(0,0){$x_5$}}

\put(49.5,33){\makebox(0,0){$x_6$}}
\put(42,10){\makebox(0,0){$x_7$}}
\put(18,10){\makebox(0,0){$x_8$}}
\put(10.5,33){\makebox(0,0){$x_9$}}
\put(37,52){\makebox(0,0){$x_{10}$}}

\put(28,26){\makebox(0,0){$p$}} 

\put(30,30){\line(0,1){25}}


\put(100,30){\makebox(0,0){$\displaystyle
\tilde B^\bullet(T_\circ)
=
\left[\,\begin{matrix}
0  & -1 & 0  &  0 & 1 \\
1  & 0  & -1 &  0 & 0 \\
0  & 1  & 0  & -1 & 0 \\
0  & 0  & 1  & 0 & -1 \\
-1  & 0  & 0  & 1 & 0 \\
\hline\\[-.15in]
-1 & 1 & 0 & 0 & 0\\
0 & -1 & 1 & 0 & 0\\
0 & 0 & -1 & 1 & 0\\
0 & 0 & 0 & -1 & 1\\
1 & 0 & 0 & 0 & -1\\
\hline\\[-.15in]
-1 & 0 &0 & 0 & 0\\
0 & 0 & 0 & 0 & 1
\end{matrix}\,
\right]
$}}
\end{picture} 
\vspace{-.25in}
\end{center} 
\caption{The (tagged) triangulation $T_\circ$
of a punctured pentagon ($n=5$), and the corresponding
matrix~$\tilde B^\bullet(T_\circ)$. 
The columns of $\tilde B^\bullet(T_\circ)$ correspond to $x_1$, \dots,
$x_5$. 
The rows correspond to 
$x_1$, \dots, $x_{10}$, $\lambda$, $\overline{\lambda}$, in~this order. 
} 
\label{fig:T-circ}
\end{figure}
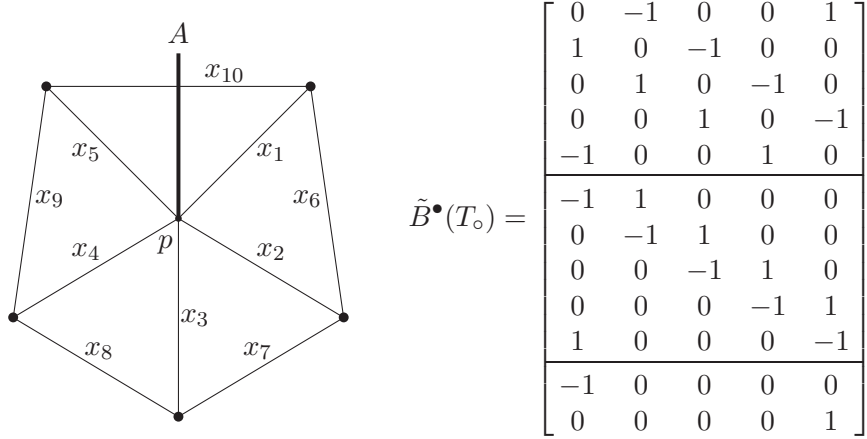

\begin{proposition}
\label{prop:seed-pattern-tagged-Dn}
For any tagged
triangulation~$T$ of~$\mathbf{P}_n^\bullet$, 
the elements of $\tilde\xx(T)$ are algebraically independent,
so  $(\tilde\xx(T),\tilde B^\bullet(T))$ is a seed in the field generated by~$\tilde\xx(T)$. 
The seeds associated to tagged triangulations related by a flip are
related to each other by the corresponding mutation. 
\end{proposition}

\begin{proof} 
It is a straightforward but tedious exercise to verify that the
matrix $\tilde B^\bullet(T)$ undergoes a mutation when a tagged arc
in~$T$ is flipped to produce a new tagged triangulation~$T'$. 
We already know that the elements of $\tilde\xx(T)$ and
$\tilde\xx(T')$ satisfy the corresponding exchange relation. 
It remains to show algebraic independence.
Since each exchange is a birational transformation, 
it suffices to prove algebraic independence for one particular choice
of~$T$. 

Consider the triangulation~$T_\circ$ made up of $n$ plain radii, 
cf.\ Figure~\ref{fig:T-circ}
and Example~\ref{example:Dtriangulation}.
Then $\lambda$,  $\overline{\lambda}$, and 
$\langle v_n, A v_1 \rangle$ are the only elements of
$\tilde\xx(T_\circ)$ which involve $\lambda$,  $\overline{\lambda}$,
or~$\overline{a}$. 
The remaining $2n-1$ elements  are 
$2\times 2$ minors of the matrix with columns  
$v_1, v_2, \dots,v_n,a$; 
they form an extended cluster in the corresponding
seed pattern of type~$A_{n-2}$. 
Hence they are algebraically independent, and the claim follows. 
\end{proof}

\noindent
\textbf{Proof of 
Theorem~\ref{th:type-D-is-finite}}.
By Proposition~\ref {prop:seed-pattern-tagged-Dn}, the seeds 
$(\tilde\xx(T),\tilde B^\bullet(T))$ form a seed pattern.
This pattern has finitely many seeds because a once-punctured polygon
has finitely many tagged triangulations.
As verified in Example \ref{example:Dtriangulation},
the extended exchange matrix $\tilde
B^\bullet(T_\circ)$ has full $\ZZ$-rank. 
The theorem follows by Corollary~\ref{cor:z-span} and 
\cref{rem:full-z-rank}.
\endproof

\begin{corollary}
\label{cor:Dn-structure}
Cluster variables in a seed pattern of type~$D_n$ can be
labeled by the tagged arcs in a once-punctured convex $n$-gon~$\mathbf{P}_{n}^\bullet$ so that
clusters correspond to tagged triangulations, 
and mutations correspond to flips. 
Cluster variables labeled by different tagged arcs are distinct.
There are altogether $n^2$ cluster variables
and $\frac{3n-2}{n}\binom{2n-2}{n-1}$ seeds (or clusters). 
\end{corollary}

Most of the work in the proof of the corollary concerns 
the enumeration of seeds.  
Let $a_n$ (resp.,~$d_n$) denote the number of seeds in a pattern of
type~$A_n$ (resp.,~$D_n$), including $a_0=1$, $d_2=4$, and $d_3=14$ 
by convention. 
We know from Corollary~\ref{cor:An-structure} that $a_n=\frac{1}{n+2}\binom{2n+2}{n+1}$. 

\begin{lemma}
\label{lem:count-with-radius}
The number of tagged triangulations~$T$ containing a given
radius~$\gamma$ 
is equal
to~$a_{n-1}$. 
\end{lemma}

\begin{proof}
Assume without loss of generality that $\gamma$ is a plain radius
connecting the puncture~$p$ to a boundary point~$i$. 
If $T$ contains the notched counterpart of~$\gamma$, then the
remaining $n-1$ arcs of~$T$ form a triangulation of the $(n+1)$-gon
obtained from $\mathbf{P}_{n}^\bullet$ by cutting it open
along~$\gamma$. 
The number of such triangulations is~$a_{n-2}$. 
If $T$ does not contain the notched counterpart of~$\gamma$, then 
it is an ordinary triangulation of $\mathbf{P}_{n}^\bullet$ cut open
along~$\gamma$, in which we are not allowed to use the arc connecting 
the two boundary points obtained from~$i$. 
The number of such triangulations is~$a_{n-1}-a_{n-2}$. 
\end{proof}

\begin{lemma}
\label{lem:recurrence-Dn}
The numbers $d_n$ satisfy the recurrence
\[
d_{n}=\sum_{k=0}^{n-3} a_k \,d_{n-1-k} + 2a_{n-1}\,.
\]
\end{lemma}

\begin{proof}
Consider two ways of
counting the triples $(T,\gamma,i)$ in which $T$ is a tagged
triangulation of~$\mathbf{P}_{n}^\bullet$,
$\gamma$~is a tagged arc in~$T$, and $i$~is an endpoint~of~$\gamma$.
Selecting~$T$ first, then~$\gamma$ and then~$i$, we see that the
number of such triples is equal to~$d_n \cdot n \cdot 2$. 
Selecting~$i$ first, then~$\gamma$ and~$T$ together, 
treating the cases $i\neq p$ and $i=p$ separately, and using
Lemma~\ref{lem:count-with-radius}, 
we obtain: 
\[
2nd_n=n\Bigl(2 \sum_{k=0}^{n-3} a_k \,d_{n-1-k}+2a_{n-1}\Bigr) + 2na_{n-1}\,,
\] 
as desired. (The factor of~$2$ before the sum accounts for the two
ways of cutting up the polygon: 
leaving the puncture to the left or to the right of~$\gamma$, as we
move away from~$i$.)
\end{proof}

\begin{proof}[Proof of Corollary~\ref{cor:Dn-structure}]
The statements in the first sentence of the corollary have already
been established. 
The claim of distinctness can be verified in the special case $n=4$ by
direct calculation; the general case then follows by restriction. 

The cluster variables are labeled by
$\frac{n^2-3n+2}{2}$ ordinary arcs not crossing the cut,
$\frac{n^2-n-2}{2}$ ordinary arcs crossing the cut,
and $n$ radii of each of the two flavors, bringing the total
to~$n^2$. 

The formula $d_n=\frac{3n-2}{n}\binom{2n-2}{n-1}$ can now be proved by
induction using the recurrence in Lemma~\ref{lem:recurrence-Dn}. 
We leave this step to the reader. 
\end{proof}

We conclude this section by examining a couple of examples of cluster
algebras of type~$D_n$ that occur in the settings discussed in
Chapter~\ref{ch:tp-examples}. 

\pagebreak[3]

\begin{example} 
\label{example:mat33}
The ring of polynomials in 9 variables $z_{ij}$ ($i,j\in\{1,2,3\}$),
viewed as matrix entries of a $3\times 3$ matrix
\[
z=\begin{bmatrix}
z_{11} & z_{12} & z_{13} \\[.05in]
z_{21} & z_{22} & z_{23} \\[.05in]
z_{31} & z_{32} & z_{33}
\end{bmatrix}
\in\operatorname{Mat}_{3,3}(\CC)\cong\CC^9,
\]
carries a natural cluster algebra structure of  type~$D_4$,
see, e.g.,~\cite{cdm}.
The seed pattern giving rise to this cluster structure 
contains, among others, the seeds associated with double wiring
diagrams with $3$~wires of each kind. 

Let us use the initial seed associated with 
the diagram~$D$ in 
Figure~\ref{fig:chamber-quiver2}. 
Thus the initial cluster is 
\begin{equation}
\label{eq:G36-initial-cluster}
\xx=(x_1,x_2,x_3,x_4)=(\Delta_{1,2}, \Delta_{3,2}, \Delta_{13,12},
\Delta_{13,23}),  
\end{equation}
the 5 coefficient variables are
\begin{equation}
\label{eq:3x3-coeffs}
(x_5,\dots,x_{9})=
(\Delta_{123,123}, \Delta_{12,23}, \Delta_{1,3}, \Delta_{3,1},
\Delta_{23,12}),    
\end{equation}
and the exchange relations (encoded by 
the quiver shown in Figure~\ref{fig:chamber-quiver2}) are: 
\begin{align*}
\Delta_{1,2}\,\Delta_{3,3} &=
  \Delta_{1,3}\,\Delta_{3,2} + \Delta_{13,23}\,,\\
\Delta_{3,2}\,\Delta_{1,1} &=
  \Delta_{13,12} + \Delta_{3,1}\,\Delta_{1,2}\,,\\
\Delta_{13,12}\,\Delta_{23,23} &=
  \Delta_{23,12}\,\Delta_{13,23} + \Delta_{123,123}\,\Delta_{3,2}\,,\\
\Delta_{13,23}\,\Delta_{12,12} &=
  \Delta_{123,123}\,\Delta_{1,2} + \Delta_{12,23}\,\Delta_{13,12}\,.
\end{align*}
The mutable part of $Q(D)$ is an oriented $4$-cycle, so we are indeed dealing with
a pattern of type~$D_4$. 
By Corollary~\ref{cor:Dn-structure}, it has 50~seeds,
which include 34~seeds associated with double wiring diagrams, see
Figure~\ref{fig:schemes}. 
There are 16 cluster variables,
labeled by the 16 tagged arcs in a once-punctured
quadrilateral. 
They include the 14 minors of~$z$ not listed in~\eqref{eq:3x3-coeffs}. 
As suggested by Exercise~\ref{exercise:KL},
the remaining two cluster variables are the polynomials $K(z)$ and
$L(z)$ given by~\eqref{eq:K(z)}--\eqref{eq:L(z)}. 

\begin{figure}[htbp]
\begin{center}
\begin{tabular}{ccc}
\setlength{\unitlength}{3.8pt}
\begin{picture}(33,20)(2,10)
\thinlines

\multiput(10,10)(20,0){2}{\circle*{1}}
\multiput(10,30)(20,0){2}{\circle*{1}}
\put(20,20){\circle*{1}}
\multiput(10,10)(0,20){2}{\line(1,0){20}}
\multiput(10,10)(20,0){2}{\line(0,1){20}}

\put(26,16){\makebox(0,0){$1$}}
\put(14,16){\makebox(0,0){$2$}}
\put(14,24){\makebox(0,0){$3$}}
\put(26,24){\makebox(0,0){$4$}}
\thicklines
\put(10,10){\line(1,1){20}}
\put(10,30){\line(1,-1){20}}

\end{picture}
&
\setlength{\unitlength}{3.8pt}
\begin{picture}(33,20)(2,10)
\thinlines
\multiput(10,10)(20,0){2}{\circle*{1}}
\multiput(10,30)(20,0){2}{\circle*{1}}
\put(20,20){\circle*{1}}

\put(28,16){\makebox(0,0){$\Delta_{1,2}$}}
\put(12,16){\makebox(0,0){$\Delta_{3,2}$}}
\put(10,25){\makebox(0,0){$\Delta_{13,12}$}}
\put(30,25){\makebox(0,0){$\Delta_{13,23}$}}
\thicklines
\put(10,10){\line(1,1){20}}
\put(10,30){\line(1,-1){20}}

\end{picture}
\end{tabular}
\vspace{-.1in}
\end{center}
\caption{The triangulation of a once-punctured square 
representing the initial
cluster~\eqref{eq:G36-initial-cluster}. 
Note that the frozen variables are not shown; 
the frozen variables of the initial cluster at the right 
do not correspond to the edges of the square.
}
\label{fig:arcs-G36}
\end{figure}
\end{example}

\begin{example}
The coordinate ring $\CC[\SL_5]^U$ of the 
basic affine space for~$\SL_5$ has a natural cluster algebra structure of
type~$D_6$, to be discussed in detail in Example~\ref{ex:SL5}.  
This cluster algebra has 36 cluster variables and 8 frozen variables.
This set of 44 generators includes $2^5\!-\!2\!=\!30$ flag minors plus 14
non-minor elements. 
The total number of clusters is 672.  

\end{example}

\begin{remark}
For any irreducible representation of a semisimple algebraic 
group~$G$, Lusztig~\cite{Lusztig-book}
introduced the concept of a canonical (resp., dual canonical) basis; 
these correspond to the lower and upper global bases of Kashiwara~\cite{Kashiwara}.
While we will not define the dual canonical bases here,
we note that they are strongly connected to 
the theory of cluster algebras. \linebreak[3]
In particular, it was conjectured in~\cite[p.~498]{ca1}
that every cluster monomial in~$\CC[\SL_k]^U$
(i.e., a monomial in the elements of some extended cluster) 
belongs to the dual canonical basis; 
this conjecture was proved in~\cite{kkko}.
For $k\le 5$, cluster monomials
make up the entire dual canonical basis in~$\CC[\SL_k]^U$.  
\end{remark}

\begin{example}
The homogeneous coordinate ring of a Schubert divisor in the
Grassmannian~$\operatorname{Gr}_{2,n+2}$ has a structure of a cluster algebra of type~$D_n$; 
see Example~\ref{ex:realization-Dn} for details. 
\end{example}


\section{Seed patterns of types $B_n$ and $C_n$}
\label{sec:BC-finite}

By \cref{def:type},  a seed pattern of rank~$n\ge 2$ (or the associated cluster
algebra) is of type~$B_n$ 
if one of its exchange matrices 
 is 
\begin{equation}
\label{typeB}
B=
\left[\,\begin{matrix}
0  & -2 & 0 & 0  & \cdots & 0 & 0 \\
1  & 0  & -1 & 0 & \cdots & 0 & 0 \\
0  & 1  & 0  & -1 & \cdots & 0 & 0 \\
0  & 0  & 1  & 0 & \cdots & 0 & 0 \\[-.05in]
\vdots & \vdots & \vdots & \vdots & \ddots & \vdots & \vdots \\
0 & 0 & 0 & 0 & \cdots & 0 & -1 \\
0 & 0 & 0 & 0 & \cdots &1 & 0 
\end{matrix}\,
\right]
\end{equation}
(up to simultaneous permutation of rows and columns). 

Similarly, a seed pattern (or cluster algebra) of rank~$n\ge 3$ is of
  type~$C_n$ 
if  one of its exchange matrices has the form 
\begin{equation}
\label{typeC}
B=
\left[\,\begin{matrix}
0  & -1 & 0 & 0  & \cdots & 0 & 0 \\
2  & 0  & -1 & 0 & \cdots & 0 & 0 \\
0  & 1  & 0  & -1 & \cdots & 0 & 0 \\
0  & 0  & 1  & 0 & \cdots & 0 & 0 \\[-.05in]
\vdots & \vdots & \vdots & \vdots & \ddots & \vdots & \vdots \\
0 & 0 & 0 & 0 & \cdots & 0 & -1 \\
0 & 0 & 0 & 0 & \cdots &1 & 0 
\end{matrix}\,
\right].
\end{equation}
(We continue to use the conventions of~\cite{kac},
cf.\ Definition~\ref{def:coxeter-dynkin}.) 

\begin{theorem}
\label{th:type-B-is-finite}
Seed patterns of type~$B_n$ are of finite type. 
\end{theorem}

\begin{theorem}
\label{th:type-C-is-finite}
Seed patterns of type~$C_n$ are of finite type. 
\end{theorem}

Note that type~$B_2$ is already covered by
Theorem~\ref{th:finite-type-rank2}, with $bc=2$. 

We will prove
Theorems~\ref{th:type-B-is-finite}--\ref{th:type-C-is-finite} using
the technique of folding 
introduced in Section~\ref{sec:folding}. 
Specifically, we will obtain seed patterns of type~$C_n$ by folding
seed patterns of type~$A_{2n-1}$, and we will get type~$B_n$ from type~$D_{n+1}$. 

The following statement follows easily from the definitions. 

\begin{lemma}
\label{lem:admissible-if-add-frozen}
Let $Q$ be a quiver globally foldable with respect to an action of a group~$G$.
Let $\overline{Q}$ be a quiver constructed from~$Q$ by adding some new
frozen vertices together with some arrows connecting them to the mutable vertices in~$Q$.
We extend the action of $G$ from $Q$ to 
$\overline{Q}$ by making $G$ fix every newly added vertex.
Assume that the new arrows are compatible with the action of~$G$ on~$\overline Q$, i.e.
this action satisfies condition~{\rm(2)} of \cref{def:folding}.
Then the quiver $\overline{Q}$ is globally foldable with respect to~$G$.
\end{lemma}

\pagebreak[3]

\begin{corollary}
\label{cor:folding-finite-type}
Let $Q$ be a quiver without frozen vertices. 
Suppose that $Q$ is globally foldable with respect to an action of a group~$G$. 
If every seed pattern with the initial exchange matrix~$B(Q)$
is of finite type (regardless of the choice of an initial extended exchange matrix~$\tilde B$
containing~$B(Q)$),
then every seed pattern with the initial exchange matrix~$B(Q)^G$ 
is of finite type. 
\end{corollary}

\begin{proof}
Any extended exchange matrix~$\tilde B$ that extends $B(Q)^G$ 
can be obtained from an extended exchange matrix that extends $B(Q)$ 
via the folding procedure described in Lemma~\ref{lem:admissible-if-add-frozen}. 
The claim then follows from  Corollary~\ref{cor:always-admissible}. 
\end{proof}

\begin{proof}[Proof of \cref{th:type-C-is-finite}]
Our proof strategy is as follows. 
We will construct a type $A_{2n-1}$ quiver~$Q_0$ with a group $G$ acting on its vertices,
so that $Q_0$ is globally foldable with respect to~$G$, 
and $B(Q_0)^G$ is the $n \times n$ exchange matrix  of type~$C_n$
from \cref{typeC}.
\cref{th:type-C-is-finite} will then follow from
\cref{th:type-A-is-finite}
and \cref{cor:folding-finite-type}.

The combinatorial model for a seed pattern of type~$A_{2n-1}$
presented in \cref{sec:type-A}
uses triangulations
of a convex $(2n+2)$-gon~$\mathbf{P}_{2n+2}$, with vertices numbered 
$1,2,\dots, 2n+2$ in clockwise order.
Consider the centrally symmetric triangulation $T_0$
(see \cref{fig:centrallysymmetric})
consisting of the following diagonals:
\begin{itemize}[leftmargin=.2in]
\item a ``diameter'' $d_1$ connecting vertices $1$ and $n+2$;
\item diagonals $d_2, d_3, \dots, d_n$ connecting 
$n+2$ with $2, 3, \dots, n$;
\item diagonals $d_{2'}, d_{3'}, \dots, d_{n'}$ connecting~$1$ with $n+3$, $n+4, \dots, 2n+1$.
\end{itemize}
Let $Q_0$ denote the 
mutable part of the quiver associated to the triangulation $T_0$
(see Section~\ref{sec:triangulations}); $Q_0$ 
is an orientation 
of the type $A_{2n-1}$ Dynkin diagram with  vertices
labeled $n', \dots, 3', 2', 1, 2, 3, \dots, n$ in order, 
and arrows directed towards the central vertex~$1$,
see \cref{fig:centrallysymmetric}.
The group $G=\ZZ/2\ZZ$ acts on the vertices of $Q_0$ by 
exchanging $i'$ and~$i$ for $2 \leq i \leq n$, and fixing the vertex~$1$.
It is easy to see that $Q_0$ is $G$-admissible, and moreover  
$B(Q_0)^G$ is the exchange matrix of type~$C_n$ from \cref{typeC}.

\begin{figure}[ht]
\begin{center}
\setlength{\unitlength}{2.4pt}
\begin{picture}(60,60)(0,0)
\thinlines
  \multiput(0,20)(60,0){2}{\line(0,1){20}}
  \multiput(20,0)(0,60){2}{\line(1,0){20}}
  \multiput(0,40)(40,-40){2}{\line(1,1){20}}
  \multiput(20,0)(40,40){2}{\line(-1,1){20}}

  \multiput(20,0)(20,0){2}{\circle*{1}}
  \multiput(20,60)(20,0){2}{\circle*{1}}
  \multiput(0,20)(0,20){2}{\circle*{1}}
  \multiput(60,20)(0,20){2}{\circle*{1}}

\put(40,60){\line(-1,-3){20}}
\put(40,60){\line(-1,-1){40}}
\put(40,60){\line(-2,-1){40}}

\put(20,0){\line(2,1){40}}
\put(20,0){\line(1,1){40}}


\thicklines

{\put(41,61){$1$}}
{\put(62,39){$2$}}
{\put(62,19){$3$}}
{\put(41,-3){$4$}}
{\put(17,-3){$5$}}
{\put(-4,19){$6$}}
{\put(-4,39){$7$}}
{\put(17,61){$8$}}

{\put(30,30){\circle{3}}}
{\put(32.5,30){$d_1$}}
{\put(40,20){\circle{3}}}
{\put(42.5,19){$d_2$}}
{\put(42,11){\circle{3}}}
{\put(39,5.5){$d_3$}}
{\put(20,40.5){\circle{3}}}
{\put(12.5,40){$d_{2'}$}}
{\put(20,50){\circle{3}}}
{\put(19,53){$d_{3'}$}}

\linethickness{1.2pt}

\put(41.3,13){{\vector(-1,4){1.3}}}
\put(38.5,21.5){{\vector(-1,1){7}}}
\put(21.5,38.5){{\vector(1,-1){7}}}
\put(20,47.5){{\vector(0,-1){5}}}
\end{picture}
\end{center}

\caption{A centrally symmetric triangulation $T_0$ and its quiver
  $Q_0=Q(T_0)$. 
The folded matrix $B(Q(T_0))^G$ has 
type~$C_n$; here $n=3$.}
\label{fig:centrallysymmetric}
\end{figure}
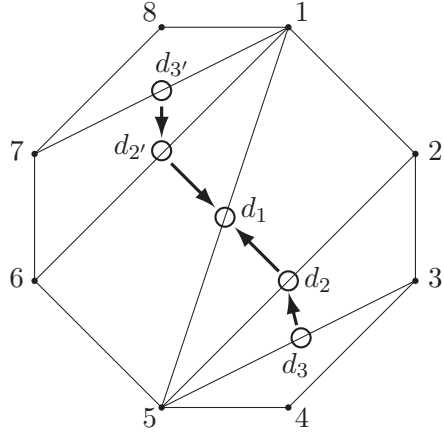

It remains to show that $Q_0$ is globally foldable. 
Besides acting on the set
$\{n', \dots, 3', 2', 1, 2, 3, \dots, n\}$,
the group $G=\ZZ/2\ZZ$ naturally acts---by central symmetry---on 
the set of diagonals of the polygon~$\mathbf{P}_{2n+2}$. 
This action has two kinds of orbits:
(1) the ``diameters'' of~$\mathbf{P}_{2n+2}$ fixed by the $G$-action,
and (2) pairs of centrally symmetric diagonals. 

A transformation  $\mu_{J_k} \circ \dots \circ \mu_{J_1}$ associated with 
a sequence $J_1, \dots, J_k$ of $G$-orbits in $\{n', \dots, 3', 2', 1, 2, 3, \dots, n\}$
corresponds to a sequence of flips 
associated with diameters or pairs of centrally symmetric diagonals. 
Such a sequence transforms $T_0$ into another centrally symmetric triangulation~$T$ in which 
the diameter retains the label~$1$ and each pair of centrally symmetric diagonals 
have labels $i$ and~$i'$, for some~$i$.
It is easy to see that the quiver associated to~$T$ is $G$-admissible.
Thus $Q_0$ is globally foldable.
\end{proof}

\begin{proof}[Proof of Theorem~\ref{th:type-B-is-finite}]
We will follow the strategy used 
in the proof of Theorem~\ref{th:type-C-is-finite}.  Namely, we will
construct a type $D_{n+1}$ quiver $Q_0$ 
with a group $G$ acting on its vertices, 
so that $Q_0$ is globally foldable with respect to~$G$, 
and $B(Q_0)^G$ is the exchange matrix  of type~$B_n$
from~\eqref{typeB}. 
\cref{th:type-B-is-finite} will then follow from
\cref{th:type-D-is-finite}
and \cref{cor:folding-finite-type}.

The combinatorial model for a seed pattern of type~$D_{n+1}$ ($n\ge 3$)
uses tagged triangulations of a convex $(n+1)$-gon
$\mathbf{P}_{n+1}^\bullet$ with a
puncture~$p$  in its interior.   The vertices of 
$\mathbf{P}_{n+1}^\bullet$ are numbered $1, 2, \dots, n+1$ in clockwise order.
Consider the tagged triangulation $T_0$ formed by:
\begin{itemize}[leftmargin=.2in]
\item two radii $d_1$ and $d_{2}$ (tagged plain and notched)
connecting the vertex~$1$ with the puncture~$p$;
\item plain arcs $d_3, d_4, \dots, d_{n+1}$ connecting 
the vertex~$1$ with vertices $2,3, \dots, n$, respectively, as shown in \cref{fig:D-triangulation}.
\end{itemize}

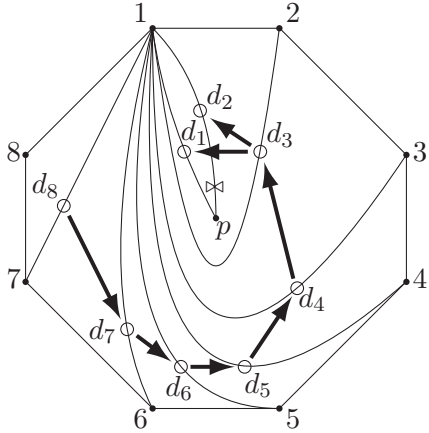
\begin{figure}[ht]
\begin{center}
\setlength{\unitlength}{2.8pt}
\begin{picture}(60,60)(0,0)
\thinlines
  \multiput(0,20)(60,0){2}{\line(0,1){20}}
  \multiput(20,0)(0,60){2}{\line(1,0){20}}
  \multiput(0,40)(40,-40){2}{\line(1,1){20}}
  \multiput(20,0)(40,40){2}{\line(-1,1){20}}

  \multiput(20,0)(20,0){2}{\circle*{1}}
  \multiput(20,60)(20,0){2}{\circle*{1}}
  \multiput(0,20)(0,20){2}{\circle*{1}}
  \multiput(60,20)(0,20){2}{\circle*{1}}

\put(30,30){\circle*{1}}
\qbezier(20,60)(30,50)(30,30)
\qbezier(20,60)(22,45)(30,30)
\put(29.8, 35){\makebox(0,0){$\notch$}}

\qbezier(20,60)(30,-15)(40,60)
\qbezier(60,40)(20,-20)(20,60)
\qbezier(60,20)(15,-20)(20,60)
\qbezier(40,0)(10,0)(20,60)
\qbezier(20,0)(10,20)(20,60)
\qbezier(0,20)(10,40 )(20,60)

{\put(25,40.5){\circle{2}}}
{\put(24.5,42.3){$d_1$}}
{\put(30,27.5){$p$}}

{\put(27.5,47){\circle{2}}}
{\put(28.5,48){$d_2$}}

{\put(37,40.5){\circle{2}}}
{\put(38,42){$d_3$}}

{\put(42.8,19){\circle{2}}}
{\put(43,16){$d_4$}}

{\put(34.5,6.6){\circle{2}}}
{\put(34.5,2.5){$d_5$}}

{\put(24.5,6.6){\circle{2}}}
{\put(22,2){$d_6$}}

{\put(16,12.5){\circle{2}}}
{\put(10,11){$d_7$}}

{\put(6,32){\circle{2}}}
{\put(1,34){$d_8$}}

\thicklines

{\put(41,61){$2$}}
{\put(61,39){$3$}}
{\put(61,19){$4$}}
{\put(41,-3){$5$}}
{\put(17,-3){$6$}}
{\put(-3,19){$7$}}
{\put(-3,39){$8$}}
{\put(17,61){$1$}}

\linethickness{1.5pt}

\put(35,40.5){{\vector(-1,0){8.5}}}
\put(35.5,41.5){{\vector(-3,2){6.5}}}
\put(42.2, 20.5){{\vector(-1,4){4.6}}}
\put(35.5,8){{\vector(2,3){6.5}}}
\put(26,6.6){{\vector(1,0){7}}}
\put(17.5,11.5){{\vector(4,-3){5.5}}}
\put(6.8,30.4){{\vector(1,-2){8.2}}}

\end{picture}
\end{center}
\caption{A tagged triangulation $T_0$ with quiver $Q_0$ such that 
$B(Q_0)^G$ is the type $B_n$ exchange matrix;  
here $n=7$.}
\label{fig:D-triangulation} 
\end{figure}

The corresponding quiver $Q_0=Q(T_0)$ is an orientation of the type 
$D_{n+1}$ Dynkin diagram.
The group $G=\ZZ/2\ZZ$ acts on the vertices of $Q_0$ by exchanging 
the vertices $1$ and~$2$, and fixing all the other vertices.
It is easy to see that $Q_0$ is $G$-admissible, and moreover 
$B(Q_0)^G$ is the exchange matrix of type~$B_n$ 
from \cref{typeB}.

It remains to show that $Q_0$ is globally foldable. 
In addition to the action on the set $\{1, 2, \dots, n+1\}$
described above,
the group $G=\ZZ/2\ZZ$ acts on the set of tagged arcs in 
the punctured polygon~$\mathbf{P}_{n+1}^\bullet$
by fixing every non-radius arc, and toggling the tags of the radii. 
A transformation  \hbox{$\mu_{J_k} \circ \dots \circ \mu_{J_1}$} associated with 
a sequence $J_1, \dots, J_k$ of $G$-orbits in $\{1, 2, \dots, n+1\}$
corresponds to a sequence of flips associated with non-radius arcs
or pairs of ``parallel'' radii with different tagging. 
(The latter step replaces a pair of parallel radii inside a punctured digon 
by another such pair of radii -- the pair incident to the other vertex of the 
digon.)
Such a sequence transforms $T_0$ into another $G$-invariant tagged triangulation~$T$
of the punctured polygon~$\mathbf{P}_{n+1}^\bullet$
in which the labels $1$ and~$2$ are assigned to a pair of parallel radii. 
It is easy to see that the quiver associated to~$T$ is $G$-admissible
(cf., e.g., the right picture in \cref{fig:octagon-D8}),
and so $Q_0$ is globally foldable.
\end{proof}

\begin{lemma}
The cluster variables labeled by different $G$-orbits in a cluster
algebra of type $B_n$ (respectively, $C_n$) are not equal to each 
other.
\end{lemma}
\begin{proof}
If two cluster variables appear in the same cluster, then they are necessarily distinct.
We now consider two cluster variables which do not appear in the same cluster.

In type $C$, consider two pairs of centrally symmetric diagonals, 
or a diameter and a pair of centrally symmetric diagonals, or 
two diameters. By the observation in the first paragraph, we may assume that 
	at least of the diagonals/diameters cross over each other.
	In all three cases, there is an
octagon that contains the two $G$-orbits.  
Freezing all cluster variables except those corresponding to the diagonals of this octagon
yields a seed subpattern of type~$C_3$.  
In type~$C_3$, one can check
directly that each of 12 distinct $G$-orbits 
corresponds to a different cluster variable,
say by verifying that the Laurent expansions expressing these cluster variables in terms
of a particular initial cluster have different denominators. 

In type $B$, consider two pairs of radii in $\mathbf{P}_{n+1}^\bullet$,
or a pair of radii and an arc which is not a radius, or 
two arcs which are not radii.  In all three cases, the two $G$-orbits
lie inside a certain punctured quadrilateral, 
reducing the problem to the treatment of a seed subpattern of type~$B_3$. 
In type~$B_3$, there are $12$ different $G$-orbits;
using the same method as in type~$C_3$, we can verify that  
they correspond to $12$ distinct cluster variables.
\end{proof}

\begin{exercise}\label{ex:BC}
By enumerating $G$-orbits, verify that 
a cluster algebra of type $B_n$ or $C_n$ has $n^2+n$ cluster variables.
By enumerating $G$-invariant tagged triangulations and 
centrally symmetric triangulations, verify that 
a cluster algebra of type $B_n$ or $C_n$ has 
$\binom{2n}{n}$ clusters (or seeds). 
\end{exercise}



\begin{proposition}
\label{pr:Bn-vs-Cn}
The same seed pattern cannot be simultaneously of type~$B_n$ and of type~$C_n$
for $n \geq 3$.
\end{proposition}

\begin{proof}
By Lemma \ref{lem:vector}, if the diagram of an 
exchange
matrix is connected, then its skew-symmetrizing vector
is unique up to rescaling.  
By Exercise \ref{ex:mat-mut-simple}, the skew-symmetrizing
vector is preserved under mutation. 
It remains to note that the 
skew-symmetrizing vectors for the exchange matrices of types~$B_n$ and~$C_n$ 
are $(1,2,2, \dots, 2)$ and $(2,1,1, \dots, 1)$,
respectively (up to rescaling and permuting the entries).
\end{proof}

Examples of coordinate rings having  natural cluster algebra structures 
of types $B_n$ and~$C_n$
will be given in 
Section~\ref{sec:models-classical}. 


\section{Seed patterns of types $E_6$, $E_7$, $E_8$}
\label{sec:exceptional}

In this section, we describe a computer-assisted proof of the statement that the cluster
algebras (or seed patterns) of exceptional types $E_6$, $E_7$, $E_8$
are of finite type.
The proof utilizes one of several software 
packages for computing with cluster algebras, freely available online.  
Our personal
favorites are 
 the \texttt{Java} applet
~\cite{Keller-aplet}
and the \texttt{Sage} package
~\cite{Musiker-Stump},
cf.\ the links at~\cite{Portal}. 
Among other things, 
both the applet and the \texttt{Sage} package allow one to 
compute seeds and Laurent expansions of 
cluster variables 
obtained by applying a sequence of mutations to a given initial seed.



\begin{theorem}\label{thm:E}
Seed patterns 
of types $E_6$, $E_7$, and $E_8$ are of finite type.
\end{theorem}

\begin{proof}
It suffices to verify that a seed pattern of 
type $E_8$ is of finite type. 
A~pattern of type~$E_6$ or~$E_7$ can be viewed as a 
subpattern of a pattern of type~$E_8$,
so if the latter has finitely many seeds, then so does the former. 

\pagebreak[3]

The main part of the proof is a verification that a cluster 
algebra of type $E_8$ \emph{with trivial coefficients} has finitely many seeds. 
(The case of general coefficients will follow easily, see below.) 
With the \texttt{Sage} package~\cite{Musiker-Stump}, this is done as follows.
The \texttt{Sage} command 

\begin{tabular}{l}
\texttt{{S24}} = \texttt{{ClusterSeed([[0,1],[1,2],[2,3],[4,5],[5,6],[6,7],}}\\
\qquad\quad \texttt{{[0,4],[1,5],[2,6],[3,7],[5,0],[6,1],[7,2]]);}}
\end{tabular}

\noindent
defines a seed \texttt{S24} with the quiver
shown in Figure~\ref{fig:2x4-grid}. 
Recall that by Exercise~\ref{exercise:D4-E8}, this quiver is 
mutation equivalent to any orientation of a Dynkin
  diagram of type~$E_8$.
%
%
Next, the \texttt{Sage} command 

\begin{tabular}{l}
\texttt{VC = S24.variable\underbar{\ }class(ignore\underbar{\ }bipartite\underbar{\ }belt=True);}
\end{tabular}

\noindent
performs an exhaustive depth-first search to find all 
seeds that can be obtained from \texttt{S24} using $\le N$ mutations,
for $N=0, 1, 2, 3, \dots$; the cluster variables 
appearing in these seeds are recorded in the list~\texttt{VC}. 
As soon as the calculation stops, the finiteness of the seed pattern is thereby established.
Then the command 

\begin{tabular}{l}
\texttt{len(VC);} 
\end{tabular}

\noindent
produces the output 

\begin{tabular}{l}
\texttt{128}
\end{tabular}

\noindent
which is the total number of cluster variables in the pattern. 
These 128 cluster variables, or more precisely their Laurent expansions in terms of the chosen initial seed,
can be displayed by executing the commands 

\begin{tabular}{l}
\texttt{for k in range(128): print(VC[k]); print("...");}
\end{tabular}

\noindent
To get a better idea of what happens in the course of the depth-first search,
one can run the command 

\begin{tabular}{l}
\texttt{SC=S24.mutation\underbar{\ }class(show\underbar{\ }depth=True,$\,$return\underbar{\ }paths=True);}
\end{tabular}

\noindent
its output will show how many seeds have been obtained after each stage. 
These data are recorded in Figure~\ref{fig:2x4-grid}. 
We see that no new seeds are found for $N=14$. 
The total number of seeds is $25080$.   

\begin{figure}[ht]
\begin{center}
\setlength{\unitlength}{2.8pt}
\begin{picture}(40,16)(0,-3)
\put(0,0){\circle*{1.5}}
\put(10,0){\circle*{1.5}}
\put(20,0){\circle*{1.5}}
\put(30,0){\circle*{1.5}}
\put(0,10){\circle*{1.5}}
\put(10,10){\circle*{1.5}}
\put(20,10){\circle*{1.5}}
\put(30,10){\circle*{1.5}}
\put(0,-3){\makebox(0,0){$0$}} 
\put(0,13){\makebox(0,0){$4$}} 
\put(10,-3){\makebox(0,0){$1$}} 
\put(10,13){\makebox(0,0){$5$}} 
\put(20,-3){\makebox(0,0){$2$}} 
\put(20,13){\makebox(0,0){$6$}} 
\put(30,-3){\makebox(0,0){$3$}} 
\put(30,13){\makebox(0,0){$7$}} 

\thicklines
\put(2,0){\vector(1,0){6}}
\put(2,10){\vector(1,0){6}}
\put(12,0){\vector(1,0){6}}
\put(12,10){\vector(1,0){6}}
\put(22,10){\vector(1,0){6}}
\put(22,0){\vector(1,0){6}}
\put(0,2){\vector(0,1){6}}
\put(10,2){\vector(0,1){6}}
\put(20,2){\vector(0,1){6}}
\put(30,2){\vector(0,1){6}}
\put(8,8){\vector(-1,-1){7}}
\put(18,8){\vector(-1,-1){7}}
\put(28,8){\vector(-1,-1){7}}
\end{picture}
\ \\[.25in] 
\small
\begin{tabular}{|c|c|c|c|c|c|c|c|c|c|c|c|c|c|c|c|}
\hline
$\!0\!$ & $\!1\!$ & $2$ & $3$ & $4$ & $5$ & $6$  & $7$  & $8$  & $9$  & $10$  & $11$  & $12$  & $\ge 13$ \\
\hline
$\!1\!$ & $\!9\!$ & $\!50\!$ & $\!196\!$ & $\!614\!$ & $\!1582\!$ & $\!3525\!$ & $\!6863\!$  
& $\!11626\!$ & $\!17098\!$ & $\!21706\!$ & $\!24220\!$ & $\!24974\!$ & $\!25080\!$  \\
\hline
\end{tabular}
\end{center}
\caption{Top: the triangulated grid quiver of type~$E_8$.
Bottom: the table showing, for each $N\ge 0$, 
the number of distinct seeds that can be obtained 
from the seed with this quiver using $\le N$ mutations.} 
\label{fig:2x4-grid}
\end{figure}

The proof of Theorem~\ref{thm:E} for the case of general coefficients
can now be completed using a standard
argument based on
Corollary~\ref{cor:z-span} and Remark~\ref{rem:full-z-rank}.   
One only needs to check that the $8 \times 8$ 
exchange matrix 
\[
B=
\left[\,\begin{matrix}
0  & 1 & 0 & 0  & 1 & -1 & 0 & 0 \\
-1  & 0 & 1 & 0  & 0 & 1 & -1 & 0 \\
0  & -1 & 0 & 1  & 0 & 0 & 1 & -1 \\
0  & 0 & -1 & 0  & 0 & 0 & 0 & 1 \\
-1  & 0 & 0 & 0  & 0 & 1 & 0 & 0 \\
1  & -1 & 0 & 0  & -1 & 0 & 1 & 0 \\
0  & 1 & -1 & 0  & 0 & -1 & 0 & 1 \\
0  & 0 & 1 & -1  & 0 & 0 & -1 & 0 
\end{matrix}\,
\right].
\]
associated to the quiver in Figure~\ref{fig:2x4-grid} 
has full $\ZZ$-rank. 
\end{proof}

\pagebreak[3]

\begin{remark}
The reader may be wondering: why we chose as the initial quiver 
the triangulated grid quiver in Figure~\ref{fig:2x4-grid},
rather than an orientation of a Dynkin diagram of type~$E_8$? 
The answer is that a straightforward implementation of the latter strategy appears to be 
computationally infeasible. \linebreak[3] 
A~seed pattern of type~$E_8$ has $128$ cluster variables.  
The formulas expressing them as Laurent polynomials in the $8$~initial cluster variables 
are recursively computed in the process of the depth-first search,
and then used to compare the seeds to each other. 
When the triangulated grid quiver is chosen as the initial one, these Laurent polynomials
turn out to be quite manageable. \linebreak[3]
As a result, the entire calculation took less than one hour on a MacBook Pro laptop 
computer (manufactured in~2013), with a 2.6 GHz processor and 8~GB RAM.  
By comparison, a similar calculation using, as 
the initial quiver, the Dynkin diagram of type~$E_8$ with an alternating orientation
(i.e., each vertex is either a source or a sink) did not terminate within a few days. 
The explanation likely lies in the fact that the Laurent polynomials expressing 
some of the 128 cluster variables in terms of the initial ones are extremely cumbersome
in this case. 
To get an idea of the size of these Laurent polynomials 
(which are known to have positive coefficients), 
one can specialize the $8$~initial variables to~1, and compute the remaining 120 cluster variables 
recursively. 
When the initial quiver is the alternating Dynkin quiver of type~$E_8$, 
the largest of these specializations turns out to be equal to 2820839; 
for the triangulated grid quiver, the corresponding value is~107. 
\end{remark}

\begin{remark}
One naturally arising cluster algebra of type~$E_8$ 
is the homogeneous coordinate
ring of the Grassmannian $\Gr_{3,8}$ of $3$ planes in $8$-space. 
Another closely related example is the coordinate ring of 
the affine space of $3\times 5$ matrices.
(See Chapter~\ref{ch:rings} and Chapter~\ref{ch:Grassmannians} for more
details.) 
Each of these constructions can in principle be used, with or without
a computer,  
to verify that seed patterns of type~$E_8$ are of finite type.
\end{remark}

\pagebreak[3]

\section{Seed patterns of types $F_4$ and $G_2$}
\label{sec:exceptional2}

We will now use folding to take
care of types $F_4$ and~$G_2$.


By \cref{ex:E6F4}, one can realize a type $F_4$ exchange
matrix as~$B(Q)^G$, where $Q$ is an orientation of the type $E_6$ 
Dynkin diagram, $G = \ZZ/2\ZZ$, and $Q$ is globally foldable. 
Now \cref{cor:folding-finite-type}, together with the fact that 
type~$E_6$ cluster algebras are of finite type,
implies that the same is true in type~$F_4$.

An alternative approach
is to use a computer to check directly that 
a cluster algebra of type $F_4$ (with no frozen variables) has finitely many seeds
(there are $105$ of them).  
We can then use our standard argument 
based on Remark~\ref{rem:full-z-rank}, 
together with the fact that the matrix $\tilde B^G$ from 
Figure~\ref{fig:E6-F4} has full $\ZZ$-rank, to complete
the proof. 

We now turn to cluster algebras of type~$G_2$. 
While it follows from the results of \cref{sec:finite-type-rank2} that
cluster algebras of type~$G_2$ are of finite type, 
this result can also be obtained via folding of the type~$D_4$ quiver   
shown in Figure~\ref{fig:D4-G2}. 
It is easy to see that this quiver is $G$-foldable, 
with respect to the natural action of $G=\ZZ/3\ZZ$. 
Since every seed pattern of type $D_4$ is of finite type, 
Corollary~\ref{cor:folding-finite-type} implies the same for the type~$G_2$.

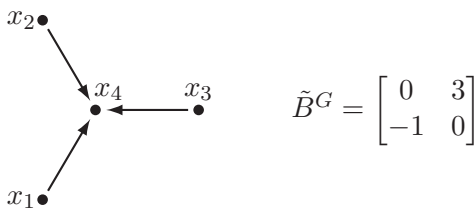
\begin{figure}[ht]
\setlength{\unitlength}{1pt} 
\begin{picture}(210,68)(-30,0) 
\thicklines 
  \put(2,3.4){\vector(10,17){16}} 
  \put(2,64.6){\vector(10,-17){16}} 

\put(55,34){\vector(-1,0){31}}
  \put(20,34){\circle*{4}} 
  \put(59,34){\circle*{4}} 
  \put(0,0){\circle*{4}} 
  \put(0,68){\circle*{4}} 

\put(-8,0){\makebox(0,0){$x_1$}}
\put(-8,68){\makebox(0,0){$x_2$}}
\put(25,41){\makebox(0,0){$x_4$}}
\put(59,41){\makebox(0,0){$x_3$}}

\put(130,34){\makebox(0,0){$\tilde B^G=\begin{bmatrix}
0 & 3 \\
-1 & 0
\end{bmatrix}$}}

\end{picture} 

\caption{The generator of the group $G=\ZZ/3\ZZ$ acts on the vertices 
of the quiver shown on the left by sending $1\mapsto 2\mapsto 3\mapsto
1$ and $4\mapsto 4$. 
All vertices are mutable.  
The rows and columns of the matrix $\tilde B^G$ are indexed by the $G$-orbits
$\{1,2,3\}$ and $\{4\}$.}
\label{fig:D4-G2}
\end{figure}

\section{Decomposable types} 
\label{sec:decomposable-types}

We refer to the disjoint union of two Dynkin diagrams of types 
$X_n$ and~$Y_{n'}$ as a Dynkin diagram of type 
$X_n \sqcup Y_{n'}$; and similarly for disjoint unions of three or 
more Dynkin diagrams.

We have now defined cluster algebras of types 
$A_n$, $B_n$, $C_n$, $D_n$, $E_6$, $E_7$, $E_8$, 
$F_4$, and $G_2$, corresponding to the indecomposable Cartan matrices
of finite type, or equivalently, the connected Dynkin diagrams.  
A seed pattern (or cluster algebra) of rank $n+n'$ is said to be of type 
$X_n \sqcup Y_{n'}$ if one of its exchange matrices
is (up to a simultaneous permutation of rows
and columns) a block-diagonal matrix with blocks whose Cartan counterparts
are of types $X_n$ and~$Y_{n'}$.  We have the following simple lemma.

\begin{lemma} \label{lem:decomposable} 
If $\mathcal{A}$ is 
a cluster algebra of 
decomposable type $X_n\sqcup Y_{n'}$,
then the number of cluster variables is the sum of those 
of the cluster algebras of types $X_n$ and~$Y_{n'}$,
and the number of clusters is the product. 
\end{lemma}

\begin{proof}
Suppose that $\mathcal{A}$ is of type $X_n \sqcup Y_{n'}$.  
Then (after a simultaneous permutation of rows and columns)
one of its exchange matrices $B$ is block-diagonal with blocks whose
Cartan counterparts are of types $X_n$ and $Y_{n'}$.
Then the labeled seeds of 
$\mathcal{A}$ are all 
pairs of the form 
$(\xx^1 \sqcup  \xx^{2}, B^1 \sqcup  B^{2})$,
where each $(\xx^i, B^i)$ is a labeled seed of the cluster algebra
associated to the $i$th block of $B$,
$\xx^1 \sqcup \xx^{2}$ denotes the concatenation
of the clusters $\xx^1$ and $\xx^{2}$, 
and $B^1 \sqcup  B^{2}$ denotes the block-diagonal 
matrix with blocks $B^1$ and $B^{2}$.
The statement of the lemma follows.
\end{proof}

We are now ready to complete 
the proof of the ``if'' direction of 
Theorem~\ref{th:finite-type-classification}, restated below
for the convenience of the reader.

\begin{corollary}
\label{if:main}
If the Cartan counterpart of an exchange matrix of a seed pattern
(or a cluster algebra)
is a Cartan matrix of finite type, then the
seed pattern is of finite
type.
\end{corollary}

\begin{proof}
Putting together Theorems~\ref{th:type-A-is-finite},
\ref{th:type-D-is-finite},
\ref{th:type-B-is-finite},
\ref{th:type-C-is-finite}, and 
the results of Sections~\ref{sec:exceptional}--\ref{sec:exceptional2},
we conclude that if the Cartan counterpart $A=A(B)$ of 
one of the exchange matrices $B$ associated to 
a seed pattern is an {indecomposable} Cartan matrix of 
finite type, then the seed pattern has finite type.
In the general (decomposable) case,
 the same conclusion follows from 
Lemma \ref{lem:decomposable}.
\end{proof}

\section{Enumeration of clusters and cluster variables}
\label{sec:type}

In this section we present formulas for 
the number
of cluster variables and clusters in each cluster algebra
of finite type. These enumerative invariants provide a way 
to distinguish cluster algebras of different types, leading to a 
proof of 
\cref{thm:type}.


One important property of finite type cluster algebras is that 
the underlying combinatorics does not depend on the choice of coefficient tuple.
(Conjecturally this holds for arbitrary cluster algebras,
see Section~\ref{sec:exchangegraph}.)
We have seen that in cluster algebras of type~$A_n$, seeds
are in bijection with triangulations of an $(n+3)$-gon, regardless 
of the choice of coefficient tuple.  Similarly, in cluster algebras
of type~$D_n$, seeds are in bijection with tagged triangulations
of a punctured $n$-gon, for any choice of  coefficients.
Seeds of cluster algebras of types~$C_n$ and $B_n$ are in bijection 
with folded triangulations and tagged triangulations, respectively.
For cluster algebras of exceptional types other than~$E_7$,
the exchange matrices all have full $\ZZ$-rank. 
It then follows from 
Corollary~\ref{cor:z-span} and Remark~\ref{rem:full-z-rank}
that the combinatorics of seeds is independent of the choice
of coefficients: if an exchange matrix $B(t)$ has full 
$\ZZ$-rank, and $\tilde B(t)$ is obtained from $B(t)$ by 
appending some additional rows, then the rows of $\tilde B(t)$ 
lie in the $\ZZ$-span of the rows of $B(t)$, and also 
the rows of $B(t)$ lie in the $\ZZ$-span of the rows of 
$\tilde B(t)$.  Therefore Corollary~\ref{cor:z-span}
implies that the seeds of the seed patterns associated with $B(t)$
and with $\tilde B(t)$ are in bijection with each other.

In type~$E_7$, one needs to add one additional row to the exchange matrix $B$ in 
order to obtain 
an exchange matrix $\tilde B$ of full $\ZZ$-rank.  One can then check 
(for example by computer) that the seeds  of the seed patterns associated
with $B$ and~$\tilde B$, respectively, are in bijection.
The same argument as before implies that the underlying combinatorics of 
any cluster algebra of type~$E_7$ does not depend on the choice of coefficient tuple.


\begin{proposition}
\label{pr:number-of-clusters}
Let $X_n$ be a connected Dynkin diagram.
The numbers of seeds 
and cluster variables 
in a seed pattern of type~$X_n$
are given by
the values in the corresponding column of the table in \cref{tab:cluster-numbers}.
Alternatively, let $\Phi$ be an irreducible finite crystallographic 
root system of type~$X_n$. 
Then 
\begin{align}
\label{eq:Seeds}
\Seeds(X_n) & = \prod_{i=1}^n \frac{e_i+h+1}{e_i+1}, \\
\label{eq:Vars}
\Vars(X_n) & = \frac{n(h+2)}{2},  
\end{align}
where $e_1,\dots,e_n$ are the \emph{exponents}
of~$\Phi$, 
and $h$ is the corresponding \emph{Coxeter number}.  
\end{proposition}

\begin{figure}[ht]
\begin{center}
\begin{tabular}{|c|c|c|c|c|c|c|c|c|c|}
\hline
$X_n$ & $A_n$ & $B_n,C_n$ & $D_n$ & $E_6$ & $E_7$ & $E_8$ & $F_4$  & $G_2$ \\
\hline
&&&&&&&&\\[-.1in]
$\!\Seeds (X_n)\!$ & 
$\!\frac{1}{n+2} \binom{2n+2}{n+1}\!$
& $\binom{2n}{n}$ 
&   $\!\frac{3n-2}{n}\binom{2n-2}{n-1}\!$ 
&
\!833\!
&
\!4160\!
&
\!25080\!
&
\!105\!
& $\!8\!$ \\[.05in]
\hline
$\!\Vars (X_n)\!$ & $\frac{n(n+3)}{2}$ & $\!n(n+1)\!$ & $n^2$ & $42$ & $70$ & $128$ & $28$ & $\!8\!$  \\
\hline
\end{tabular}
\end{center}
\medskip
\caption{Enumeration of seeds and cluster variables}
\label{tab:cluster-numbers}
\end{figure}

\vspace{-.1in}

\begin{proof}
The values in \cref{tab:cluster-numbers} can be verified case by case. 
The types  $A_n$, $B_n$/$C_n$, and $D_n$ were 
worked out in \cref{cor:An-structure}, \cref{ex:BC}, and 
\cref{cor:Dn-structure}, respectively.
Exceptional types can be handled using  
the software packages discussed in 
\cref{sec:exceptional}. 
It is then straightforward to check that the formulas \eqref{eq:Seeds}--\eqref{eq:Vars}
match the values shown in Figure~\ref{tab:cluster-numbers}.
(See, e.g., \cite{bourbaki, pcmi} for the values of 
the exponents and Coxeter numbers for all types.) 
\end{proof}

\begin{remark}
Since the number 
$\frac{1}{n+2} \binom{2n+2}{n+1}$
of seeds in type $A_n$ is a Catalan number,
the numbers in \eqref{eq:Seeds}  can be regarded as generalizations of the
Catalan numbers to arbitrary Dynkin diagrams.
\end{remark}

\begin{remark}
The number of cluster variables is alternatively given by
\[
\Vars(X_n) = \textstyle\frac{\#\roots(X_n)}{2}+n,
\]
where $\#\roots(X_n)$ denotes the number of roots in the root system~$\Phi$ of type~$X_n$.
Thus cluster variables are equinumerous to the roots which are either positive
or negative simple (i.e., the negatives of simple roots).
A~natural labeling of the cluster variables by these ``almost positive'' roots 
was described and studied in~\cite{yga}. 
\end{remark}

The reader is referred to~\cite{pcmi} for a~detailed discussion of cluster combinatorics of finite type, 
and further references. 

Recall from Lemma \ref{lem:decomposable} that if $\mathcal{A}$ is 
a cluster algebra of 
decomposable type $X_n\sqcup Y_{n'}$,
then the number of cluster variables is the sum of those 
for the cluster algebras of types $X_n$ and~$Y_{n'}$,
and the number of seeds is the product. 
Therefore Proposition~\ref{pr:number-of-clusters} 
allows us to compute the number of cluster variables and seeds
for any cluster algebra of finite type.

\begin{proof}[Proof of \cref{thm:type}]
The implication (1)$\Rightarrow$(2) is easy to establish. 
Suppose that the Cartan counterparts $A(B')$ and $A(B'')$ 
are Cartan matrices of the same finite type.  
In the case of simply laced types $ADE$, this means that the corresponding quivers
are (possibly different) orientations of isomorphic Dynkin diagrams. 
By  Exercise~\ref{ex:orientations-of-a-tree}, these two quivers are related to each other
by a sequence of mutations at sources and sinks, 
and consequently $B'$ and $B''$ are mutation equivalent.
The remaining cases $BCFG$ are treated in a similar fashion,
using an appropriate analogue of Exercise~\ref{ex:orientations-of-a-tree}. 


Let us prove the implication (2)$\Rightarrow$(1). 
We first observe that it suffices to establish this result in the indecomposable case,
since mutations transform the connected components of a quiver 
(or their analogues for skew-symmetrizable matrices) 
independently of each other. 

Let $B'$ and $B''$ be mutation equivalent exchange matrices 
of types $X'$ and~$X''$, respectively, 
where $X'$ and $X''$ are connected Dynkin diagrams.
We need to show that $X'$ and $X''$ are of the same type. 
This is done as follows. 
By Proposition~\ref{pr:number-of-clusters}, 
the number of cluster variables in a seed pattern associated with such an exchange matrix 
is uniquely determined by the type of the corresponding Dynkin diagram
(i.e., it does not depend on the choice of coefficient tuple). 
Moreover no two connected
Dynkin types of the same rank
produce the same number of cluster variables---with the exception of the pairs $(B_n,C_n)$.
For $B_n$ versus $C_n$, the claim follows from
Proposition~\ref{pr:Bn-vs-Cn}. 
\end{proof}

We conclude this section by an elementary observation 
that shows that the problem of enumerating cluster variables 
does not make sense outside of finite type. 

\begin{proposition}
\label{cor:finitely-many-seeds}
A seed pattern is of finite type if and only if
it has finitely many cluster variables. 
\end{proposition}

\begin{proof}
One direction is obvious:  
if a seed pattern has finitely many distinct seeds, then it has 
finitely many cluster variables.  Conversely, if a seed pattern
has finitely many cluster variables, then the only way it could
possibly have infinitely many distinct seeds is if there
were infinitely many distinct extended exchange matrices.  
By the Pigeonhole principle, one cluster would have to appear with 
infinitely many different extended exchange matrices $\tilde B$, which implies
that one of its cluster variables $x_j$ would appear in infinitely
many exchange relations, leading to infinitely many different $x'_j$'s.
\end{proof}

\pagebreak[3]

\section{
$2$-finite exchange matrices}
\label{sec:2-finite}

In this section, we complete the proof of
Theorem~\ref{th:finite-type-classification}, closely following~\cite{ca2}. 
The notion of a $2$-finite (skew-symmetrizable) matrix
introduced in Definition~\ref{def:2-finite} plays a key role. 

We establish Theorem~\ref{th:finite-type-classification} by including it 
in the following statement.  

\begin{theorem}
\label{th:finite-type-bound-by-3}
For a seed $\Sigma=(\xx, \yy, B)$,
the following are equivalent:
\begin{enumerate}[leftmargin=.3in]
\item[{\rm (1)}]
there exists a matrix $B'$ mutation equivalent to~$B$ such that
its Cartan counterpart $A(B')$ is a Cartan matrix of finite type;
\item[{\rm (2)}]
the seed pattern (or the cluster algebra) defined by~$\Sigma$ is of
finite~type; 
\item[{\rm (3)}]
the exchange matrix $B$ is $2$-finite.
\end{enumerate}
%
\end{theorem}


The implication $\boxed{(1) \Rightarrow(2)}$ in Theorem~\ref{th:finite-type-bound-by-3} 
is nothing but \cref{if:main}.
To the best of our knowledge, the only known proof of the reverse
implication $(2) \Rightarrow (1)$ 
goes through property~(3). 

The implication $\boxed{(2) \Rightarrow (3)}$ 
is precisely Corollary~\ref{cor:finite-type-2finite}. 

All that remains in order to complete the proof of
Theorem~\ref{th:finite-type-bound-by-3} 
(hence Theorem~\ref{th:finite-type-classification})
is to prove the implication
$\boxed{(3)\Rightarrow(1)}$. 
We reformulate the latter below as a standalone statement. 



\begin{proposition}
\label{pr:2-fin-type-class}
Let $B=(b_{ij})$ be a $2$-finite
skew-symmetrizable integer matrix. 
Then there exists a matrix $B'$ mutation equivalent to~$B$ such that
its Cartan counterpart $A(B')$ is a Cartan matrix of finite type.
\end{proposition}



In the rest of this section, we outline the proof of
Proposition~\ref{pr:2-fin-type-class} given in \cite[Sections~7--8]{ca2}. 
The proof is purely combinatorial and rather technical. 
The missing details (all of them relatively minor)
can be found in \emph{loc.\ cit.}

\pagebreak[3]

The proof of
Proposition~\ref{pr:2-fin-type-class} makes heavy use of the notion of 
diagram from Definition \ref{def:diagramofB}.  Recall from 
Proposition~\ref{pr:diagram-mutation} that mutation is well-defined for 
diagrams, and we write $\Gamma \sim \Gamma'$ to denote that 
diagrams $\Gamma$ and $\Gamma'$ are mutation equivalent.

A diagram $\Gamma$ is called
\emph{$2$-finite} if every diagram $\Gamma'\sim\Gamma$
has all edge weights equal to $1,2$ or~$3$;
otherwise we refer to $\Gamma$ as \emph{$2$-infinite.}
Thus a matrix $B$ is $2$-finite if and only if its diagram $\Gamma(B)$ is $2$-finite.
Note that a diagram is $2$-finite if and only if so are all its
connected components.

We now restate Proposition~\ref{pr:2-fin-type-class}
in the language of diagrams. 

\begin{proposition}
\label{pr:diagrams-fin-CK}
Any  $2$-finite diagram is mutation equivalent to
an orientation of a Dynkin diagram (where the weight of each edge
is understood as its multiplicity in the corresponding
Dynkin diagram).
\end{proposition}

We note that all orientations of a given Dynkin diagram 
are mutation equivalent to each other, 
as they are related by
source-or-sink mutations as in the proof of
Theorem~\ref{thm:type}.
This is true more generally for any diagram whose underlying graph is
a tree. 

A~\emph{subdiagram} of a diagram $\Gamma$ is a diagram $\Gamma'\subset\Gamma$
obtained by taking an induced directed subgraph of $\Gamma$ and 
keeping all its edge weights intact.

The proof of Proposition~\ref{pr:diagrams-fin-CK} 
repeatedly makes use of the following obvious property:
any subdiagram of a $2$-finite diagram is $2$-finite.
Equivalently, any diagram that has a $2$-infinite subdiagram is $2$-infinite.
Thus, in order to show that a given diagram is
$2$-infinite, it suffices to exhibit a sequence of mutations that
creates an edge of weight~$4$ or larger, or a subdiagram which
is already known to be $2$-infinite.
The strategy is to catalogue enough $2$-infinite subdiagrams 
to be able to show that any diagram
avoiding them has to be mutation equivalent to
an orientation of a Dynkin diagram.

We first examine two special classes of diagrams, those whose underlying graphs are
trees or cycles, respectively.
We refer to them as \emph{tree diagrams} and
\emph{cycle diagrams}, respectively.


\begin{proposition}
\label{pr:trees-Dynkin-vs-infty}
Any $2$-finite tree diagram is an orientation of a connected Dynkin diagram.
\end{proposition}

For the proof we consider a class of diagrams defined as follows.
A diagram $\Gamma$ is called an \emph{extended Dynkin tree diagram} if
\begin{itemize}[leftmargin=.2in]
\item
$\Gamma$ is a tree diagram with edge weights $\leq 3$;
\item
$\Gamma$ is not on the Dynkin diagram list;
\item
any proper subdiagram of $\Gamma$ is a Dynkin diagram (possibly disconnected).
\end{itemize}
To prove the proposition, it is enough to show that any orientation
of any extended
Dynkin tree diagram is $2$-infinite.
Direct inspection shows that Figure~\ref{fig:extended-dynkin-diagrams}
provides a complete list of such diagrams (as discussed above, the
choice of an orientation for a tree diagram is immaterial).
We note that all these diagrams are associated with untwisted affine Lie algebras
and can be found in the tables in \cite{bourbaki} or in
\cite[Chapter 4, Table Aff~1]{kac}.
The only diagram from those tables that is missing in
Figure~\ref{fig:extended-dynkin-diagrams} is $A_n^{(1)}$, which is
an $(n+1)$-cycle; it will appear later in our discussion of
cycle diagrams.

\begin{figure}[ht]
\vspace{-.1in}
\[
\begin{array}{ccl}
B_n^{(1)}
&&
\setlength{\unitlength}{1.5pt}
\begin{picture}(170,15)(0,-2)
\put(20,0){\line(1,0){120}}
\put(0,10){\line(2,-1){20}}
\put(0,-10){\line(2,1){20}}
\multiput(20,0)(20,0){7}{\circle*{2}}
\put(0,10){\circle*{2}}
\put(0,-10){\circle*{2}}
\put(130,4){\makebox(0,0){$2$}}
\put(165,0){\makebox(0,0){($n\ge 3$)}}
\end{picture}
\\[.2in]
C_n^{(1)}
&&
\setlength{\unitlength}{1.5pt}
\begin{picture}(140,17)(0,-2)
\put(0,0){\line(1,0){140}}
\multiput(0,0)(20,0){8}{\circle*{2}}
\put(10,4){\makebox(0,0){$2$}}
\put(130,4){\makebox(0,0){$2$}}
\put(165,0){\makebox(0,0){($n\ge 2$)}}
\end{picture}
\\[.1in]
D_n^{(1)}
&&
\setlength{\unitlength}{1.5pt}
\begin{picture}(140,17)(0,-2)
\put(20,0){\line(1,0){100}}
\put(0,10){\line(2,-1){20}}
\put(0,-10){\line(2,1){20}}
\put(120,0){\line(2,-1){20}}
\put(120,0){\line(2,1){20}}
\multiput(20,0)(20,0){6}{\circle*{2}}
\put(0,10){\circle*{2}}
\put(0,-10){\circle*{2}}
\put(140,10){\circle*{2}}
\put(140,-10){\circle*{2}}
\put(165,0){\makebox(0,0){($n\ge 4$)}}
\end{picture}
\\[.1in]
E_6^{(1)}
&&
\setlength{\unitlength}{1.5pt}
\begin{picture}(140,17)(0,-2)
\put(0,0){\line(1,0){80}}
\put(40,0){\line(0,-1){40}}
\put(40,-20){\circle*{2}}
\put(40,-40){\circle*{2}}
\multiput(0,0)(20,0){5}{\circle*{2}}
\end{picture}
\\[.7in]
E_7^{(1)}
&&
\setlength{\unitlength}{1.5pt}
\begin{picture}(140,17)(0,-2)
\put(0,0){\line(1,0){120}}
\put(60,0){\line(0,-1){20}}
\put(60,-20){\circle*{2}}
\multiput(0,0)(20,0){7}{\circle*{2}}
\end{picture}
\\[.25in]
E_8^{(1)}
&&
\setlength{\unitlength}{1.5pt}
\begin{picture}(140,17)(0,-2)
\put(0,0){\line(1,0){140}}
\put(40,0){\line(0,-1){20}}
\put(40,-20){\circle*{2}}
\multiput(0,0)(20,0){8}{\circle*{2}}
\end{picture}
\\[.3in]
F_4^{(1)}
&&
\setlength{\unitlength}{1.5pt}
\begin{picture}(140,17)(0,-2)
\put(0,0){\line(1,0){80}}
\multiput(0,0)(20,0){5}{\circle*{2}}
\put(30,4){\makebox(0,0){$2$}}
\end{picture}
\\[.1in]
G_2^{(1)}
&&
\setlength{\unitlength}{1.5pt}
\begin{picture}(140,17)(0,-2)
\put(0,0){\line(1,0){40}}
\multiput(0,0)(20,0){3}{\circle*{2}}
\put(10,4){\makebox(0,0){$3$}}
\put(30,4){\makebox(0,0){$a$}}
\put(75,0){\makebox(0,0){($a\in\{1,2,3\}$)}}
\end{picture}
\end{array}
\]
\vspace{-.1in}
\caption{Extended Dynkin tree diagrams. 
Each tree $X_n^{(1)}$ has $n+1$ vertices.
All unspecified edge weights are equal to~$1$. 
}
\label{fig:extended-dynkin-diagrams}
\end{figure}
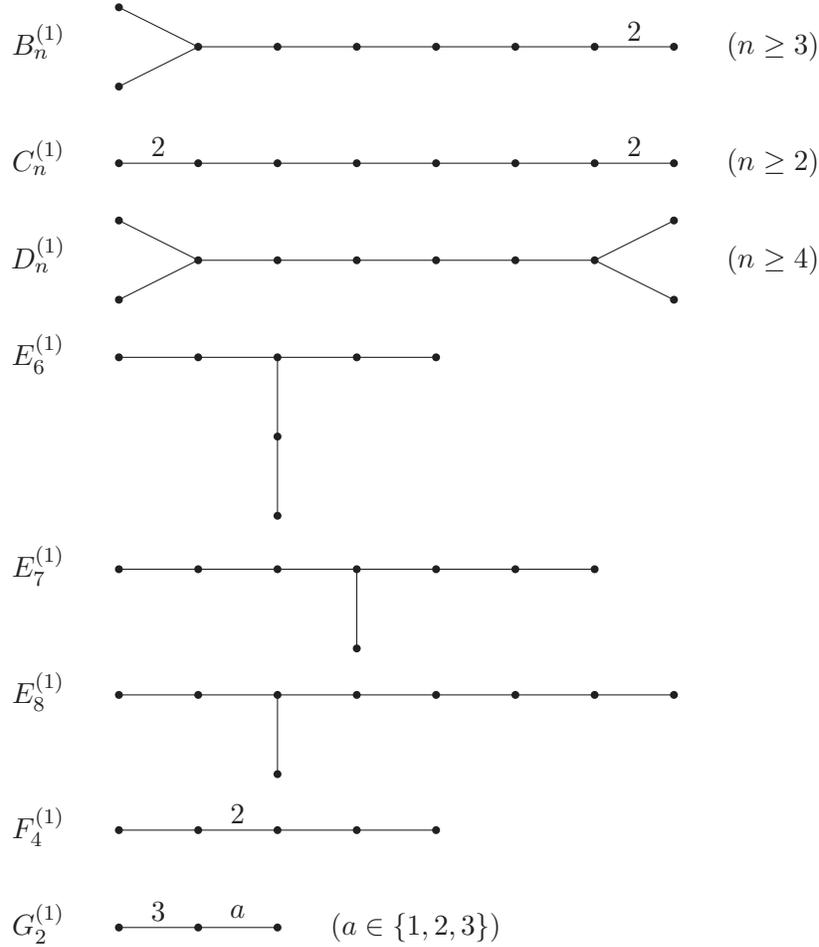

\pagebreak[3]

In showing that an extended Dynkin tree diagram is
$2$-infinite, we can arbitrarily choose its orientation. 
We start with the three infinite series $B_n^{(1)}, C_n^{(1)}$,
and $D_n^{(1)}$, each time orienting all the edges
left to right.
Let us denote the diagram in question by $X_n^{(1)}$,
and let $n_\circ$ be the minimal value~of~$n$. 
So if $X = D$ (resp., $B$,~$C$), then $n_\circ$ equals~$4$ 
(resp., $3$,~$2$).
If $n>n_\circ$, then
mutating at the second vertex from the left, and
subsequently removing this vertex (together with all incident~edges)
leaves us with a subdiagram of type~$X_{n-1}^{(1)}$.
Using induction on~$n$, it suffices to check the base cases
$D_4^{(1)}, B_3^{(1)}$ and~$C_2^{(1)}$.
It is not hard to check that each of these three diagrams is $2$-infinite. 
The same applies to extended Dynkin trees of types $F_4^{(1)}$
and~$G_2^{(1)}$. 

The remaining three cases $E_6^{(1)}, E_7^{(1)}$ and $E_8^{(1)}$
can be treated in a similar manner (with or without a computer)
but we prefer another approach.
To describe it, we will need to introduce some notation.

\pagebreak[3]

\begin{definition}
\label{def:Tabc}
For $p,q,r\in\ZZ_{\geq 0}$, we denote by
$T_{p,q,r}$ the tree diagram (with 
all edge weights equal to $1$) on $p+q+r+1$ vertices obtained by
connecting an endpoint of each of the three chains $A_p$, $A_q$ and~$A_r$
to a single extra vertex (see Figure~\ref{fig:Tabc}).
\end{definition}

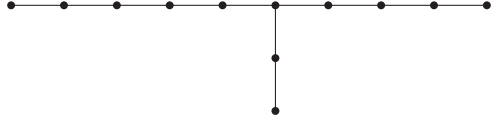
\begin{figure}[ht]
\setlength{\unitlength}{1pt}
\begin{picture}(180,43)(0,0)

\put(0,40){\line(1,0){180}}
\put(100,40){\line(0,-1){40}}

\multiput(0,40)(20,0){10}{\circle*{3}}
\multiput(100,20)(0,-20){2}{\circle*{3}}

\end{picture}
\caption{The tree diagram $T_{5,4,2}\,$.}
\label{fig:Tabc}
\end{figure}

\begin{definition}
\label{def:Spqrs}
For $p,q,r \in \ZZ_{> 0}$ and $s \in \ZZ_{\geq 0}$,
let $S_{p,q,r}^s$ denote the diagram (with all edge weights equal
to $1$) on $p+q+r+s$ vertices obtained by attaching three branches
$A_{p-1}$, $A_{q-1}$, and $A_{r-1}$ to three consecutive vertices
on a \emph{cyclically oriented} $(s+3)$-cycle (see Figure~\ref{fig:Spqrs}).
\end{definition}

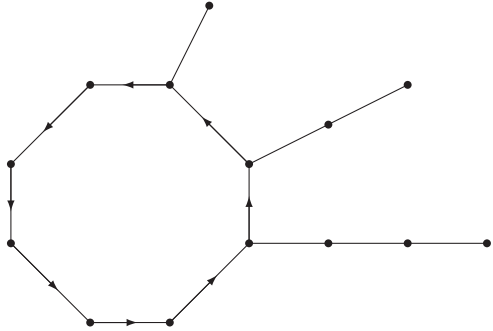
\begin{figure}[ht]
\setlength{\unitlength}{1.5pt}
\begin{picture}(60,83)(0,-2)
\put(20,0){\line(1,0){20}}
\put(20,60){\line(1,0){20}}
\put(0,20){\line(0,1){20}}
\put(60,20){\line(0,1){20}}
\put(0,20){\line(1,-1){20}}
\put(40,60){\line(1,-1){20}}
\put(0,40){\line(1,1){20}}
\put(40,0){\line(1,1){20}}

\put(20,0){\vector(1,0){12}}
\put(40,60){\vector(-1,0){12}}
\put(0,40){\vector(0,-1){12}}
\put(60,20){\vector(0,1){12}}
\put(0,20){\vector(1,-1){12}}
\put(60,40){\vector(-1,1){12}}
\put(20,60){\vector(-1,-1){12}}
\put(40,0){\vector(1,1){12}}

\multiput(20,0)(20,0){2}{\circle*{2}}
\multiput(20,60)(20,0){2}{\circle*{2}}
\multiput(0,20)(60,0){2}{\circle*{2}}
\multiput(0,40)(60,0){2}{\circle*{2}}
\put(50,80){\circle*{2}}

\multiput(80,20)(20,0){3}{\circle*{2}}
\multiput(80,50)(20,10){2}{\circle*{2}}
\put(40,60){\line(1,2){10}}
\put(60,20){\line(1,0){60}}
\put(60,40){\line(2,1){40}}


\end{picture}
\caption{The diagram $S_{4,3,2}^5\,$. 
}
\label{fig:Spqrs}
\end{figure}
\vspace{-.1in}

Note that in both definitions, the choice of orientations for the edges
where they are not shown is immaterial: different choices lead to
mutation-equivalent diagrams.
For~$T_{p,q,r}\,$, this follows from
Exercise~\ref{ex:orientations-of-a-tree};
for~$S_{p,q,r}^s$, 
one needs a slight generalization of this result, see
\cite[Proposition~9.2]{ca2}.


\begin{exercise}
\label{exercise:crown}
Show that the diagram $S_{p,q,r}^s$ is mutation equivalent to $T_{p+r-1,q,s}$.
\end{exercise}

With the help of Exercise~\ref{exercise:crown}, 
the proof of Proposition~\ref{pr:trees-Dynkin-vs-infty} can now be
completed, using the observations that 
\[
\begin{array}{l}
E_6^{(1)} = T_{2,2,2} \sim S_{2,2,1}^2 \supset D_5^{(1)} \,;\\
E_7^{(1)} = T_{3,1,3} \sim S_{3,1,1}^3 \supset E_6^{(1)} \,;\\
E_8^{(1)} = T_{2,1,5} \sim S_{2,1,1}^5 \supset E_7^{(1)} \,. 
\end{array}
\]

Turning to the cycle diagrams, we have the following
classification. 

\begin{exercise}
\label{exercise:cycles-Dynkin-vs-infty}
Let $\Gamma$ be a $2$-finite diagram whose underlying graph is an
$n$-cycle 
(with some orientation of edges).
Show that $\Gamma$ is cyclically oriented, and moreover
it must be one of the following
(see Figure~\ref{fig:3-4-cycles}):
\begin{itemize}[leftmargin=.2in]
\item[{\rm(a)}] an $n$-cycle with all weights equal
 to~$1$ (in this case, $\Gamma\sim D_n$); 
\item[{\rm(b)}] a $3$-cycle with
edge weights $2,2,1$ (in this case, $\Gamma\sim B_3$);
\item[{\rm(c)}] a $4$-cycle with
edge weights $2,1,2,1$ (in this case, $\Gamma\sim F_4$).
\end{itemize}
\end{exercise}

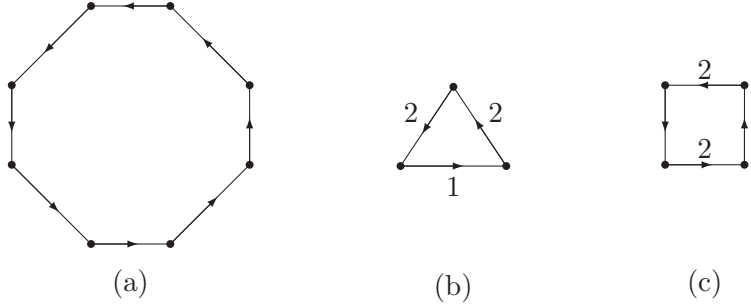
\begin{figure}[ht]
\setlength{\unitlength}{1.5pt}
\begin{picture}(60,65)(0,-7)
\put(20,0){\line(1,0){20}}
\put(20,60){\line(1,0){20}}
\put(0,20){\line(0,1){20}}
\put(60,20){\line(0,1){20}}
\put(0,20){\line(1,-1){20}}
\put(40,60){\line(1,-1){20}}
\put(0,40){\line(1,1){20}}
\put(40,0){\line(1,1){20}}

\put(20,0){\vector(1,0){12}}
\put(40,60){\vector(-1,0){12}}
\put(0,40){\vector(0,-1){12}}
\put(60,20){\vector(0,1){12}}
\put(0,20){\vector(1,-1){12}}
\put(60,40){\vector(-1,1){12}}
\put(20,60){\vector(-1,-1){12}}
\put(40,0){\vector(1,1){12}}

\multiput(20,0)(20,0){2}{\circle*{2}}
\multiput(20,60)(20,0){2}{\circle*{2}}
\multiput(0,20)(60,0){2}{\circle*{2}}
\multiput(0,40)(60,0){2}{\circle*{2}}

\put(30,-10){\makebox(0,0){(a)}}

\end{picture}
\hspace{.5in}
\setlength{\unitlength}{2pt}
\begin{picture}(30,17)(-5,-17)
\put(0,3){\line(1,0){20}}
\put(0,3){\vector(1,0){12}}
\put(0,3){\line(2,3){10}}
\put(10,18){\vector(-2,-3){6}}
\put(10,18){\line(2,-3){10}}
\put(20,3){\vector(-2,3){6}}
\put(0,3){\circle*{1.5}}
\put(20,3){\circle*{1.5}}
\put(10,18){\circle*{1.5}}
\put(2,13){\makebox(0,0){$2$}}
\put(18,13){\makebox(0,0){$2$}}
\put(10,-1){\makebox(0,0){$1$}}
\put(10,-20){\makebox(0,0){(b)}}
\end{picture}
\hspace{.5in}
\setlength{\unitlength}{1.5pt}
\begin{picture}(25,27)(-2,-27)
\multiput(0,0)(0,20){2}{\line(1,0){20}}
\put(20,20){\vector(-1,0){12}}
\put(0,0){\vector(1,0){12}}
\multiput(0,0)(20,0){2}{\line(0,1){20}}
\put(0,20){\vector(0,-1){12}}
\put(20,0){\vector(0,1){12}}
\multiput(0,0)(20,0){2}{\circle*{2}}
\multiput(0,20)(20,0){2}{\circle*{2}}
\multiput(10,4)(0,20){2}{\makebox(0,0){$2$}}
\put(10,-30){\makebox(0,0){(c)}}
\end{picture}
\caption{$2$-finite cycles. 
}
\label{fig:3-4-cycles}
\end{figure}
\vspace{-.1in}

\begin{proof}[Proof of Proposition~\ref{pr:diagrams-fin-CK}]
We proceed by induction on $n$, the number of vertices in 
$\Gamma$. 
If $n \leq 3$, then $\Gamma$ is either a tree or a cycle, and the 
theorem follows by 
Proposition~\ref{pr:trees-Dynkin-vs-infty} and 
Exercise~\ref{exercise:cycles-Dynkin-vs-infty}. 
So let us assume that the statement is already known for some 
$n \geq 3$; we need to show that it holds for a diagram $\Gamma$ on 
$n+1$ vertices. 
Pick a vertex $v\in\Gamma$ such that the subdiagram 
$\Gamma' = \Gamma - \{v\}$ is connected. 
Since $\Gamma'$ is $2$-finite, it is 
mutation equivalent to some Dynkin diagram~$X_n\,$. 
Furthermore, we may assume that $\Gamma'$ is (isomorphic to) 
our favorite representative of the mutation equivalence class 
of~$X_n\,$. For each $X_n\,$, we will choose a representative that is 
most convenient for the purposes of this proof,
and use the classifications of $2$-finite tree and cycle diagrams
obtained above to achieve the desired goal. 
 
\noindent
\textbf{Case 1:} 
\emph{$\Gamma'$ is an orientation of a  Dynkin diagram 
with no branching point, i.e., is 
of one of the types $A_n$, $B_n$, $C_n$, $F_4$, or $G_2$.} 
Let us orient the edges of $\Gamma'$ so that they all point in the 
same direction. 
If $v$ is adjacent to exactly one vertex of $\Gamma'$, then 
$\Gamma$ is a tree, and we are done by 
Proposition~\ref{pr:trees-Dynkin-vs-infty}. 
If $v$ is adjacent to more than two vertices of $\Gamma'$, then 
$\Gamma$ has a cycle subdiagram whose edges are not cyclically 
oriented, contradicting Exercise~\ref{exercise:cycles-Dynkin-vs-infty}. 
Thus we may assume that $v$ is adjacent to precisely two vertices 
$v_1$ and $v_2$ of $\Gamma'$, see 
\cref{fig:B-triangle-with-ears}. 
Then $\Gamma$ has precisely one cycle~$\mathcal{C}$, which 
must be of one of the types (a)--(c) 
shown in Figure~\ref{fig:3-4-cycles}. 
 
Suppose that $\mathcal{C}$ is an oriented cycle with unit edge
weights.  
If $\Gamma$ has an edge of weight $\geq 2$, then it contains a 
subdiagram of type $B_m^{(1)}$ or $G_2^{(1)}$, unless $\mathcal{C}$ is 
a $3$-cycle, in which case $\mu_v(\Gamma)$ is a tree, and we are done by 
Proposition~\ref{pr:trees-Dynkin-vs-infty}. 
If all edges in $\Gamma$ are of weight~$1$, then it is one of 
the diagrams $S_{p,q,r}^s$ in Exercise~\ref{exercise:crown} (with $q=0$). 
Hence $\Gamma$ is mutation equivalent to a tree, and we are done.
 
Suppose that $\mathcal{C}$ is as in Figure~\ref{fig:3-4-cycles}(b). 
If one of the edges $(v,v_1)$ and $(v,v_2)$ has weight~$1$, 
then $\mu_v$ removes the edge $(v_1, v_2)$, resulting in a tree, and 
we are done again. 
So assume that both $(v,v_1)$ and $(v,v_2)$ have weight~$2$. 
If at least one edge outside $\mathcal{C}$ 
has weight $\geq 2$, then $\Gamma\supset C_m^{(1)}$ or 
$\Gamma\supset G_2^{(1)}\,$. 
It remains to consider the case shown in Figure~\ref{fig:B-triangle-with-ears}. 
A direct check shows that 
$\mu_l \circ \cdots \circ \mu_2 \circ \mu_1 \circ \mu_{v_2} 
\circ \mu_v(\Gamma)=B_{n+1}\,$, and we are done. 
 
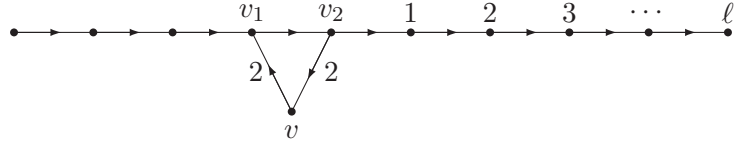
\begin{figure}[ht]
\setlength{\unitlength}{1.5pt} 
\begin{picture}(180,30)(0,-20) 
\put(0,0){\line(1,0){180}} 
\multiput(0,0)(20,0){10}{\circle*{2}} 
\put(70,-20){\circle*{2}} 
\put(70,-20){\line(-1,2){10}} 
\put(70,-20){\line(1,2){10}} 
 
\put(80,0){\vector(-1,-2){6}} 
\put(70,-20){\vector(-1,2){6}} 
 
\multiput(0,0)(20,0){9}{\vector(1,0){12}} 
 
\put(70,-25){\makebox(0,0){$v$}} 
\put(60,5){\makebox(0,0){$v_1$}} 
\put(61,-10){\makebox(0,0){$2$}} 
\put(80,-10){\makebox(0,0){$2$}} 
\put(80,5){\makebox(0,0){$v_2$}} 
\put(100,5){\makebox(0,0){$1$}} 
\put(120,5){\makebox(0,0){$2$}} 
\put(140,5){\makebox(0,0){$3$}} 
\put(160,5){\makebox(0,0){$\cdots$}} 
\put(180,5){\makebox(0,0){$\ell$}} 
\end{picture} 
\caption{Second subcase in Case~1.
} 
\label{fig:B-triangle-with-ears} 
\end{figure} 
 
Suppose that $\mathcal{C}$ is as in 
Figure~\ref{fig:3-4-cycles}(c). 
It suffices to show that any diagram $\mathcal{C}'$ 
ob\-tained from $\mathcal{C}$ 
by adjoining a single vertex adjacent to one of its vertices 
is $2$-infinite. 
Indeed, if this extra edge has weight $1$ (resp., $2$, $3$), then 
$\mathcal{C}'$ has a $2$-infinite subdiagram of type 
$B_3^{(1)}$ (resp., $C_2^{(1)}$, $G_2^{(1)}$). 
 
\noindent
\textbf{Case 2:} 
\emph{$\Gamma'\sim D_n$ ($n\geq 4$).} 
We may assume that 
$\Gamma'$ is an oriented $n$-cycle with unit edge weights. 

If $v$ is adjacent to two 
non-adjacent vertices of~$\Gamma'$ (and possibly others), 
then $\Gamma$ contains an improperly oriented cycle, contradicting 
Exercise~\ref{exercise:cycles-Dynkin-vs-infty}. 
 
Suppose $v$ is adjacent to a single 
vertex $v_1\in\Gamma'$. 
If the edge $(v,v_1)$ has weight $\geq 2$, then 
$\Gamma$ has a subdiagram $B_3^{(1)}$ or 
$G_2^{(1)}$. 
If $(v,v_1)$ has weight~$1$, then 
by Exercise~\ref{exercise:crown}, $\Gamma$ is mutation equivalent to a 
tree, and we are done. 
 
Suppose that $v$ is adjacent to exactly two 
vertices $v_1$ and $v_2$  which are adjacent to each
other. 
Then the triangle $(v,v_1,v_2)$ is either an oriented $3$-cycle with 
unit edge weights or the diagram in Figure~\ref{fig:3-4-cycles}(b). 
In the former case, $\mu_v(\Gamma)$ is an oriented $(n+1)$-cycle, so 
$\Gamma\sim D_{n+1}$. 
In the latter case, $\mu_v(\Gamma)$ contains an improperly oriented (hence 
$2$-infinite) cycle. 
 
\noindent
\textbf{Case 3:} 
\emph{$\Gamma'\sim E_n=T_{1,2,n-4}\,$, for $n\!\in\!\{6,7,8\}$.} 
By Exercise~\ref{exercise:crown}, we may assume that
$\Gamma'=S_{1,2,1}^{n-4}\,$, 
i.e., $\Gamma'$~consists of an oriented $(n-1)$-cycle 
$\mathcal{C}$ with unit edge weights, and an extra edge of weight $1$ 
connecting a vertex in 
$\mathcal{C}$ to a vertex 
$v_1\notin\mathcal{C}$. 

There are several subcases to examine, depending on how $v$
connects~to~$\mathcal{C}$. 
It is routine (if tedious) to check that in
each of these subcases, $\Gamma$~must be equivalent to an orientation
of a Dynkin diagram (e.g.\ because it is equivalent to a tree, or to
one of the diagrams treated in Cases 1 and~2 above), 
or else $\Gamma$ is not $2$-finite. 
Details can be found in~\cite{ca2}. 

This concludes the proof of Proposition~\ref{pr:diagrams-fin-CK}. 
As a consequence, we obtain Proposition~\ref{pr:2-fin-type-class}, 
Theorem~\ref{th:finite-type-bound-by-3}, and 
Theorem~\ref{th:finite-type-classification}. 
\end{proof}

\begin{remark}
\label{finite-type-hereditary}
The property of being $2$-finite is clearly 
hereditary.  Therefore  the other two equivalent properties
of exchange matrices appearing in  
\cref{th:finite-type-bound-by-3} are hereditary as well.
\end{remark}

\section{Quasi-Cartan companions}
\label{sec:quasi-cartan}

An unpleasant feature of the finite type classification
(see Theorem~\ref{th:finite-type-bound-by-3}) 
is that it does not provide an effective way to verify whether a
given exchange matrix~$B$ defines a seed pattern of finite type:
both condition~(3)
\hbox{($2$-finiteness)} and condition~(1) 
(being mutation equivalent to
a skew-sym\-metrizable version of a Cartan matrix)
impose a restriction on \emph{all} matrices in the mutation class
of~$B$.
An alternative criterion formulated directly in terms of the
matrix~$B$ (rather than its mutation class) 
was given in~\cite{bargezel}.
We reproduce this result below while omitting the technical part
of the proof. 

\begin{remark}
A different finite type recognition criterion was given in~\cite{seven},
by explicitly listing all minimal obstructions to finite type. 
More precisely, \cite{seven} provides a list of all (up to isomorphism)
\emph{minimal $2$-infinite diagrams}, i.e., all diagrams which are not
$2$-finite but whose proper subdiagrams are all $2$-finite.
Then $B$ is of finite type if and only if $\Gamma(B)$ does not contain
a subdiagram on this list.
Unfortunately, the list is rather long: it includes 10 infinite
series and a large number 
of exceptional diagrams of size~$\le 9$. 
\end{remark}

\begin{definition}
A \emph{quasi-Cartan matrix} is a symmetrizable (square)
matrix~$A=(a_{ij})$ 
with integer entries such that $a_{ii}=2$ for all~$i$. 
Note that in such a matrix, the entries $a_{ij}$ and $a_{ji}$ always 
have the same sign.
Unlike for generalized Cartan matrices, 
they are allowed to be positive. 

A quasi-Cartan matrix~$A$ is \emph{positive} if the 
symmetrized matrix is positive definite, or equivalently if the
principal minors of~$A$ are all positive. 

A quasi-Cartan matrix~$A$ is called a \emph{quasi-Cartan companion} of 
a skew-symmetrizable integer matrix~$B$
if $|a_{ij}| = |b_{ij}|$ for all~$i \neq j$.
Thus $B$ can have several quasi-Cartan companions
one of which is the Cartan counterpart of~$B$ given by
Definition~\ref{def:assoc-cartan}.
(To be precise, the number of quasi-Cartan companions of~$B$ is $2^e$ where
$e$ is the number of edges in the diagram of~$B$.) 
\end{definition}

A \emph{chordless cycle} in the diagram $\Gamma(B)$
is an induced subgraph isomorphic to a cycle (with arbitrary orientation). 



\begin{theorem}
\label{thm:bargezel-main}
For a skew-symmetrizable integer matrix~$B$, 
each of the conditions 
(1)--(3) 
in Theorem~\ref{th:finite-type-bound-by-3}
is equivalent to 
\begin{itemize}[leftmargin=.2in]
\item[\rm (4)]
every chordless cycle in $\Gamma(B)$ is cyclically oriented, and~$B$
has a positive quasi-Cartan companion.
\end{itemize}
\end{theorem}

The key ingredient in the proof of Theorem~\ref{thm:bargezel-main}
given in~\cite{bargezel}
is the following lemma, whose proof we omit.

\begin{lemma}
[{\cite[Lemma~4.1]{bargezel}}]
\label{lem:mutations-preserve-4}
Property {\rm (4)} in Theorem~\ref{thm:bargezel-main} is preserved
under mutations of skew-symmetrizable integer matrices.
\end{lemma}

\begin{proof}[Proof of Theorem~\ref{thm:bargezel-main} modulo
    Lemma~\ref{lem:mutations-preserve-4}] 
We deduce the implications $(1) \Rightarrow (4)\Rightarrow (3)$ 
from Lemma~\ref{lem:mutations-preserve-4}.
To prove that $(1) \Rightarrow (4)$, it is enough to observe that if 
$\Gamma(B)$ is a Dynkin diagram, then $B$ satisfies~(4). (Indeed,
the Cartan counterpart of $B$ is positive, and $\Gamma(B)$ 
has no cycles.)
To prove $(4) \Rightarrow (3)$, note that
any positive quasi-Cartan matrix~$A=(a_{ij})$ satisfies $|a_{ij}
a_{ji}| \leq 3$ for all $i \neq j$ because of the positivity of
the principal $2 \times 2$ minor of~$A$ occupying the rows 
and columns $i$ and~$j$.
\end{proof}

\begin{remark}
As explained above, one can use Lemma~\ref{lem:mutations-preserve-4}  to 
establish the implications $(1) \Rightarrow (4)\Rightarrow(3)$. 
In combination with the arguments given in Section~\ref{sec:2-finite},
this yields a self-contained combinatorial proof of the equivalence 
$(3)\Leftrightarrow(1)$. 
\end{remark}

A skew-symmetrizable integer matrix~$B$ can have many quasi-Cartan companions~$A$,
corresponding to different choices for the signs of its off-diagonal matrix
entries. 
Note that the positivity property of~$A$
is preserved by simultaneous sign changes in rows and columns.
It turns out that for the purposes of checking (4) for a given matrix~$B$,
there is a \emph{unique}, up to these transformations, sign pattern for~$A$ that needs to be
checked for positivity.
More precisely, we have the following 
result, cf.\ \cite[Propositions~1.4--1.5]{bargezel}.

\begin{proposition}
\label{pr:cyclic-B-to-A}
Let $B$ be a skew-symmetrizable integer matrix such that 
each chordless cycle in $\Gamma(B)$ is cyclically
oriented. 
Then~$B$ has a quasi-Cartan companion~$A=(a_{ij})$ 
such that the sign condition 
\begin{equation}
\label{eq:sign-condition-A}
\prod_{\{i,j\} \in Z} (-a_{ij})<0
\end{equation}
(product over all edges $\{i,j\}$ in~$Z$) 
is satisfied for every chordless cycle~$Z$. 
In fact, $A$ is unique up to simultaneous sign changes in rows and
columns. 
Moreover $B$ satisfies conditions 
(1)--(3) in \cref{th:finite-type-bound-by-3} 
if and only if $A$ is positive. 
\end{proposition}

\begin{remark}
Several characterizations of positive quasi-Cartan matrices have been
given in \cite[Proposition~2.9]{bargezel}.
In particular, these matrices are, up to certain equivalence, 
also classified by Cartan-Killing types.
More precisely, any positive quasi-Cartan matrix corresponding to a
root system $\Phi$ has the entries $a_{ij} = \langle \beta_i^\vee, \beta_j \rangle$,
where $\{\beta_1, \dots, \beta_n\} \subset \Phi$ is a $\ZZ$-basis
of the root lattice generated by $\Phi$, and
$\beta^\vee$ is the coroot dual to a root $\beta \in \Phi$.
\end{remark}

We conclude this section by an example illustrating the use of
Proposition~\ref{pr:cyclic-B-to-A} 
for checking whether a particular skew-symmetric matrix 
is an exchange matrix of a seed pattern of finite type.

\pagebreak[3]

\begin{example}
\vbox{Let $Q(n)$ be the following quiver with vertices $1,2,\dots,n$:\nopagebreak
\begin{center}
\setlength{\unitlength}{1.5pt}
\begin{picture}(170,30)(0,-5)
\put(0,20){\line(1,0){130}}
\put(20,0){\line(1,0){130}}
\multiput(20,0)(40,0){4}{\circle*{2}}
\multiput(0,20)(40,0){4}{\circle*{2}}
\multiput(20,0)(40,0){4}{\line(-1,1){20}}
\multiput(40,20)(40,0){3}{\line(-1,-1){20}}

\multiput(20,0)(40,0){4}{\vector(-1,1){10}}
\multiput(40,20)(40,0){3}{\vector(-1,-1){10}}
\multiput(0,20)(40,0){3}{\vector(1,0){20}}
\multiput(20,0)(40,0){3}{\vector(1,0){20}}

\put(0,25){\makebox(0,0){$1$}}
\put(40,25){\makebox(0,0){$3$}}
\put(80,25){\makebox(0,0){$5$}}
\put(120,25){\makebox(0,0){$7$}}
\put(20,-5){\makebox(0,0){$2$}}
\put(60,-5){\makebox(0,0){$4$}}
\put(100,-5){\makebox(0,0){$6$}}
\put(140,-5){\makebox(0,0){$8$}}

\put(160,0){\makebox(0,0){$\cdots$}}
\put(140,20){\makebox(0,0){$\cdots$}}

\end{picture}
\end{center}
}
\noindent
The quiver $Q(n)$ is the diagram of its $n \times n$ exchange matrix 
\[
B(n)=B(Q(n))=\begin{bmatrix}
0 & -1 & 1 & 0 & \cdots & 0 & 0\\
1 &  0 &-1 & 1 & \cdots & 0 & 0\\
-1&  1 & 0 &-1 & \cdots & 0 & 0\\
0 & -1 & 1 & 0 & \cdots & 0 & 0\\[-.05in]
\vdots & \vdots & \vdots & \vdots & \ddots & \vdots & \vdots\\
0 &  0 & 0 & 0 & \cdots & 0 & -1\\
0 &  0 & 0 & 0 & \cdots & 1 & 0
\end{bmatrix}
\]
(that is, $\Gamma(B(n))=Q(n)$).
This quiver has $n-2$ chordless cycles, the $3$-cycles with vertices
$\{i,i+1,i+2\}$, for $i=1,\dots,n-2$. 
All of them are cyclically oriented.
Now let~$A(n)$ be the quasi-Cartan companion of~$B(n)$ such that $a_{ij} =
b_{ij}$ for $i<j$.
One immediately checks that $A(n)$ satisfies the sign
condition~\eqref{eq:sign-condition-A}. 
Let $\delta_n = \det(A(n))$. 
By Sylvester's criterion, $A(n)$ is positive if and only if 
all the numbers $\delta_1,\dots,\delta_n$ are positive. 

It is not hard to compute the generating function of the
sequence~$(\delta_n)$, with the convention~$\delta_0 = 1$:
\begin{equation}
\label{eq:detAn}
\sum_{n \geq 0} \delta_n\, x^n = \frac{(1+x)(1+x+x^2)(1+x^2)(1+x^3)}
{1-x^{12}} \,.
\end{equation}
We see that $\delta_{n+12} \!=\! \delta_n$ for $n \geq 0$.
Since the numerator in \eqref{eq:detAn} is a polynomial of degree~$8$,
we conclude that $\delta_9 \!=\! \delta_{10} \!=\! \delta_{11} \!=\! 0$.
The fact that $\delta_9\!=\!0$ implies that $A(n)$ is not positive 
(hence $B(n)$ is not $2$-finite) for $n\ge 9$.

The values of~$\delta_n$ for $1 \leq n \leq 8$ are given in
Figure~\ref{fig:det-An}; cf.\ Exercise~\ref {exercise:D4-E8}. 
As all of them are positive, we conclude that 
$A(n)$ is positive (and so $B(n)$ is $2$-finite) 
if and only if $n \leq 8$. 
The corresponding Cartan-Killing types are 
shown in Figure~\ref{fig:det-An};
we leave the verification to the reader.
\end{example}

\begin{figure}[ht]
\begin{center}
\begin{tabular}{|c|c|c|c|c|c|c|c|c|}
\hline
&&&&&&&&\\[-.1in]
$n$ & 1 & 2 & 3 & 4 & 5 & 6 & 7  & 8\\
\hline
&&&&&&&&\\[-.1in]
$\delta_n=\det(A(n))$ &  2 & 3 & 4 & 4 & 4 & 3 & 2  & 1\\
\hline
&&&&&&&&\\[-.1in]
Cartan-Killing type &  $A_1$ & $A_2$ & $A_3$ & $D_4$ & $D_5$ &
$E_6$ & $E_7$  & $E_8$ \\[.05in]
\hline
\end{tabular}
\end{center}
\caption{Determinants and Cartan-Killing types of the matrices~$A(n)$.}
\label{fig:det-An}
\end{figure}

 

\backmatter


\bibliographystyle{acm}
\bibliography{bibliography}
\label{sec:biblio}


\end{document}